\documentclass[11pt]{article}
\usepackage[utf8]{inputenc}
\usepackage[T1]{fontenc}
\usepackage{lmodern}
\usepackage[english]{babel}

\usepackage[a4paper,vmargin={3.5cm,3.5cm},hmargin={2.5cm,2.5cm}]{geometry}
\usepackage[font={small,sf}, labelfont={sf,bf}, margin=1cm]{caption}

\usepackage[nocompress]{cite} 

\usepackage[pdftex,colorlinks=true]{hyperref}
\usepackage[expansion]{microtype}
\usepackage[pdftex]{color,graphicx}
\usepackage{amsmath,amsfonts,amssymb,amsthm,mathrsfs,ulem}
\usepackage{stackrel,framed}
\usepackage{subfig}
\usepackage{stmaryrd} 

\numberwithin{equation}{section}

\theoremstyle{plain}

\newtheorem{theo}{Theorem}[section]

\newtheorem{coro}{Corollary}[section]
\newtheorem{prop}[coro]{Proposition}
\newtheorem{lemm}[coro]{Lemma}
\newtheorem{rema}[coro]{Remark}
\newtheorem{defi}[coro]{Definition}

\usepackage[title]{appendix}

\graphicspath{{images/}}

\newcommand{\alphaa}{a}

\newcommand{\ps}{\oplus}
\newcommand{\ns}{\ominus}

\renewcommand{\leq}{\leqslant}
\renewcommand{\geq}{\geqslant}

\hypersetup{
    pdfauthor   = {Albenque, Ménard},%
    pdftitle    = {Spin clusters in random triangulations coupled with an Ising Model}}

\begin{document}
\normalem
\title{\bf Geometric properties of spin clusters in random triangulations coupled with an Ising Model.}
\author{{Marie Albenque}\footnote{{LIX, UMR CNRS 7161, \'Ecole Polytechnique, 91120 Palaiseau, France}. \href{mailto:albenque@lix.polytechnique.fr}{albenque@lix.polytechnique.fr}} \, and {Laurent Ménard}\footnote{{New York University Shanghai, 200122 Shanghai, China,} and {Modal'X, UMR CNRS 9023, UPL, Univ. Paris Nanterre, F92000 Nanterre, France}. \href{mailto:laurent.menard@normalesup.org}{laurent.menard@normalesup.org}}}
\date{January 27, 2022}

\maketitle


\begin{abstract}
We investigate the geometry of a typical spin cluster in random triangulations sampled with a probability proportional to the energy of an Ising configuration on their vertices, both in the finite and infinite volume settings. This model is known to undergo a combinatorial phase transition at an explicit critical temperature, for which its partition function has a different asymptotic behavior than uniform maps. The purpose of this work is to give geometric evidence of this phase transition.

In the infinite volume setting, called the Infinite Ising Planar Triangulation, we exhibit a phase transition for the existence of an infinite spin cluster: for critical and supercritical temperatures, the root spin cluster is finite almost surely, while it is infinite with positive probability for subcritical temperatures. Remarkably, we are able to obtain an explicit parametric expression for this probability, which allows to prove that the percolation critical exponent is $\beta=1/4$.

We also derive critical exponents for the tail distribution of the perimeter and of the volume of the root spin cluster, both in the finite and infinite volume settings. Finally, we establish the scaling limit of the interface of the root spin cluster seen as a looptree. In particular in the whole supercritical temperature regime, we prove that the critical exponents and the looptree limit are the same as for critical Bernoulli site percolation. 

Our proofs mix combinatorial and probabilistic arguments. The starting point is the gasket decomposition, which makes full use of the spatial Markov property of our model. This decomposition enables us to characterize the root spin cluster as a Boltzmann planar map in the finite volume setting. We then combine precise combinatorial results obtained through analytic combinatorics and universal features of Boltzmann maps to establish our results.

\bigskip

\noindent{\bf MSC 2010 Classification:}
05A15, 
05A16, 
05C12, 
05C30, 
60C05, 
60D05, 
60K35, 
82B44  
\end{abstract}


\newpage
\tableofcontents
\newpage

\section{Introduction}
In recent years, a lot of attention has been devoted to the mathematical study of random planar maps (graphs embedded into surfaces). One of the original motivation, coming from theoretical physics and quantum gravity, is to provide generic models for 2-dimensional random geometries.

The combinatorial study of maps originated in the work of Tutte~\cite{Tuttetri,Tuttecensus}, who obtained closed enumerative formulas for many classes of maps. The combinatorial properties of these models are now fairly well understood. Indeed, semi-automatic methods to enumerate planar maps have been developed, see for instance the monograph by Flajolet and Sedgewick on analytic combinatorics \cite{FS} and the work by Bousquet-Mélou and Jehanne~\cite{BMJ}. In addition, many bijections between planar maps and decorated trees are available to explain the simplicity of their enumeration, see for example Schaeffer's thesis~\cite{Sch98} and the work by Bouttier, Di Francesco and Guitter \cite{BDFG}. We refer to Schaeffer's survey~\cite{autopromo} for a review of the earlier combinatorial literature on random maps.

This has led to spectacular results in the probabilistic study of uniform models of random maps. Important milestones of this field are the introduction of \emph{infinite volume limits} for the local topology by Angel and Schramm \cite{AngelSchramm}, as well as Chassaing and Durhuus~\cite{ChassaingDurhuus} and Krikun~\cite{Kr,KrikunQuad}, and the proof that the \emph{scaling limit of uniform quadrangulations} is the Brownian sphere for the Gromov-Hausdorff topology by Le Gall~\cite{LGbm} and Miermont~\cite{Miebm}. The Brownian sphere has then been proved to share a deep connection with Liouville Quantum Gravity by Miller and Sheffield in a series of articles \cite{MSa,MSb,MSc}, as well as by Holden and Sun \cite{HoldenSun}. 
We refer to the recent surveys by Le Gall~\cite{LGsurvey}, Miermont \cite{MieSF} and Miller \cite{MillerSurvey} for nice entry points to this field.

\bigskip

The common behavior of large scale properties of uniform models of random maps and the ubiquity of the Brownian sphere in this setting are the archetype of a universality class. In the physics literature, this class is called \emph{pure 2d quantum gravity}. From a physics perspective, coupling gravity and matter should lead to other universality classes of \emph{2d quantum gravity with matter}. A natural way to proceed from a mathematical point of view is to consider planar maps with a statistical physics model.

A first important set of such models consists of maps decorated by a critical statistical model for which there exists a \emph{mating-of-trees} bijection. These bijections are generalizations of Mullin's bijection~\cite{Mullin}, which encodes a spanning tree decorated map by a walk in $\mathbb Z^2$. The walks coding the maps together with the statistical model can then be interpreted as a pair of shuffled trees, hence the name \emph{mating-of-trees}. This category includes maps decorated with a critical FK percolation~\cite{She,C}, bipolar-oriented maps~\cite{Bipolar}, and Schnyder-wood decorated maps~\cite{Schnyder}. The bijection allows to study the convergence of these models in the \emph{Peano} sense, or for the local topology. We refer to the survey by Gwynne, Holden and Sun for an overview of this approach \cite{GHSSurvey}.

Models for which no \emph{mating-of-trees} bijection is available have also been considered recently, among which planar maps decorated by the $\mathcal O(n)$-model and triangulations decorated by an Ising model.
The $\mathcal O(n)$-model is a critical model of loops and has been studied by Borot, Bouttier and Guitter~\cite{BBGa,BBGb,BBGc}, as well as Borot, Bouttier and Duplantier~\cite{BBD}. The Ising model was initially introduced by Lenz and studied by Ising in the 1920s~\cite{Ising} to model magnetism, and was proved to exhibit a phase transition in 2-dimension by Onsager~\cite{Onsager}. Probabilistic aspects of planar maps decorated by an Ising model were studied recently by the authors and Schaeffer~\cite{IsingAMS}, and by Chen and Turunen~\cite{ChenTurunen,ChenTurunen2}, as well as Turunen~\cite{Turunen}. In particular, it was proved in~\cite{IsingAMS} that large random planar triangulations coupled with an Ising model converge in law for the local topology for any value of the temperature parameter. The limiting object is called the Ising Infinite Planar Triangulation (IIPT).

\bigskip

\begin{figure}[t!]
\centering
 \subfloat[Triangulation with spins,]{\quad\quad
      \includegraphics[width=0.32\textwidth,page=1]{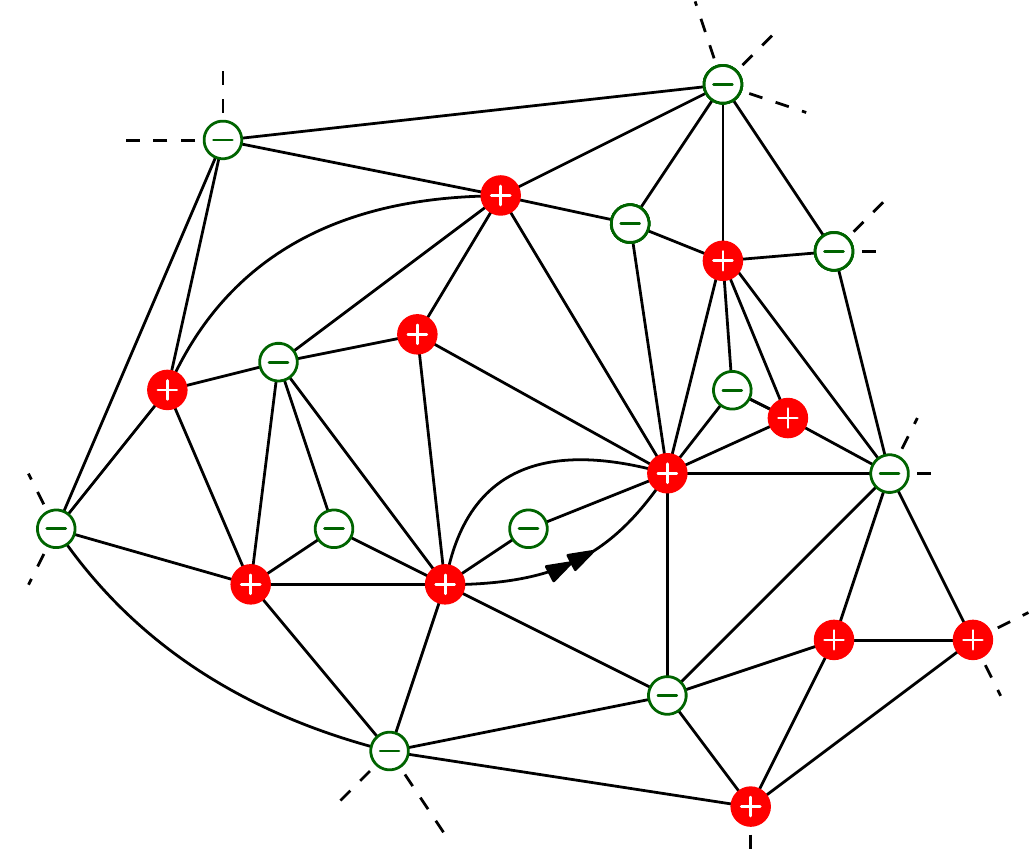}
      \label{subfig:hulla}\quad\quad} 
    \subfloat[its root spin cluster,]{\quad\quad
      \includegraphics[width=0.32\textwidth,page=2]{hull.pdf}
      \label{subfig:hullb}\quad\quad}\\
     \subfloat[and its hull cluster.]{\quad
      \includegraphics[width=0.32\textwidth,page=3]{hull.pdf}
      \label{subfig:hullc}\quad}
    \caption{Example of the root spin cluster $\mathfrak C$ and its hull $\mathfrak H$ for an Ising-weighted triangulation with monochromatic root edge. The boundary $\partial \mathfrak C$ of the root spin cluster is in bold in the last picture and has perimeter $8$.}
    \label{fig:IntroCluster}
\end{figure}

In this paper, we focus on triangulations decorated by an Ising model. Our aim is to study the geometry of a typical spin cluster (see Figure~\ref{fig:IntroCluster} for an illustration), both in the finite volume setting and in the IIPT. An important aspect of our work is that the model is studied not only at criticality, but for any temperature. This allows to capture an interesting phase transition for the geometric properties of the clusters. As we will see, this phase transition coincides with a combinatorial phase transition where the partition function of our model has been known to behave differently than uniform models of random maps since the work of Boulatov and Kazakov~\cite{BoulatovKazakov}.

We quantify several aspects of this phase transition. First, in the IIPT, we show that the cluster of the origin is infinite with positive probability if and only if the temperature is smaller than the critical temperature. We also calculate explicitly this probability and the associated critical exponent. Secondly, we calculate the critical exponents for the tail distribution of the volume and perimeter of the cluster of the origin. Finally, we derive the scaling limit in the Gromov--Hausdorff topology of the interface of this cluster in the setting of looptrees. All these quantities also undergo a phase transition at the critical temperature. Strikingly, we prove that, for all these quantities, the whole supercritical temperature regime behaves like critical Bernoulli percolation.

There are bijections between decorated trees and maps with an Ising model, see Bousquet--Mélou and Schaeffer~\cite{BMS} and Bouttier, Di Francesco and Guitter~\cite{BDFG}. However these bijections are not of the mating-of-tree type and it is not clear if they can give some insight on the geometric properties of our model.

Apart from bijections, the main classical tools to study models of random maps with decorations are the so-called \emph{peeling process} and \emph{gasket decomposition}. Both approaches make use of the spatial Markov properties of planar maps. Roughly speaking, the peeling process is a Markovian exploration that can be tailored to follow interfaces between clusters. It was first used by Angel to study volume growth and site percolation on the Uniform Infinite Planar Triangulation \cite{AngelPerco}, and later by many others. See the lectures notes by Curien~\cite{CurienSF} for a presentation of this procedure.

The peeling process is however not well suited to study the geometry of clusters.
In this work, we rely on the gasket decomposition, which consists on decomposing a triangulation with an Ising model into a spin cluster, and blocks filling the faces of this cluster. This approach has been used successfully by Borot, Bouttier, Duplantier and Guitter to study the $\mathcal O(n)$ model on random maps~\cite{BBGa,BBGb,BBGc,BBD}, and by Bernardi, Curien and Miermont to study Bernoulli bond and site percolation on random triangulations~\cite{BeCuMie}.

\bigskip

The rest of this introduction consists in the definition of our model, the presentation of our main results, the organization of the article and our strategy, as well as how our results relate to important conjectures. 

\subsection{Models studied in this work.}

To state our main results, let us introduce some notation and terminology. Precise definitions will be given in sections~\ref{sec:prelim} and~\ref{sec:combi}. If $\mathfrak t$ is a finite rooted triangulation of the sphere, a spin configuration on the vertices of $\mathfrak t$ is a mapping $\sigma : V(\mathfrak t) \to \{\ns,\ps \}$ and for a triangulation with spins $(\mathfrak t,\sigma)$, we denote by $m(\mathfrak t,\sigma)$ its number of monochromatic edges. We also denote by $|\mathfrak t|$ the size of $\mathfrak t$, that is its number of edges. Denoting by $\mathcal T$ the set of finite triangulations with a spin configuration, we define the partition function of the Ising model on triangulations of the sphere by
\begin{equation} \label{eq:sphereGS}
\mathcal Z (\nu,t) = \sum_{(\mathfrak t,\sigma) \in \mathcal T} \nu^{m(\mathfrak t,\sigma)} t^{|\mathfrak t|}.
\end{equation}
Writing $\nu = \exp (2 / T)$, we have
\[
\nu^{m(\mathfrak t,\sigma)} = \exp \left(\frac{1}{T} \sum_{\{v,v'\}\in E(\mathfrak t)} \sigma(v) \, \sigma(v') \right) \, \exp \left(  |\mathfrak t| / T\right),
\]
and the partition function $\mathcal Z (\nu,t)$ can be seen as the usual partition function of the Ising model with temperature $T$ and no external magnetic field. In particular, when $\nu >1$ the temperature is positive and the model is ferromagnetic. When $\nu <1$ the temperature is negative and the model is antiferromagnetic. When $\nu = 1$, the temperature is infinite and the model corresponds to Bernoulli site percolation with parameter $1/2$, which is the critical parameter for site percolation on triangulations (see \emph{e.g.} the original work of Angel \cite{AngelPerco}).

This partition function has been studied by Boulatov--Kazakov \cite{BoulatovKazakov}, Bernardi--Bousquet-Mélou \cite{BernardiBousquet} and the authors of this article together with Schaeffer \cite{IsingAMS}. For every $\nu >0$, let us denote by $t_\nu$ the radius of convergence in $t$ of $\mathcal Z(\nu,t)$. The value of $t_\nu \in (0,\infty)$ is known explicitly as well as the value $\mathcal Z(\nu,t_\nu)<\infty$, see for example \cite{BernardiBousquet} or \cite{IsingAMS}. It is shown in the works \cite{IsingAMS,BernardiBousquet,BoulatovKazakov,BDFG,BMS} that $\mathcal Z$ undergoes a combinatorial phase transition at
\begin{equation}\label{eq:nuc}
\nu_c := 1 + \frac{\sqrt 7}{7}.
\end{equation}
When $\nu \neq \nu_c$ the coefficients in $t$ of the series $\mathcal Z(\nu,t)$ exhibits the same universal asymptotic behavior as undecorated models of planar maps with exponent $5/2$. On the other hand, when $\nu = \nu_c$ the model falls in a different universality class and the coefficients exhibit an asymptotic behavior with exponent $7/3$.

\bigskip

To simplify our statements,
we restrict our attention to triangulations with spins such that both end vertices of its root edge have spin $\ps$. We denote by $\mathcal T^\ps$ the set of all such triangulations and by $\mathcal Z^{\ps} (\nu , t)$ their partition function, that is:
\begin{equation} \label{eq:GSplus}
\mathcal Z^\ps (\nu,t) = \sum_{(\mathfrak t,\sigma) \in \mathcal T^\ps} \nu^{m(\mathfrak t,\sigma)} t^{|\mathfrak t|}.
\end{equation}
The partition function $\mathcal Z^\ps$ has the same radius of convergence $t_\nu$ as $\mathcal Z$ and undergoes the same combinatorial phase transition.

We will consider three different models of random triangulations coupled with an Ising model, which correspond informally to finite random volume, finite fixed volume, and infinite volume limit:
\begin{itemize}
\item The finite random volume distribution is the probability measure on $\mathcal T^\ps$ defined by
\begin{equation} \label{eq:Pnudef}
\forall (\mathfrak t,\sigma) \in \mathcal T^\ps , \quad \mathbb P^\nu \left( \{(\mathfrak t,\sigma) \}\right) = 
\frac{t_\nu ^{|\mathfrak t|} \nu^{m(\mathfrak t,\sigma)}}{\mathcal Z^{\ps}(\nu,t_\nu)}.
\end{equation}
A random triangulation with spins with law $\mathbb P^\nu$ is called an Ising random triangulation with parameter $\nu$ and will be denoted by $\mathbf T^\nu$.
\item The finite fixed volume distribution $\mathbb P^\nu_n$ is the probability $\mathbb P^\nu$ conditioned on the event where the triangulation has size $3n$:
\begin{equation} \label{eq:Pnundef}
\forall (\mathfrak t,\sigma) \in \mathcal T^\ps , \quad \mathbb P_n^\nu \left( \{(\mathfrak t,\sigma) \}\right) = \mathbb P^\nu \Big( \{(\mathfrak t,\sigma) \} \Big| |\mathfrak t| = 3n \Big) = 
\frac{\nu^{m(\mathfrak t,\sigma)}}{[t^{3n} ] \mathcal Z^{\ps}(\nu,t)} \mathbf{1}_{\{ |\mathfrak t| = 3n  \}}.
\end{equation}
A random triangulation with spins with law $\mathbb P^\nu_n$ will be denoted by $\mathbf T^\nu_n$.
\item The infinite volume distribution $\mathbb P^\nu_\infty $ is the weak limit of $\mathbb P^\nu_n$ when $n\to \infty$ for the local topology as defined in \cite{IsingAMS}.
A random triangulation with spins with law $\mathbb P^\nu_\infty$ is called an Infinite Ising Planar Triangulation with parameter $\nu$ (or $\nu$-IIPT) and will be denoted by $\mathbf T^\nu_\infty$. See Section \ref{sec:IIPTdef} for details.
\end{itemize}

\subsection{Main results}

We now delve into the presentation of our main results. 
For a rooted (possibly infinite) triangulation with spins $(\mathfrak t,\sigma)$, its root spin cluster (or root sign cluster) is the connected component of the root vertex in the submap of $\mathfrak{t}$ spanned by monochromatic edges, and is denoted by $\mathfrak{C}(\mathfrak t,\sigma)$. See Figure \ref{fig:IntroCluster} for an illustration. When the root edge of $\mathfrak t$ is monochromatic, the root spin cluster contains the root edge and is hence a rooted planar map. In this case, we denote by $\partial \mathfrak{C}(\mathfrak t,\sigma)$ the boundary of its root face -- which is the face lying on the right-hand side of the root --  and define the perimeter $|\partial \mathfrak{C}(\mathfrak t,\sigma)|$ of the root spin cluster as the length of $\partial \mathfrak{C}(\mathfrak t,\sigma)$,  see Figure \ref{fig:IntroCluster}.

\bigskip

This work focuses on the geometry of $\mathfrak C(\mathbf{T}^\nu)$, $\mathfrak C(\mathbf{T}_n^\nu)$ and $\mathfrak C(\mathbf{T}_\infty^\nu)$. Since the distribution of these random maps is invariant by re-rooting along a simple random walk, the root spin cluster has the same geometric properties as a typical spin cluster in any of these three models.

\paragraph{Cluster percolation probability and critical exponent.}

Our first main result establishes a phase transition at $\nu_c$ for percolation of the root spin cluster $\mathfrak C(\mathbf T_\infty ^\nu)$ of the $\nu$-IIPT, and gives the value of the associated critical exponent $\beta$:
\begin{theo} \label{th:expo}
The root spin cluster of the $\nu$-IIPT is almost surely finite when $\nu \leq \nu_c$ and infinite with positive probability when $\nu > \nu_c$.

In addition, we have an explicit parametric formula for the probability $\mathbb P_\infty^\nu \left( \left| V(\mathfrak C (\mathbf T^\nu_\infty)) \right| = \infty \right)$, see Figure~\ref{fig:ProbaPerco} and Equation~\eqref{eq:ProbaPerco}. From this formula, we can establish that the percolation critical exponent $\beta$ of the root spin cluster is $1/4$, that is
\[
\mathbb P_\infty^\nu \left( \left| V(\mathfrak C (\mathbf T^\nu_\infty)) \right| = \infty \right)
\underset{\nu \to \nu_c^+}{\sim} \kappa \, \left( \nu - \nu_c \right)^{1/4},
\]
with 
\[
\kappa = \frac{\sqrt{3 + 2 \sqrt{3}} \,  2^{3/4} \, 7^{3/8} \, (63 + 18 \sqrt{7} - 10 \sqrt 3 \, \sqrt 7 - 23 \sqrt 3)}{144}.
\]
\end{theo} 

\begin{figure}[t!]
\centering
\includegraphics[width=0.8\linewidth]{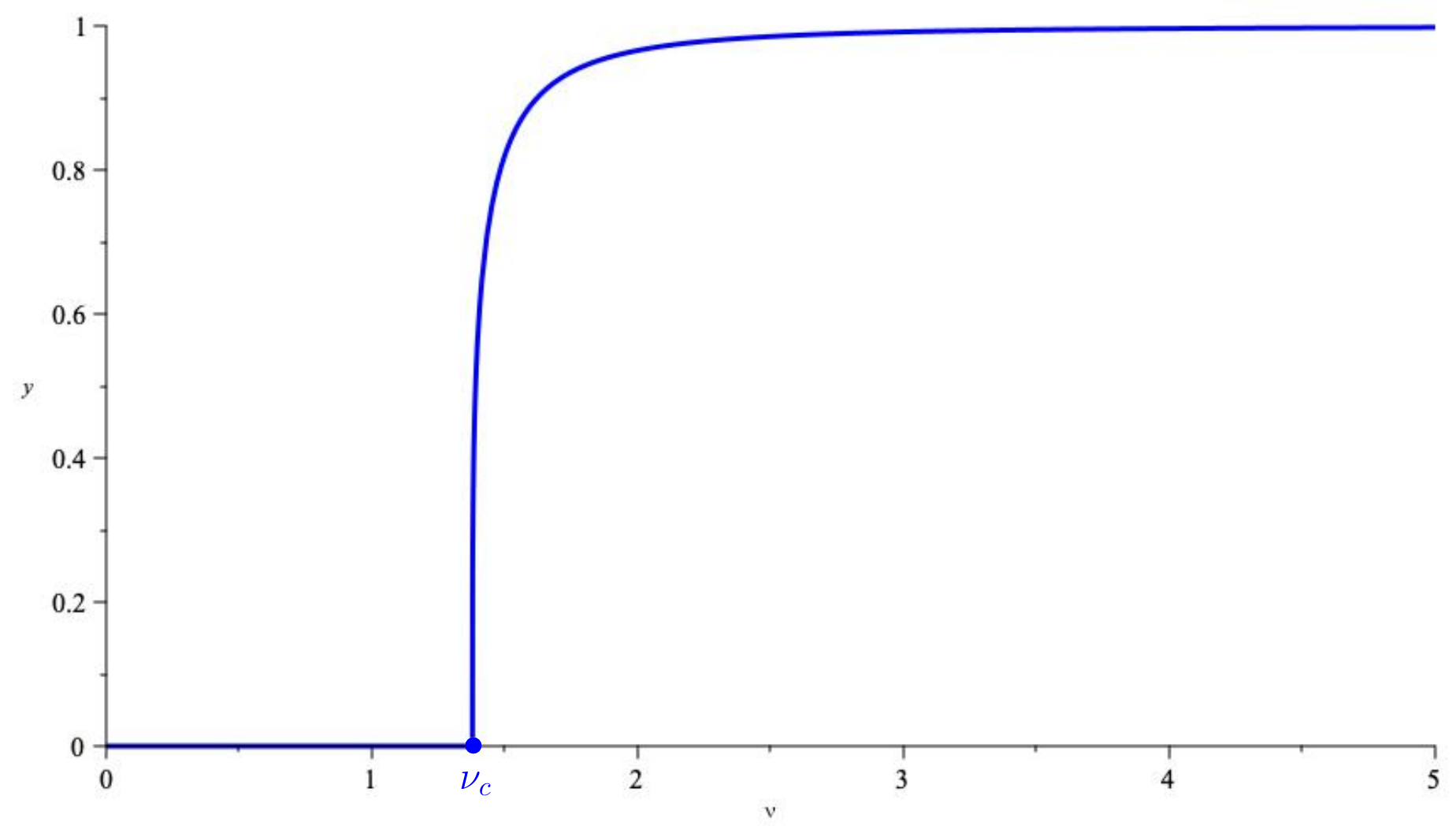}
\caption{\label{fig:ProbaPerco} Curve of $\mathbb P_\infty^\nu \left( \left| V(\mathfrak C (\mathbf T^\nu_\infty)) \right| = \infty \right) $ as a function of $\nu >0$.}
\end{figure}

This result is reminiscent of the percolation properties of the spin clusters in the Ising model on a 2-dimensional Euclidean lattice, for which a similar phase transition has been established using the Edwards-Sokal coupling with Fortuin-Kasteleyn percolation, see for example Duminil-Copin and Sminorv \cite{DCS} for a review on the subject and much more. In addition, the exponent $1/4$ of Theorem~\ref{th:expo} -- also called the spontaneous magnetization exponent in the physics literature -- has a counterpart equal to $1/8$ for the Ising model on the square lattice which was predicted by Onsager \cite{Onsager} and proved by Yang \cite{Yang}. It is not clear to us at the moment whether these two exponents can be related via the KPZ formula \cite{KPZ} and scaling relations \emph{\`a la} Kesten \cite{Kesten}.

\paragraph{Cluster volume and perimeter.}

Our next results deal with the tail distribution of the perimeter and of the volume of the root spin cluster in the three models, for which we establish sharp asymptotic estimates. 

\bigskip

We start by stating our results in the infinite volume case and further characterize the phase transition in the geometry of the root spin cluster. For $\nu\leq \nu_c$, i.e. when the root spin cluster is finite almost surely, we establish tail distributions for its volume and its perimeter. For $\nu>\nu_c$, we establish tail distribution for its perimeter, which stays almost surely finite even if the cluster is infinite with positive probability:
\begin{theo} \label{th:mainIIPT}
In the $\nu$-IIPT, the volume and the perimeter of the root spin cluster exhibit the following asymptotics: 
\begin{itemize}
\item High temperature case. For $0< \nu < \nu_c$, there exist positive constants $\mathrm{C}^{\mathrm{v}}_{\infty}(\nu)$ and $\mathrm{C}^{\mathrm{p}}_{\infty}(\nu)$ such that:
\[
\mathbb P_{\infty}^\nu \left( |V(\mathfrak C (\mathbf T^\nu_\infty))| > n \right)
\underset{n \to \infty}{\sim} \mathrm{C}^{\mathrm{v}}_{\infty}(\nu)\,  n^{-{1}/{7}} \quad \text{and} \quad \mathbb P_{\infty}^\nu \left( |\partial \mathfrak C (\mathbf T^\nu_\infty)| = p \right)
\underset{p \to \infty}{\sim} \mathrm{C}^{\mathrm{p}}_{\infty}(\nu)\, p^{-{4}/{3}}.
\]
\item Critical case. For $\nu = \nu_c = 1 + 1/\sqrt{7}$, there exist positive constants $\mathrm{C}^{\mathrm{v}}_{\infty}(\nu_c)$ and $\mathrm{C}^{\mathrm{p}}_{\infty}(\nu_c)$ such that:
\[
\mathbb P_{\infty}^\nu \left( |V(\mathfrak C (\mathbf T^\nu_\infty))| > n \right)
\underset{n \to \infty}{\sim} \mathrm{C}^{\mathrm{v}}_{\infty}(\nu_c)\, n^{-{1}/{11}} \quad \text{and} \quad \mathbb P_{\infty}^\nu \left( |\partial \mathfrak C (\mathbf T^\nu_\infty)| = p \right)
\underset{p \to \infty}{\sim} \mathrm{C}^{\mathrm{p}}_{\infty}(\nu_c)\, p^{-2}.
\]
\item Low temperature case. For $\nu > \nu_c$, the perimeter $|\partial \mathfrak C (\mathbf T^\nu_\infty)|$ is finite almost surely and has exponential tail.
\end{itemize}
\end{theo}

Let us make a few comments on this result. First, we see that for $\nu \leq \nu_c$, the root spin cluster volume has infinite expectation. In particular, the statement of Theorem~\ref{th:expo} on the finiteness of the root spin cluster for this range of values of $\nu$ is not a direct consequence of Theorem~\ref{th:mainIIPT}. We will see that it turns out that the proofs of Theorem~\ref{th:expo} and Theorem~\ref{th:mainIIPT} are mostly independent.

Secondly, the high temperature regime $\nu < \nu_c$ includes the antiferromagnetic regime $\nu <1$ and the infinite temperature regime $\nu = 1$, where the spins are i.i.d. with probability $1/2$ to be $\ps$ or $\ns$, corresponding to critical site percolation on the UIPT as proved by Angel~\cite{AngelPerco}. As a consequence, in the whole high temperature regime, the geometry of the root spin cluster is similar to the geometry of the root spin cluster for critical site percolation, at least in term of the critical exponent $1/7$ for the volume and $4/3$ for the perimeter tail distribution. This can be seen as a Quantum Gravity version of the long standing conjecture that in the high temperature regime, the Ising model on planar lattices is in the same universality class as critical Bernoulli percolation, see for example B\'alint, Camia and Meester~\cite{IsingBernoulli} and the references therein.

The perimeter exponent $4/3$ was previously established for critical site percolation on the UIPT by Curien and Kortchemski \cite{CuKo} with different methods. However, the volume exponent $1/7$ is new even for critical site percolation on the UIPT and answers a conjecture by Gorny--Maurel Segala--Singh \cite{GMSS}. 

Lastly, we also mention that counterparts of Theorem~\ref{th:expo} and of Theorem~\ref{th:mainIIPT} for critical and off-critical site percolation on the UIPT are established independently by the second author \cite{Mperco}.

\bigskip

Our second set of results deals with triangulations with finite volume. First, we study the root spin cluster in the triangulation $\mathbf T^\nu$ with finite random volume:

\begin{theo} \label{th:mainBoltzmann}
In $\mathbf T^\nu$, the volume and the perimeter of the root spin cluster exhibit the following asymptotics: 
\begin{itemize}
\item High temperature case. For $0< \nu < \nu_c$, there exist positive constants $\mathrm{C}^{\mathrm{v}}(\nu)$ and $\mathrm{C}^{\mathrm{p}}(\nu)$ such that:
\[
\mathbb P^\nu \left( |V(\mathfrak C (\mathbf T^\nu))| > n \right)
\underset{n \to \infty}{\sim} \mathrm{C}^{\mathrm{v}}(\nu) \ n^{-{13}/{7}} \quad \text{and} \quad \mathbb P^\nu \left( |\partial \mathfrak C (\mathbf T^\nu)| = p \right)
\underset{p \to \infty}{\sim} \mathrm{C}^{\mathrm{p}}(\nu) \ p^{-{10}/{3}}.
\]
\item Critical case. For $\nu = \nu_c = 1 + 1/\sqrt{7}$, there exist positive constants $\mathrm{C}^{\mathrm{v}}(\nu_c)$ and $\mathrm{C}^{\mathrm{p}}(\nu_c)$ such that:
\[
\mathbb P^\nu \left( |V(\mathfrak C (\mathbf T^\nu))| > n \right)
\underset{n \to \infty}{\sim} \mathrm{C}^{\mathrm{v}}(\nu_c) \ n^{-{17}/{11}} \quad \text{and} \quad \mathbb P^\nu \left( |\partial \mathfrak C (\mathbf T^\nu)| = p \right)
\underset{p \to \infty}{\sim} \mathrm{C}^{\mathrm{p}}(\nu_c) \ p^{-{14}/{3}}.
\]
\item Low temperature case. For $\nu > \nu_c$, there exists a positive constant $\mathrm{C}^{\mathrm{v}}(\nu)$ such that:
\[
\mathbb P^\nu \left( |V(\mathfrak C (\mathbf T^\nu))| = n \right)
\underset{n \to \infty}{\sim} \mathrm{C}^{\mathrm{v}}(\nu) \ n^{-{5}/{2}} \quad \text{and} \quad |\partial \mathfrak C (\mathbf T^\nu)| \, \text{has exponential tail}.
\]
\end{itemize}
\end{theo}

Again, we see that volume and perimeter critical exponents are the same in the whole high temperature regime, which includes the antiferromagnetic setting and the critical percolation setting. In particular, the case $\nu =1$ of Theorem~\ref{th:mainBoltzmann} recovers the exponents for critical percolation derived by Bernardi--Curien--Miermont \cite{BeCuMie}.

\bigskip

We also study the root spin cluster of the triangulations $\mathbf T_n^\nu$ with finite fixed volume. In this setting, we establish sharp asymptotics for the expected volume of the root spin cluster:

\begin{theo} \label{th:mainsizen}
In $\mathbf T^\nu_n$, the volume and the perimeter of the root spin cluster exhibit the following asymptotics: 
\begin{itemize}
\item High temperature case. For $0 < \nu < \nu_c$, there exists a positive constant $\mathrm{C}^{\mathrm{v}}_{\mathrm f}(\nu)$ such that:
\[
\mathbb E_{n}^\nu \left( |V(\mathfrak C (\mathbf T^\nu_n))|  \right)
\underset{n \to \infty}{\sim} \mathrm{C}^{\mathrm{v}}_{\mathrm f}(\nu) \, n^{3/4}.
\]
\item Critical case. For $\nu = \nu_c =  1 + 1/\sqrt{7}$, there exists a positive constant $\mathrm{C}^{\mathrm{v}}_{\mathrm f}(\nu_c)$ such that:
\[
\mathbb E_{n}^\nu \left( |V(\mathfrak C (\mathbf T^\nu_n))| \right)
\underset{n \to \infty}{\sim} \mathrm{C}^{\mathrm{v}}_{\mathrm f}(\nu_c) \, n^{5/6}.
\]
\item Low temperature case. For $\nu > \nu_c$, there exists a positive constant $\mathrm{C}^{\mathrm{v}}_{\mathrm f}(\nu)$ such that:
\[
\mathbb E_{n}^\nu \left( |V(\mathfrak C(\mathbf T^\nu_n))| \right)
\underset{n \to \infty}{\sim} \mathrm{C}^{\mathrm{v}}_{\mathrm f}(\nu) \, n.
\]
\end{itemize}
\end{theo}

The exponents $3/4$ and $5/6$ of Theorem~\ref{th:mainsizen} already appeared in an article by Borot, Bouttier and Duplantier \cite{BBD} as critical volume exponents for the gasket of an $\mathcal O (n)$ model on random maps. More precisely, the exponent $3/4$ corresponds to $n=1$ in the dense phase, corresponding to critical percolation. And the exponent $5/6$ corresponds to $n=1$ in the dilute phase, corresponding to the critical temperature Ising model.

\paragraph{Cluster boundary and looptrees.}

\begin{figure}[ht!]
\begin{center}
   \includegraphics[width=0.75\textwidth]{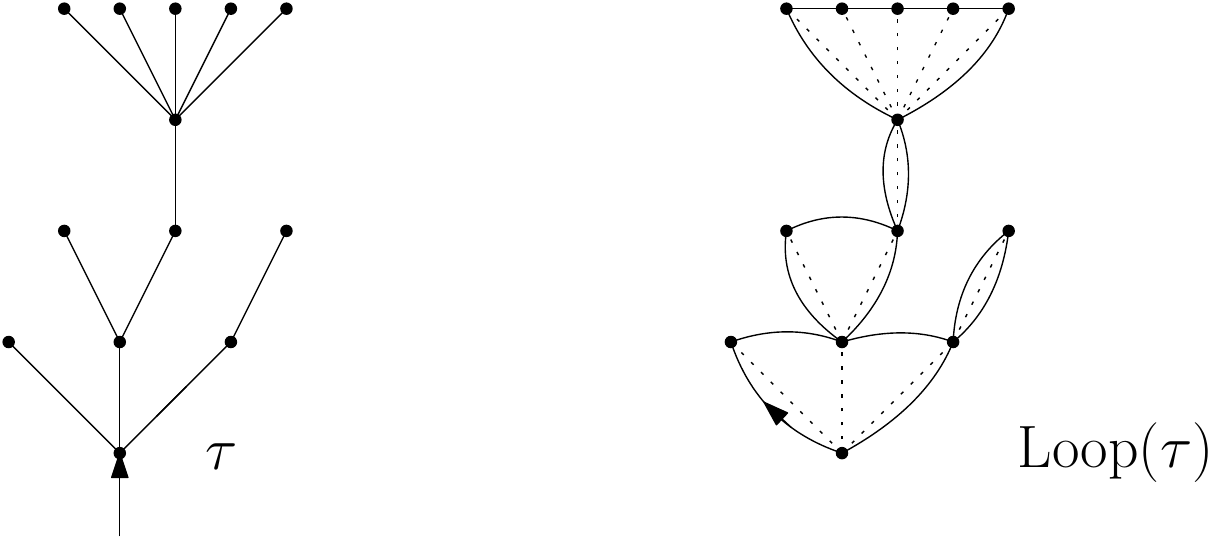}
    \caption{\label{fig:deflooptree} A planar rooted tree and its associated looptree.}
\end{center}
\end{figure}

Our last main result is about the geometry of the boundary of the root spin cluster seen as a discrete looptree and its scaling limit seen as a continuous looptree.

Looptrees were defined by Curien and Kortchemski in \cite{CKlooptrees} as a sort of dual objects for trees. The discrete case is as follows. To any tree $\tau$ we associate a graph $\mathrm{Loop}(\tau)$, called a looptree. The graph $\mathrm{Loop}(\tau)$ has the same vertex set as $\tau$, and there is an edge between two vertices $u$ and $v$ if and only if $u$ and $v$ are consecutive children of the same vertex in  $\tau$, or $v$ is the first child or the last child of $u$ in $\tau$. See Figure~\ref{fig:deflooptree} for an illustration. A continuous version of this definition is possible for random stable trees \cite{CKlooptrees}, which are Gromov-Hausdorff limits of critical Galton-Watson trees with offspring distribution in the domain of attraction of a stable distribution. Of particular interest is the random continuous looptree $\mathscr L_{3/2}$ associated to a random stable tree with index $3/2$, which is proved to be the scaling limit of the boundary of a critical site percolation cluster in the UIPT for the Gromov-Hausdorff topology by Curien and Kortchemski \cite{CuKo}. In the same article, the authors also prove that the scaling limit boundary of a supercritical site percolation cluster in the UIPT is the unit length cycle $\mathscr C_1$ for the Gromov-Hausdorff topology. Similar results were obtained for the boundary of random bipartite Boltzmann maps by Kortchemski and Richier \cite{CRlooptrees}.

\bigskip

Before stating our result, we need to introduce some additional notation. Denote by $\mathfrak H (\mathbf T_\infty^\nu)$ the hull of the cluster $\mathfrak C(\mathbf T_\infty^\nu)$, that is the submap of $\mathbf T_\infty^\nu$ spanned by vertices that are not inside the root face of $\mathfrak C(\mathbf T_\infty^\nu)$ (see Figure~\ref{fig:IntroCluster} for an illustration). From the one ended property of $\mathbf T_\infty^\nu$, either $\mathfrak H (\mathbf T_\infty^\nu)$ or $\mathbf T_\infty^\nu \setminus \mathfrak H (\mathbf T_\infty^\nu)$ is infinite. When $\mathfrak H (\mathbf T_\infty^\nu)$ is finite, its boundary $\partial \mathfrak C(\mathbf T_\infty^\nu)$ can be thought of as a typical spin interface. We denote by $\partial \mathfrak C_\infty^\nu(n)$ the boundary $\mathfrak C(\mathbf T_\infty^\nu)$ conditioned on the event that $\mathfrak H (\mathbf T_\infty^\nu)$ is finite and that the perimeter of $\mathfrak C(\mathbf T_\infty^\nu)$ is $n$. In a similar way, we denote by $\partial \mathfrak C^\nu(n)$ the boundary $\mathfrak C(\mathbf T^\nu)$ conditioned on the event that its perimeter is $n$.
We have the following scaling limits depending on the value of $\nu$:
\begin{theo} \label{th:looptrees}
For every $\nu >0$, there exists a positive constant $\mathrm{C}^{\mathrm{loop}}(\nu)$ such that the following convergences hold in distribution for the Gromov--Hausdorff topology:
\begin{itemize}
\item High temperature case. For $\nu < \nu_c$,
\begin{align*}
\frac{1}{n^{2/3}} \partial \mathfrak C_\infty^\nu(n) &\underset{n\to \infty}{\rightarrow} \mathrm{C}^{\mathrm{loop}}(\nu) \, \mathscr{L}_{3/2},
\quad \text{and} \quad
\frac{1}{n^{2/3}} \partial \mathfrak C^\nu(n) \underset{n\to \infty}{\rightarrow} \mathrm{C}^{\mathrm{loop}}(\nu) \, \mathscr{L}_{3/2}.
\end{align*}
\item Critical temperature case. For $\nu = \nu_c$, $\mathrm{C}^{\mathrm{loop}}(\nu_c) = \frac{1}{1+\sqrt{7}}$ and
\begin{align*}
\frac{1}{n} \partial \mathfrak C_\infty^\nu(n) &\underset{n\to \infty}{\rightarrow} \frac{1}{1+\sqrt{7}} \, \mathscr{C}_{1},
\quad \text{and} \quad
\frac{1}{n} \partial \mathfrak C^\nu(n) \underset{n\to \infty}{\rightarrow} \frac{1}{1+\sqrt{7}} \, \mathscr{C}_{1}.
\end{align*}
\item Low temperature case. For $\nu > \nu_c$, $\mathrm{C}^{\mathrm{loop}}(\nu) \in ( \frac{1}{1+\sqrt{7}} , \frac{\sqrt{3}-1}{2})$ and
\begin{align*}
\frac{1}{n} \partial \mathfrak C_\infty^\nu(n) &\underset{n\to \infty}{\rightarrow} \mathrm{C}^{\mathrm{loop}}(\nu) \, \mathscr{C}_{1},
\quad \text{and} \quad
\frac{1}{n} \partial \mathfrak C^\nu(n) \underset{n\to \infty}{\rightarrow} \mathrm{C}^{\mathrm{loop}}(\nu) \, \mathscr{C}_{1}.
\end{align*}
\end{itemize}
\end{theo}

\bigskip

Here again, we see that the boundary of the root spin cluster in the high temperature regime exhibits the same behavior as the boundary of a critical site percolation cluster on the UIPT as both models converge towards the stable looptree $\mathscr{L}_{3/2}$ in the scaling limit as established in \cite{CuKo}. The boundary of the root spin cluster in the high temperature regime also matches the behavior of the boundary of a supercritical site percolation cluster. Indeed, in the same work \cite{CuKo}, Curien and Kortchemski also prove that the boundary of the supercritical root spin cluster converges to the circle $\mathscr{C}_{1}$ in the scaling limit.

The limit at the critical temperature regime is not surprising either. We will see in this work that in this regime, the root spin cluster corresponds to a non generic critical Boltzmann map. Kortchemski and Richier \cite{CRlooptrees} proved that, in the bipartite case, the boundary of such random maps converges to the circle $\mathscr{C}_{1}$ in the scaling limit, agreeing with our result. This result at criticality can shed some light on the link between random maps coupled with an Ising model at the critical temperature and the $\mathcal O \left(\sqrt 2 \right)$ model in the dilute regime, where interfaces are supposed to be simple in the scaling limit, see for example the papers by Borot, Bouttier, Duplantier and Guitter \cite{BBD,BBGa,BBGc,BBGb}.

\subsection{General strategy and organization of the paper}

\begin{figure}[t!]
\begin{center}
 \subfloat[Triangulation with spins, with its root spin cluster emphasized,]{\quad\quad
      \includegraphics[width=0.35\textwidth,page=1]{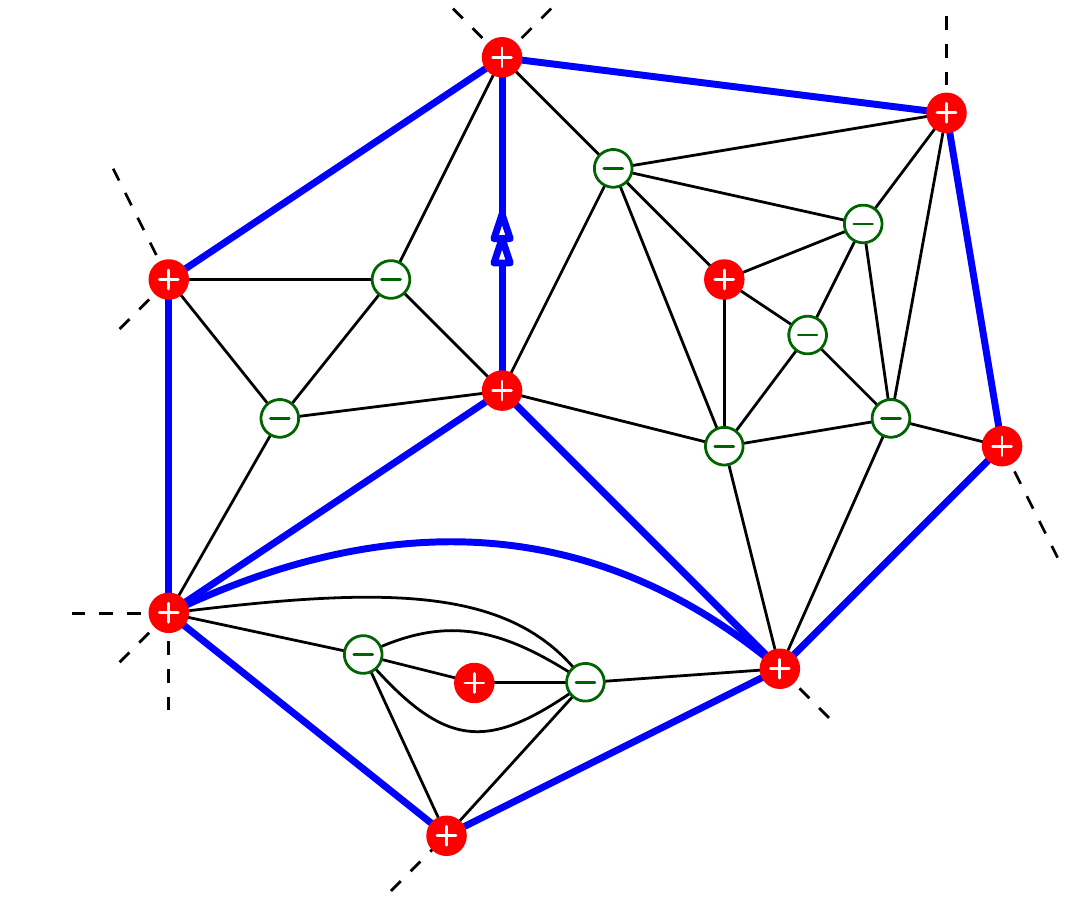}
      \label{subfig:clusterIntroa}\quad\quad\quad} 
     \subfloat[and its gasket decomposition.]{\quad\quad
      \includegraphics[width=0.35\textwidth,page=3]{ClusterIntro.pdf}
      \label{subfig:clusterIntrob}\quad}
    \caption{Example of the gasket decomposition of an Ising-weighted triangulation with monochromatic root edge. The pieces filling the holes of the gasket are the triangulations with a monochromatic boundary represented on the right.}
    \label{fig:IntroClusterGasket}
\end{center}
\end{figure}

\paragraph{Gasket Decomposition.}
The starting point of our work is the so-called gasket decomposition introduced by Borot, Bouttier and Guitter \cite{BBGa} to study the $\mathcal O(n)$ model on planar maps. Roughly speaking, this decomposition consists in partitioning a decorated planar map into an \emph{undecorated} map -- the gasket -- and pieces of the original map filling the holes of the gasket.

This decomposition was later used by Bernardi, Curien and Miermont \cite{BeCuMie} to study site percolation on random triangulations.
In their setting as well as in our setting, the gasket is the root spin cluster and the holes of the gasket are the faces of the root spin cluster, which are filled with triangulations with monochromatic boundaries as established in Proposition~\ref{prop:q_k}. See Figure~\ref{fig:IntroClusterGasket} for an illustration.

This implies that the probability to see a given root spin cluster in a random triangulation with spins is driven by the degrees of its faces. One of the most important models of random maps having this kind of property are the so-called \emph{Boltzmann maps} which have been intensively studied in the recent years \cite{MiermontInvariance,MarckertMiermontInvariance,BuddPeel,LeGallMiermont,Mar18b,CRlooptrees,EynardBook}.

\paragraph{Boltzmann planar maps.} Let $\mathbf q = \left( q_k \right)_{k\geq 1}$ be a sequence of nonnegative real numbers such that
\[
Z_\mathbf q := \sum_{\mathfrak m} \prod_{f \in \mathrm{Faces}(\mathfrak m)} q_{\mathrm{deg}(f)} < \infty,
\]
where the sum in the last display is over all finite rooted planar maps.
A $\mathbf q$-Boltzmann map is a random planar map $\mathbf M_\mathbf q$ such that for any fixed map $\mathfrak m$ one has
\[
\mathbb P \left(\mathbf M_\mathbf q = \mathfrak m  \right) = \frac{\prod_{f \in \mathrm{Faces}(\mathfrak m)} q_{\mathrm{deg}(f)}}{Z_\mathbf q}.
\]
We prove in Proposition~\ref{prop:bolt} that the root spin cluster is a Boltzmann map and we characterize its weight sequence as follows.

For every $l >0$, let $Q^+_l(\nu,t)$ be the generating series of triangulations with spins and a monochromatic, not necessarily simple, boundary of length $l$, counted with a weight $t$ per frustrated edge and $\nu t$ per monochromatic edge (see Section \ref{sec:defSG} for more details). We also define the series
\[
Q^+(\nu,t,y) = \sum_{l\geq 1} Q^+_l(\nu,t) \, y^l.
\]
We prove in Proposition~\ref{prop:bolt} that the generating series of the weight sequence $\mathbf q (\nu,t) = \left( q_k(\nu,t) \right)_{k \geq 1}$ associated to the root spin cluster is given by:
\begin{equation} \label{eq:genqk}
F(\nu ,t ,z) := \sum_{k \geq 1} q_k(\nu,t) \, z^k =  (\nu t)^{3/2} z^3 +
\frac{\sqrt{\nu t^3} z}{1- \sqrt{\nu t^3} z} \,
Q^+ \left( \nu,t , \frac{t}{1-\sqrt{\nu t^3}z} \right).
\end{equation}

\paragraph{Enumerative results.} 

The properties of a $\mathbf q$-Boltzmann map are driven by the asymptotic behavior of its weight sequence. Hence, in our case, the properties of the root spin cluster are driven by the analytic properties and the singularities of $Q^+$ via the paradigm of analytic combinatorics. We refer to the monograph by Flajolet and Sedgewick \cite{FS} for a must-read introduction to the subject.

The generating series $Q^+$ is fortunately algebraic and one of our key results is the derivation of a rational parametrization for it in Theorem~\ref{th:ratPar}. This derivation builds on two previous results. The first is a rational parametrization for $Q^+_1$ obtained by Bernardi and Bousquet--Mélou \cite{BernardiBousquet}. The second is an algebraic equation for a series analogous to $Q^+$ for simple monochromatic boundaries obtained by the authors and Schaeffer in \cite{IsingAMS}.

Thanks to this explicit rational parametrization, we are able to study extensively the analytic properties of $Q^+(\nu,t,y)$. In particular we locate and classify its singularities in $t$ or $y$ when the second variable is fixed. See for instance Proposition~\ref{lemm:weightsclusters} and Proposition~\ref{prop:serQty}.

\paragraph{Universal formulas for Boltzmann maps.}

Two fundamental quantities associated to Boltzmann maps are the so-called unpointed and pointed disk functions $W_\mathbf q$ and $W_\mathbf q^\bullet$ defined in Equations \eqref{eq:defDiskPartition} and \eqref{eq:defPointedDiskPartition}. Roughly speaking, they correspond to the generating series of the total weight of maps with a given perimeter. The pointed disk function has the following universal form:
\[
W_\mathbf q^\bullet (z) = \frac{1}{\sqrt{(z-c_+(\mathbf q))(z-c_-(\mathbf q))}},
\]
where $c_+(\mathbf q)$ and $c_-(\mathbf q)$ are real numbers depending on the weight sequence $\mathbf q$ in a non explicit way.
The unpointed function has no universal form, but it is analytic on the same domain $\mathbb C \setminus [c_-(\mathbf q),c_+(\mathbf q)]$ as the pointed function.

Remarkably, we are able to compute in Proposition~\ref{prop:Wnut} the unpointed disk function in our setting. It is given by:
\[
W_{\mathbf q(\nu,t)}(z) = \frac{1}{z} Q^+ \left(\nu,t,\frac{t}{z \sqrt{\nu t^3}} \right).
\]
This identity allows to express $c_+(\mathbf q(\nu,t))$ and $c_-(\mathbf q(\nu,t))$ in terms of the singularities of $Q^+$ and to study the distribution of the perimeter of the root spin cluster in the finite volume setting.

\bigskip

The two quantities $c_+(\mathbf q)$ and $c_-(\mathbf q)$ are also related to a system of equations which involves two bivariate functions $f_\mathbf q ^\bullet$ and $f_\mathbf q ^\diamond$, defined in \eqref{eq:deffbullet} and \eqref{eq:deffdiamond}, and originally introduced by Bouttier, Di Francesco and Guitter \cite{BDFG}. An essential feature of these two functions is that their singularities are linked to the tail distribution of the volume of the corresponding Boltzmann map.

However, as can be seen from their expressions, the dependency in $\mathbf q$ of the singularities of these two functions is nasty.
Another key point of our work is an explicit integral formula for both $f_\mathbf q ^\bullet$ and $f_\mathbf q ^\diamond$ stated in Proposition~\ref{prop:fbulletdiamond}. These formulas involve again the series $Q^+$ and allow to characterize the singularities of the two functions in Proposition~\ref{lem:asymptfbfg}, allowing in turn to study the distribution of the perimeter and volume of the root spin cluster in the finite volume setting. This allows to establish Theorem~\ref{th:mainBoltzmann} and Theorem~\ref{th:mainsizen}.

\paragraph{Infinite volume setting.}

In this setting, the root spin cluster is not a Boltzmann map, even if it is still related to this model. Indeed, for any fixed map $\mathfrak m$ and $n \geq 0$ we have
\[
\mathbb P_n^\nu \left( \mathfrak C(\mathbf T_n^\nu) = \mathfrak m \right) = \frac{[t^{3n}]\left( \prod_{f \in F(\mathfrak m)} q_{\mathrm{deg}(f)} (\nu ,t) \right)}{[t^{3n}] \mathcal Z ^\ps ( \nu ,t)}.
\]
By continuity of the event $\{ \mathfrak C = \mathfrak m\}$ for the local topology, we then have
\[
\mathbb P_\infty^\nu \left( \mathfrak C(\mathbf T_\infty^\nu) = \mathfrak m \right)
= \lim_{n \to \infty} 
\mathbb P_n^\nu \left( \mathfrak C(\mathbf T_n^\nu) = \mathfrak m \right).
\]
Thanks to our detailed analysis of the series $Q^+$, we are able to compute this limit in Proposition~\ref{prop:boltIIPT}.

Summing the limiting probabilities for every finite map $\mathfrak m$ gives
\[
\mathbb P_\infty^\nu \left( |\mathfrak C(\mathbf T_\infty^\nu) | < \infty \right)
=
\sum_{|\mathfrak m|<\infty}
\mathbb P_\infty^\nu \left( \mathfrak C(\mathbf T_\infty^\nu) = \mathfrak m \right).
\]
We can relate the latter quantity to $\mathbf q (\nu,t)$-Boltzmann maps with two boundaries and the so-called cylinder generating function for which universal formulas involving $c_+(\mathbf q(\nu,t))$ and $c_-(\mathbf q(\nu,t))$ are available. With some extra work, we manage to obtain an integral formula for this probability in Proposition~\ref{prop:formulaperco}. We then compute explicitly this integral using our rational parametrizations to establish Theorem~\ref{th:expo}.

To study the volume distribution of the root spin cluster in the IIPT, we use the volume generating function given by
\[
\mathbb E_\infty^\nu \left[ g^{|V(\mathfrak C(\mathbf T_\infty^\nu)) |}  \right]
=
\sum_{|\mathfrak m|<\infty}
g^{|V(\mathfrak m)|} \, 
\mathbb P_\infty^\nu \left( \mathfrak C(\mathbf T_\infty^\nu) = \mathfrak m \right).
\]
Again, we establish an integral formula for this quantity in Proposition~\ref{prop:formulaperco}. We are then able to analyze its singular behavior as $g \to 1^-$ thanks to our extensive analysis of the functions $f_\mathbf q ^\bullet$ and $f_\mathbf q ^\diamond$ and of their singularities, paving the way to prove Theorem~\ref{th:mainIIPT}.

\paragraph{Looptrees}

Each loop of the looptree associated to the root spin cluster corresponds to a triangulation with simple monochromatic boundary. Hence, the law of the looptree is linked to the generating series of triangulations with simple monochromatic boundary $Z^+$, instead of the series with non simple boundary $Q^+$ for the cluster itself.

Following Curien and Kortchemski \cite{CuKo}, we then encode the looptree by a two type Galton--Watson tree in Proposition~\ref{prop:2typeGWbolt} for the finite volume setting and in Proposition~\ref{prop:2typeGW} for the infinite volume setting. The offspring distribution of this random tree is a simple function of the series $Z^+$.
We study this series and establish a rational parametrization for it in Theorem~\ref{th:ratParZ}. Again, this parametrization allows to obtain all the asymptotic properties we need in Proposition~\ref{prop:asymptoZmono}.

\paragraph{Notation} A table of notations is available at the end of the paper, see page~\pageref{sec:notations}. 

\subsection{Perspectives and conjectures}

\paragraph{The $\alpha$-stable map paradigm}

We saw that the root spin cluster of the random triangulation $\mathbf T^\nu$ is a Boltzmann planar map. These maps can be encoded by multitype Galton--Watson trees via the Bouttier--Di Francesco--Guitter bijection \cite{BDFG}. The offspring distribution of these trees belongs to the domain of attraction of spectrally stable law of parameter $\alpha$, where $\alpha$ is linked to the asymptotic behavior of the disk partition function of the Boltzmann maps via the relations
\[
[z^{-l}] W_\mathbf q (z) \underset{l \to \infty}{\sim} \mathrm{Cst} \, \left(\rho_\mathbf q\right)^l \, l^{-a} \quad \text{and} \quad \alpha = a - \frac{1}{2}.
\]
Le Gall and Miermont \cite{LGM} studied the scaling limit of such maps conditioned to have size $n$ and proved that, after rescaling the distances by $n^{1/\alpha}$, they converge in law to a limit called the $\alpha$-stable map for Gromov--Hausdorff topology. To be fully accurate their result is for the bipartite setting and for subsequential limits.

In our setting, we see from Proposition~\ref{prop:asymptQk} that
\[
\alpha(\nu)
=
\begin{cases}
7/6 & \text{for $\nu < \nu_c$,}\\
11/6 & \text{for $\nu = \nu_c$,}\\
2 & \text{for }\nu > \nu_c.
\end{cases}
\]
This is strong evidence that the scaling limit of the root spin cluster is the Brownian map for $\nu > \nu_c$, the $11/6$-stable map for $\nu = \nu_c$ and the $7/6$-stable map for $\nu <\nu_c$. This paradigm can also be found for critical percolation in Bernardi, Curien and Miermont's work \cite{BeCuMie} and for the $\mathcal O(n)$ model in Borot, Bouttier and Guitter's work \cite{BBGa}.

\paragraph{Liouville Quantum Gravity and KPZ formula.}

There are several important conjectures for the scaling limits of random triangulations coupled with an Ising model. First, when seen as metric spaces, their scaling limit should be the Brownian map for $\nu \neq \nu_c$, with the usual renormalization with exponent $1/4$, and another object for $\nu=\nu_c$ with a different -- but unknown as of yet -- renormalization. In the case $\nu = 1$, our model corresponds to uniform triangulations decorated by a critical site percolation. The fact that their scaling limit is the Brownian map was proved independently by Le Gall \cite{LGbm} and Miermont \cite{Miebm}.

Secondly, some conjectures are also available in the setting of Liouville Quantum Gravity \cite{DS}:
\begin{itemize}
\item Low temperature ($\nu > \nu_c$). The scaling limit of the triangulation should be an undecorated $\sqrt{8/3}$-LQG surface.
\item High temperature ($\nu < \nu_c$). The scaling limit of the triangulation should also be a $\sqrt{8/3}$-LQG surface, but decorated by an independent CLE$_{6}$ representing the interfaces between spin clusters.
\item Critical temperature ($\nu = \nu_c$). The scaling limit of the triangulation should be a $\sqrt 3$-LQG surface decorated by an independent CLE$_{3}$ representing the interfaces between spin clusters.
\end{itemize}
The fact that the Brownian map is equivalent to a $\sqrt{8/3}$-LQG surface has been established by Miller and Sheffield in the series of papers \cite{MSa,MSb,MSc}. 
Moreover, Holden and Sun \cite{HoldenSun} proved the conjecture for $\nu = 1$ by canonically embedding the percolated triangulations in the sphere, building on a previous work by Bernardi, Holden and Sun \cite{BeHoSun} where the embedding was less explicit.

\bigskip

Our work provides evidence comforting these conjectures through the KPZ relation, which originates in the theoretical physics literature in the paper by Knizhnik, Polyakov and Zamolodchikov~\cite{KPZ}. It was then derived rigorously in some cases by Duplantier and Sheffield~\cite{DS}. See also the survey by Garban~\cite{GarbanKPZ}. This relation states that given a fractal $\mathcal X$ on a $\gamma$-LQG surface, its Euclidean conformal weight $x$ is related to its quantum counterpart $\Delta$ by
\[
x = \frac{\gamma^2}{4} \Delta^2 + \left( 1 - \frac{\gamma^2}{4}\right) \Delta.
\] 
The Euclidean conformal weight $x$ of $\mathcal X$ is related to its fractal dimension by the relation $\mathrm{dim}_H (\mathcal X) = 2 (1-x)$. Similarly, the quantum conformal weight $\Delta$ of $\mathcal X$ is defined by how the quantum area of $\mathcal X$ scales as follows: the quantum area (or mass) of the part of $\mathcal X$ inside a subset of the LQG of quantum area $A$ scales as $A^{1-\Delta}$.

Hence, we can check that the critical exponents obtained in Theorem~\ref{th:mainIIPT} and Theorem~\ref{th:mainBoltzmann} agree with their Euclidean counterparts via this relation.
Our perimeter exponents give 
\[
\Delta^{\mathrm p} (\nu) =
\begin{cases}
1/2 & \text{for $\nu < \nu_c$,}\\
1/4 & \text{for $\nu = \nu_c$.}
\end{cases}
\]
Indeed Theorem $\ref{th:mainIIPT}$ directly gives gives $\frac{3}{4} = 1- \Delta^{\mathrm p} (\nu)$ for $\nu < \nu_c$ and $\frac{1}{2} = 1- \Delta^{\mathrm p} (\nu_c)$. To check the consistency with the exponents established in Theorem \ref{th:mainBoltzmann}, we need to take into account the typical size of the Boltzmann triangulations. This gives the relations $\frac{5}{2} \cdot \frac{3}{10} = 1- \Delta^{\mathrm p} (\nu)$ for $\nu <\nu_c$ and $\frac{7}{3} \cdot \frac{3}{14} = 1- \Delta^{\mathrm p} (\nu_c)$. Similarly, the area exponents are
\[
\Delta^{\mathrm v} (\nu) =
\begin{cases}
1/8 & \text{for $\nu < \nu_c$,}\\
1/12 & \text{for $\nu = \nu_c$.}
\end{cases}
\]
coming from the relations $\frac{7}{8} = \frac{5}{2} \cdot \frac{7}{20} = 1- \Delta^{\mathrm v} (\nu)$ for $\nu < \nu_c$ and $\frac{11}{12} = \frac{7}{3} \cdot \frac{11}{28} = 1- \Delta^{\mathrm v} (\nu_c)$.

At the critical temperature, applying the KPZ relation with $\gamma = \sqrt 3$ to $\Delta^{\mathrm p} (\nu_c)$ and $\Delta^{\mathrm v} (\nu_c)$ give $x^{\mathrm{p}}(\nu_c) = \frac{5}{16}$ and $x^{\mathrm{v}}(\nu_c) = \frac{5}{192}$. These conformal weights agree with the dimension $\frac{11}{8}$ of a SLE$_3$ curve calculated by Beffara \cite{SLEdim} and the dimension $\frac{187}{96}$ of the gasket of a CLE$_3$ calculated by Miller, Sun and Wilson \cite{CLEdim}.

In the high temperature regime,applying the KPZ relation with $\gamma = \sqrt{8/3}$ to $\Delta^{\mathrm p} (\nu)$ and $\Delta^{\mathrm v} (\nu)$ give $x^{\mathrm{p}}(\nu) = \frac{1}{8}$ and $x^{\mathrm{v}}(\nu) = \frac{5}{96}$. These conformal weights agree with the dimension $\frac{7}{4}$ of a SLE$_6$ curve calculated in \cite{SLEdim} and the dimension $\frac{91}{48}$ of the gasket of a CLE$_6$ calculated in \cite{CLEdim}.

\subsection*{Acknowledgments}

We thank Jérémie Bouttier for detailed discussions on the series of papers \cite{BBGa,BBD,BBGc,BBGb}.
M.A.'s research was supported by the ANR grant GATO (ANR-16-CE40-0009-01) and the ANR grant IsOMa (ANR-21-CE48-0007). L.M. acknowledges the support of the Labex MME-DII (ANR11-LBX-0023-01). Both authors are supported by the ANR grant ProGraM (ANR-19-CE40-0025).

\section{Preliminaries} \label{sec:prelim}

\subsection{Triangulations with spins: definitions} 

A \emph{planar map} is the embedding of a planar graph in the sphere considered up to sphere homeomorphisms preserving its orientation. Note that loops and multiple edges are allowed. We only consider \emph{rooted} maps, meaning that one oriented edge is distinguished. This edge is called the \emph{root edge}, its tail the \emph{root vertex} and the face on its right the \emph{root face}. The \emph{size} of a planar map $\mathfrak m$ is its number of edges and is denoted by $|\mathfrak m|$. The set of non-atomic maps (i.e. maps with at least one edge)
is denoted by $\mathcal{M}$. A \emph{rooted and pointed map} is an element of $\mathcal{M}$ with an additional marked vertex. The set of non-atomic pointed maps is denoted by $\mathcal{M}^\bullet$.

A \emph{triangulation} is a planar map in which all the faces have degree 3.
More generally, a \emph{triangulation with a boundary} is a planar map in which all faces have degree $3$ except for the root face (which may or may not be simple) and 
a \emph{triangulation of the $p$-gon} is a triangulation whose root face is bounded by a simple cycle and has degree $p$.

A map with a \emph{spin configuration} is a pair $(\mathfrak m,\sigma)$, where $\mathfrak m$ is a planar map and $\sigma$  is a mapping from the set $V(\mathfrak m)$ of vertices of $\mathfrak m$ to the set $\{\ps,\ns\}$. When it is clear from the context that the maps considered carry a spin configuration, we will sometimes call them maps without the additional precision. An edge $\{u,v\}$ of $(\mathfrak m,\sigma)$ is called \emph{monochromatic} if $\sigma(u)=\sigma(v)$ and \emph{frustrated} otherwise. The number of monochromatic edges of $(\mathfrak m,\sigma)$ is denoted by $m(\mathfrak m,\sigma)$.

\subsection{Review of existing results about enumeration of triangulations with spins} \label{sec:sectionAMS}

\subsubsection{Rational parametrization for monochromatic simple boundary}
One of the main contributions of the previous work \cite{IsingAMS} by the authors and Gilles Schaeffer is the study of the generating series of triangulations with spins and simple boundary. We will in particular need the results of \cite{IsingAMS} for monochromatic simple boundary. Let us denote by $\mathcal T_{p}^+$ the set of all finite triangulations of the $p-$gon with spins and with positive monochromatic boundary. Define its generating series by
\begin{equation}\label{eq:defZp+}
Z_p^+ (\nu , t ) = \sum_{(\mathfrak t,\sigma) \in \mathcal T_p^+} \nu^{m(\mathfrak t,\sigma)} t^{|\mathfrak t|},
\end{equation}
and let
\begin{equation} \label{eq:defZ+}
Z^+(\nu,t,y) = \sum_{p \geq 1} Z_p^+ (\nu , t ) \, y^p.
\end{equation}
Theorem 2.11 of \cite{IsingAMS} establishes an algebraic equation linking the series $Z^+(\nu,t,y)$, $Z_1^+(\nu,t)$, $Z_2^+(\nu,t)$ and the variables $\nu,t,y$. We do not display this equation here, but it will be needed later and is included in the Maple companion file \cite{Maple}.

The study of these series relies on the following rational parametrization: 

\begin{prop}[Theorem~23 of~\cite{BernardiBousquet}, Proposition~2.12 of \cite{IsingAMS}]\label{prop:paramU}
Let $U\equiv U(\nu,t^3)$ defined as the unique formal power series in $t^3$ having constant term $0$ and satisfying
\begin{equation} \label{eq:wU}
t^3= U \frac{\Big((1+\nu)U-2 \Big)\,P(\nu,U)}{32\nu^3 (1-2U)^2},
\end{equation}
with
\begin{equation} \label{eq:defPU}
P(\nu,U)= 8(\nu+1)^2U^3-(11\nu+13)(\nu+1)U^2+2(\nu+3)(2\nu+1)U-4\nu.
\end{equation}
Then, for every $p$, the series $t^p Z^+_p$ admits a rational parametrization in terms of $U$. 
\end{prop}

The series $U$ originally appeared in the work of Bernardi and Bousquet-Mélou \cite{BernardiBousquet}. Its properties and singularities were extensively studied in \cite{IsingAMS}. We summarize briefly what will be needed in the present work.
\bigskip

Set
\begin{equation} \label{eq:nucdef}
\nu_c = 1 + 1/\sqrt 7.
\end{equation}
The radius of convergence $t_\nu^3 \in (0,\infty)$ of $U$ satisfies:
\begin{align*}
  P_2(\nu, t^3_\nu)&= 0 \hbox{ for } 0< \nu \le
  \nu_c,\\
P_1(\nu, t^3_\nu)&=0 \hbox{ for } \nu_c \le \nu,
\end{align*}
where $P_1$ and $P_2$ are the following two polynomials:
\begin{align}
 P_1(\nu,\rho) &= 
 131072\,{\rho}^{3}{\nu}^{9}-192\,{\nu}^{6} \left( 3\,\nu+5 \right) 
 \left( \nu-1 \right)  \left( 3\,\nu-11 \right) {\rho}^{2}
\notag\\
& \quad \quad-48\,{\nu}^{3} \left( \nu-1 \right) ^{2}\rho+ \left( \nu-1 \right)  \left( 4\,{\nu
}^{2}-8\,\nu-23 \right),\label{eq:P1}\\
P_2(\nu,\rho) & = 
27648\,{\rho}^{2}{\nu}^{4}+864\,\nu\, \left( \nu-1 \right)  \left( {
\nu}^{2}-2\,\nu-1 \right) \rho+ \left( 7\,{\nu}^{2}-14\,\nu-9 \right) 
 \left( \nu-2 \right) ^{2}.\label{eq:P2}
\end{align}
To define unambiguously $t_\nu^3$, we need to specify which root of $P_1$ and $P_2$ we choose. This is described in full details in~\cite[Definition~2.1]{IsingAMS} and recalled in the Maple companion file \cite{Maple}. The right solution is the only positive solution of $P_2$ for $\nu \leq \nu_c$, and the only positive decreasing branch of $P_1$ for $\nu \geq \nu_c$.

The value of $U$ at its radius of convergence will be of special interest throughout this work and is calculated in the following statement.

\begin{prop} \label{prop:Unuval}
For every $\nu >0$ define $U_\nu := U(\nu,t_\nu^3)$. This quantity satisfies:
\begin{itemize}
\item High temperature case. If $0 < \nu \leq \nu_c$ then
\begin{equation}\label{eq:UnuSubCrit}
 U_\nu = \frac{3\nu+3-\sqrt{-3\nu^2+6\nu+9}}{6(\nu+1)} \quad \text{ and }\quad\nu = \frac{-3U_\nu(U_\nu-1)}{3U_\nu^2-3U_\nu+1}, 
\end{equation}
\item Low temperature case. Set $K_{\nu_c}:=\frac{\sqrt{7}-2}{3}$ and $K_{\infty}:=\sqrt{3}$. The mapping 
\begin{equation}\label{eq:defKnu}
\hat \nu(K) = -\frac{K^3+3K^2+9K+11}{(K+3)(K^2-3)}.
\end{equation}
is increasing from $[K_{\nu_c},K_\infty)$ onto $[\nu_c,+\infty)$. Moreover, for every $\nu \geq \nu_c$, if $K_\nu \in [K_{\nu_c},K_\infty)$ is such that $\hat \nu (K_\nu) = \nu$, then
\begin{equation}\label{eq:defKU}
U_\nu = \frac{3- K_\nu^2}{6K_\nu+10},
\end{equation}
\end{itemize}
We note for later reference that $U_{\nu_c}=(5-\sqrt{7})/9$ and that $U_\nu \in (0, U_{\nu_c})$ for every $\nu \neq \nu_c$.
\end{prop}
\begin{proof}
 It was proved in~\cite[Proof of Lemma~2.7]{IsingAMS} that $U_\nu$ is the smallest positive root of the polynomial
\begin{equation}\label{eq:Unu}
\begin{cases}
3(1+\nu)U_\nu^2 - 3 (1+\nu)U_\nu +\nu=0&\text{ when }\nu\leq \nu_c\\
4(1+\nu)^2 U_\nu^3 - 3 (1+\nu)(\nu+3) U_\nu^2 + 6 (1+\nu) U_\nu -2 &\text{ when }\nu\geq \nu_c.
\end{cases}
\end{equation}
The calculation for the following discussion are available in the Maple companion file \cite{Maple}.

When $\nu \leq \nu_c$, it is easy to verify that Equation~\eqref{eq:UnuSubCrit} is indeed the smallest root and the expression of $\nu$ follows easily since the polynomial has degree $1$ in $\nu$.

When $\nu \geq \nu_c$, one can first check that the parametrization given in Equations \eqref{eq:defKnu} and \eqref{eq:defKU} are solutions of \eqref{eq:Unu}. We can then check that $\hat \nu$ is increasing in $K$ and that \eqref{eq:defKU} is decreasing in $K$. Moreover they have the right value at $\nu_c$ and since the smallest root of \eqref{eq:Unu} is never a double root, Equation \eqref{eq:defKU} is indeed the smallest root of the polynomial defining $U_\nu$ for $\nu \geq \nu_c$.
\end{proof}

\subsubsection{Asymptotic expansions}
Another result from \cite{IsingAMS} that will be heavily used in this work is the singular development of $U$ at its radius of convergence:

\begin{lemm}[Lemma 2.7 of \cite{IsingAMS}]\label{lem:propU}
For every fixed $\nu>0$, the series $U$ defined in~\eqref{eq:wU} and seen as a series in $t^3$, has nonnegative coefficients, and a unique dominant singularity at $t_\nu^3$. 
Moreover, it has the following singular behavior at $t_{\nu}^3$:
\[
U(\nu,t^3) =
\begin{cases}
U_\nu  +\displaystyle\sum_{i\in\{\frac{1}{2},1,\frac{3}{2}\}}\beth_i^U (\nu) \cdot \left( 1- \Big(\frac{t}{t_\nu}\Big)^{\!3} \right)^{i}+o\left(\left( 1- \Big(\frac{t}{t_\nu}\Big)^{\!3} \right)^{3/2}\right) & \text{ for $\nu \neq \nu_c$, }\\
U_{\nu_c} +\displaystyle \sum_{i\in\{\frac{1}{3},\frac{2}{3},1,\frac{4}{3}\}}\beth_i^U(\nu_c) \cdot \left( 1- \Big(\frac{t}{t_\nu}\Big)^{\!3} \right)^{i}+o\left(\left( 1- \Big(\frac{t}{t_\nu}\Big)^{\!3} \right)^{4/3}\right)
& \text{ for $\nu = \nu_c$},
\end{cases}
\]
where $U_\nu=U(\nu,t_\nu^3)$, and $\beth_i^U(\nu)$ are explicit functions of $\nu$ given in the Maple companion file \cite{Maple}.
\end{lemm}

The asymptotic expansions of the previous lemma can be computed explicitly to any fixed order (see the Maple companion file \cite{Maple}). This and the form on the expression of the series $t^pZ_p^+(\nu,t)$ in terms of $U$ given in Proposition~\ref{prop:paramU} leads to the following proposition proved in \cite{IsingAMS}:

\begin{prop}[Theorem 2.4 of \cite{IsingAMS}] \label{prop:asymptoZp+}
Fix $\nu > 0$ and $p \geq 1$, then $t^p Z_p^+(\nu,t)$ seen as a series in $t^3$ is algebraic and has a unique dominant singularity at $t_\nu^3$, where it has the following asymptotic expansion:
\begin{align*}
t^pZ_p^+(\nu,t) &= t^p_\nu Z_p^+(\nu,t_\nu) \\
& \qquad + \beth_1^{Z^+_p}(\nu) \left( 1- \Big(\frac{t}{t_\nu}\Big)^{\!3} \right) + \beth^{Z^+_p}(\nu) \left( 1- \Big(\frac{t}{t_\nu}\Big)^{\!3} \right)^{\gamma(\nu)-1} + o \left( \left( 1- \Big(\frac{t}{t_\nu}\Big)^{\!3} \right)^{\gamma(\nu)-1} \right),
\end{align*}
where 
\begin{equation}\label{eq:gamma}
\gamma(\nu) = 
\begin{cases}
5/2 & \text{for $\nu \neq \nu_c$,}\\
7/3 & \text{for $\nu = \nu_c$,}
\end{cases}
\end{equation}
and where $\beth_1^{Z^+_p}(\nu)$ and $\beth^{Z^+_p}(\nu)$ are non vanishing functions of $\nu$.
\end{prop}

Recall that $\mathcal T^\ps$ is the set of finite triangulations such that the root vertex carries a spin $\ps$ and the root edge is monochromatic and that $\mathcal Z^{\ps}(\nu,t)$ denotes their generating series. Opening the root edge of such triangulations gives the following identity (see Figure~\ref{fig:rootTransform}):
\begin{equation}\label{eq:condplus}
\mathcal Z^{\ps} (\nu , t) = \frac{1}{\nu t} ( Z_1^+(\nu,t)^2 + Z_{2}^+ (\nu,t)).
\end{equation}
\begin{figure}
\begin{center}
\includegraphics[width = \linewidth]{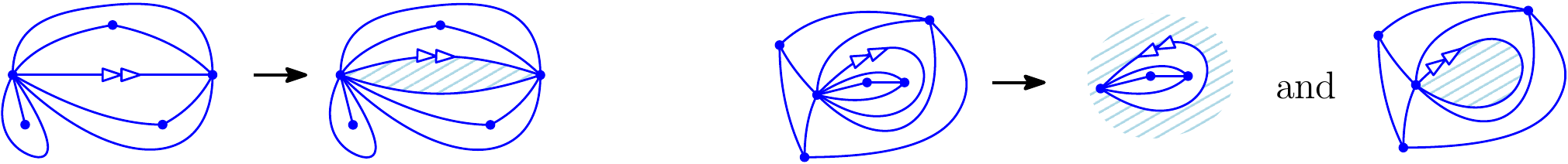}
\caption{\label{fig:rootTransform} Transformation of a triangulation of the sphere with a monochromatic root edge in either a triangulation of the 2-gon or two triangulations of the 1-gon, all with monochromatic boundary conditions.}
\end{center}
\end{figure}
Combining this identity with Proposition~\ref{prop:asymptoZp+} directly yields: 
\begin{prop} \label{prop:asymptoZmono}
Fix $\nu > 0$, then $\mathcal Z^{\ps}(\nu,t)$ seen as a series in $t^3$ is algebraic and has a unique dominant singularity at $t_\nu^3$, where it has the following asymptotic expansion:
\begin{multline*}
\mathcal Z^\ps(\nu,t) = \mathcal Z^\ps(\nu,t_\nu) + \beth_1^{\mathcal Z^\ps}(\nu) \left( 1- \Big(\frac{t}{t_\nu}\Big)^{\!3} \right)\\
 + \beth^{\mathcal Z^\ps}(\nu) \left( 1- \Big(\frac{t}{t_\nu}\Big)^{\!3} \right)^{\gamma(\nu)-1} + o \left( \left( 1- \Big(\frac{t}{t_\nu}\Big)^{\!3} \right)^{\gamma(\nu)-1} \right),
\end{multline*}
where $\gamma(\nu)$ is defined in~\eqref{eq:gamma} and $\beth_1^{\mathcal Z^\ps}(\nu)$ and $\beth^{\mathcal Z^\ps}(\nu)$ are explicit non vanishing functions of $\nu$ given in the Maple companion file \cite{Maple}.
\end{prop}
A direct application of the transfer theorem -- see Theorem~\ref{th:transfer} in Appendix~\ref{sec:algebraic} -- then gives
\begin{coro}\label{cor:equivZ+}
For any $\nu>0$, we have:
\begin{equation}\label{eq:equivZ+}
[t^{3n}]\mathcal Z^\ps(\nu,t) \stackrel[n\rightarrow \infty]{}{\sim} \frac{\beth^{\mathcal Z^\ps}(\nu)}{\Gamma\Big(1-\gamma(\nu)\Big)}\, t_\nu^{-3n}\, n^{-\gamma(\nu)}.
\end{equation}
\end{coro}

\subsection{Local topology and the IIPT} \label{sec:IIPTdef}

As mentioned in the introduction, the present work deals with the Ising Infinite Planar Triangulation (IIPT), which is the limit in law of large triangulations coupled with an Ising model for the local topology. We briefly recall the definition of this local topology and state the convergence result from \cite{IsingAMS} that we need.

\bigskip

The local topology for planar maps was first considered by Benjamini and Schramm~\cite{BS}. It can be defined from the following distance $d_{\mathrm{loc}}$ on triangulations with spins:
\begin{equation}
\label{eq:dloc}
\forall (\mathfrak t, \sigma) , \, (\mathfrak t', \sigma') \in \mathcal T : \quad
d_{\mathrm{loc}}((\mathfrak t,\sigma),(\mathfrak t',\sigma')) = \frac{1}{1 + \sup \{ R \geq 0 : B_r (\mathfrak t,\sigma) = B_r(\mathfrak t',\sigma')\}},
\end{equation}
where $B_r(\mathfrak t,\sigma)$ is the submap of $\mathfrak t$ composed by its faces having at least one vertex at distance smaller than $r$ from its root vertex, with the corresponding spins. The only difference with the usual setting is the presence of spins on the vertices and, in addition of the equality of the underlying maps, we require that spins coincide to say that two maps are equal. 

The closure $\left( \overline{\mathcal T},d_{\mathrm{loc}} \right)$ of the metric space $(\mathcal T,d_{\mathrm{loc}})$ is a Polish space and elements of $\overline{\mathcal T} \setminus \mathcal T$ are called infinite triangulations with spins. The local topology on triangulations with spins is the topology induced by $d_{\mathrm{loc}}$.

\bigskip

Recall that $\mathbb P_n^\nu$ is the probability measure on triangulations with spins with $3n$ edges (and monochromatic root edge with spin $\ps$) defined by
\[
\forall (\mathfrak t,\sigma) \in \mathcal T^\ps , \quad \mathbb P_n^\nu \left( \{(\mathfrak t,\sigma) \}\right) =
\frac{\nu^{m(\mathfrak t,\sigma)}}{[t^{3n} ] \mathcal Z^{\ps}(\nu,t)} \mathbf{1}_{\{ |\mathfrak t| = 3n  \}}.
\]
The main probabilistic result of \cite{IsingAMS} is: 
\begin{theo}[Theorem 1.1 of \cite{IsingAMS}]
\label{th:localCV}
For every $\nu>0$, the sequence of probability measures $\mathbb P_n^\nu$ converges weakly for the local topology to a limiting probability measure $\mathbb P_\infty^\nu$ supported on one-ended infinite triangulations endowed with a spin configuration.

We call a random triangulation distributed according to this limiting law the \emph{Infinite Ising Planar Triangulation with parameter $\nu$}, or \emph{$\nu$-IIPT}.
\end{theo}

\begin{rema}
In~\cite{IsingAMS}, the result is in fact established for triangulations with no spin constraint on their root edge. However, here, we stated the result for triangulations conditioned to have a monochromatic positive root edge. 
Since this conditioning is not degenerate, the result stated above follows directly from the result established in~\cite{IsingAMS}.
\end{rema}

\section{Enumerative results} \label{sec:combi}

This section is devoted to enumerative results for triangulations with a monochromatic boundary, which will be instrumental in the rest of the paper. We start by establishing a rational parametrization for the generating series $Q^+$ of triangulations with a non simple monochromatic boundary. We then compute the critical points and the singular expansions of this parametrization, which can be transfered to obtain the critical points and the singular expansions of $Q_+$. In the last section, we obtain similar results for the generating series $Z^+$ of triangulations with a \emph{simple} monochromatic boundary, which will appear naturally when we study the cluster interfaces via the framework of looptrees in Section~\ref{sec:looptrees}.

\subsection{Rational parametrization of triangulations with non simple monochromatic boundary} \label{sec:defSG}

For every $p\geq 0$, let us denote by $\mathcal Q_p^+$ the set of finite triangulations with spins and a (non necessarily simple) boundary of length $p$ with positive monochromatic boundary condition, i.e. all the vertices incident to the boundary have spin $\ps$. We adopt the convention that $\mathcal Q_0^+$ is the set containing only the atomic map with one vertex and no edge and set $\mathcal Q^+ = \cup_{p \geq 0} \mathcal Q_p^+$. We denote by $Q_p^+(\nu,t)$ the generating series of triangulations in $\mathcal Q_p^+$ counted by edges and monochromatic edges, that is
\[
Q_p^+(\nu,t) = \sum_{(\mathfrak t,\sigma) \in \mathcal Q_p^+} t^{|\mathfrak t|} \nu^{m(\mathfrak t,\sigma)}.
\]
We also set
\[
Q^+ (\nu,t,y) = \sum_{p \geq 0} Q_p^+(\nu,t) y^p.
\]
We first check that these series are well-defined:
\begin{lemm} \label{lemm:RdC}
Fix $\nu >0$. For every $p$, the radius of convergence of $Q_p^+(\nu,t)$ is equal to $t_\nu$ and $Q_p^+(\nu,t_\nu)<\infty$. This implies in particular that for any $t\in \bar D(0,t_\nu)$, $Q^+(\nu,t,y)$ is well-defined as a formal power series in $y$. 

Moreover, $t^pQ_p^+(\nu,t)$, seen as a series in $t^3$, is algebraic and has a unique dominant singularity at $t_\nu^3$.
\end{lemm}

\begin{proof} A triangulation with a non simple boundary of length $p$ can be canonically decomposed into a collection of (at most $p$) triangulations with simple boundaries, each of perimeter at most $p$. Therefore, for every $p$, $Q_p^+$ is a polynomial in $(Z_i^+)_{1\leq i \leq p}$. 
The lemma follows easily from this fact and from Proposition \ref{prop:asymptoZp+}.
\end{proof}

All the enumerative properties that we will need for the present work stem from the following algebraicity result and its associated rational parametrization:
\begin{theo}\label{th:ratPar}
Recall the definition of $U$ given in Proposition~\ref{prop:paramU}. We define similarly $V\equiv V(\nu,U,y)$ as the unique formal power series in $\mathbb{Q}[\nu,U]\llbracket y \rrbracket$ having constant term $0$ and satisfying the following equation:
\begin{align}\label{eq:defyV}
y = \hat Y (\nu,U,V),
\end{align}
where $\hat Y (\nu,U,V)$ is the following rational fraction:
\[
\begin{split}
\hat Y (\nu,U,V) &= \frac{8\nu(1-2U)}{U\Big((1+\nu)\cdot U-2\Big)}\\
&\qquad\cdot\frac{V(V+1)}{V^3+\dfrac{9(1+\nu)\cdot U^2-2(3+10\nu)U+8\nu}{U\Big((1+\nu)\cdot U-2\Big)}V^2-\dfrac{9(1+\nu)\cdot U-2(2\nu+3)}{U\Big((1+\nu)\cdot U-2\Big)}V-1}.
\end{split}
\]
The series $Q^+ (\nu,t,ty)$ is algebraic and admits the following rational parametrization in terms of $U(\nu,t^3)$ and $V(\nu,U(\nu,t^3),y)$. For every $\nu>0$, we have the identity
\begin{align}\label{eq:paramQ}
Q^+(\nu,t,ty)=\hat Q^+(\nu,U(\nu,t^3),V(\nu,U(\nu,t^3),y))
\end{align}
as formal power series in $t^3$ and $y$,
where $\hat Q^+(\nu,U,V)$ is the following rational function:
\begin{equation}\label{eq:ratParQ}
\begin{split}
\hat Q^+(\nu,U,V) &= U\cdot\frac{\Big((1+\nu)\cdot U-2\Big)(1-\nu)}{(V+1)^3\cdot P(\nu,U)}\\
&\times \left(V^3+\dfrac{9(1+\nu)\cdot U^2-2(3+10\nu)\cdot U+8\nu}{U\cdot\Big((1+\nu)\cdot U-2\Big)}\cdot V^2-\dfrac{9(1+\nu)\cdot U-2(2\nu+3)}{U\cdot\Big((1+\nu)\cdot U-2\Big)}\cdot V-1\right)\\
&\times \left(V^2+\frac{5(1+\nu)\cdot U^2-2(3\nu+2)\cdot U+2\nu}{U\cdot\Big((1+\nu)\cdot U-2\Big)}\cdot 2V-\frac{P(\nu,U)}{U\cdot\Big((1+\nu)\cdot U-2\Big)(1-\nu))}\right),
\end{split}
\end{equation}
with $P(\nu,U)$ defined in Proposition \ref{prop:paramU}.
\end{theo}
\begin{proof}
We start by the proof of algebraicity. It is classical that a triangulation with a (not necessarily simple) boundary can be decomposed into a triangulation with a simple boundary on which are grafted triangulations with a boundary. More precisely, let $\mathfrak t$ be a triangulation with a simple boundary of length $p$ and let $(\mathfrak q_1,\ldots, \mathfrak q_p)$ be a collection of triangulations with a boundary. Then, for $i\in \{1,\ldots p\}$, we can graft $\mathfrak q_i$ on $\mathfrak t$ by merging the root corner of $\mathfrak q_i$, with the $i$-th corner of the root face of $\mathfrak t$ (starting from the root corner), see Figure~\ref{fig:SimpleVsNonSimple}. This construction is in fact a bijection between $\mathcal{Q}\backslash\{\dagger\}$ (where $\dagger$ is the atomic map) on the one hand and the set: 
$ \cup_{p\geq 1}\Big\{\mathcal{T}_p \times \mathcal{Q}^p\Big\}$, where $\mathcal{Q}^p$ denotes the set of $p$-tuples of elements of $\mathcal{Q}$. This bijection can of course be extended to triangulations endowed with an Ising configuration with positive boundary conditions, and hence yields the following identity for generating series: 
\begin{equation}\label{eq:relQZ}
  Q^+(\nu,t,y)=Z^+\Big(\nu,t,yQ^+(\nu,t,y)\Big)+1.
\end{equation}
\begin{figure}
\centering
\includegraphics[width=0.9\linewidth]{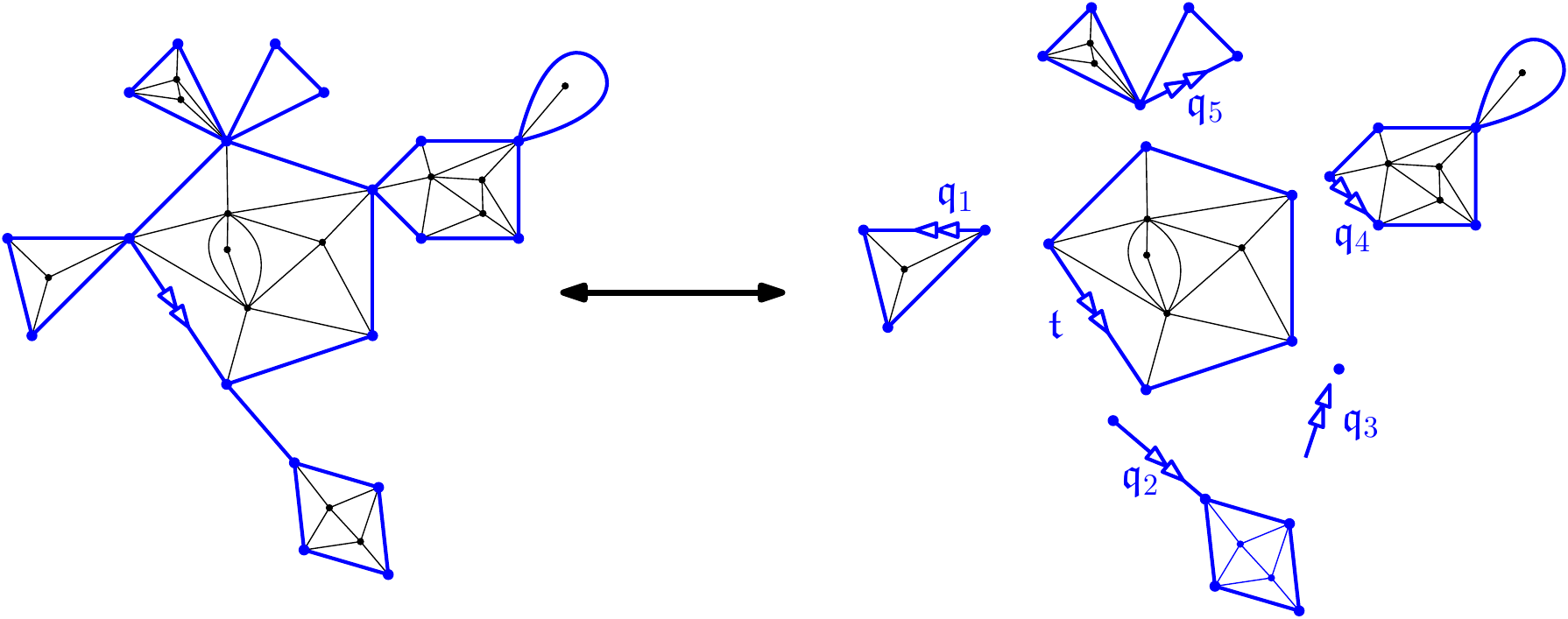}
\caption{\label{fig:SimpleVsNonSimple}Illustration of the decomposition of a triangulation with a boundary into a triangulation with a simple boundary and a collection of triangulations.}
\end{figure}
From the relation \eqref{eq:relQZ} and the algebraic equation obtained for $Z^+(\nu,t,y)$ in the preceding paper~\cite{IsingAMS}, we obtain directly the following algebraic equation for $Q^+(\nu,t,ty)$ (in terms of $Z^+_1$ and $Z^+_2$ defined in \eqref{eq:defZp+}), see also the companion Maple worksheet~\cite{Maple}:
\begin{align}\label{eq:eqAlgQplus}
0  &=\nu^2 t^6 y^5 \Big(Q^+(\nu,t,ty)\Big)^3 -t^3 \nu y^2 (\nu y^2-2 \nu y-y^2-\nu+3 y) \Big(Q^+(\nu,t,ty)\Big)^2 \notag\\
&-(2 \nu^2 t^4 y^3 Z^+_1+\nu^2 t^3 y^3-2 \nu t^3 y^3+\nu^2 y-\nu y^2-2 \nu^2+\nu y+y^2+2 \nu-2 y) Q^+(\nu,t,ty)\nonumber\\
&-\Big(2 \nu^2 t^2 (Z^+_1)^2+2 \nu^2 t^2 Z^+_2-\nu^2 tZ^+_1+\nu^2 t^3-2 \nu t^2 (Z^+_1)^2-2 \nu t^2 Z^+_2-\nu tZ^+_1+\nu+2 tZ^+_1-1\Big) y^2\notag\\
&-(\nu-1) \Big(2 \nu tZ^+_1-\nu-2\Big) y-\Big(2 (\nu-1)\Big) \nu.
\end{align}
By Proposition~\ref{prop:paramU}, $t\, Z^+_1$ and $t^2 \,Z^+_2$ can be rationally parametrized by $U$. Their expression in terms of $U$ is given in~\cite[Theorem~23]{BernardiBousquet} for $t \, Z^+_1$ (with a small change of variables) and \cite[Proposition~2.12]{IsingAMS} for $t^2 \, Z^+_2$ (see also \cite[Theorem~2.5]{IsingAMS} for the expression of $Z^+_1$ used here). By replacing $t^3$, $t \, Z^+_1$ and $t^2 \, Z^+_2$ by their expression in terms of $U$ in \eqref{eq:eqAlgQplus}, we can check directly that the parametrization given in Theorem~\ref{th:ratPar} is solution of \eqref{eq:eqAlgQplus}. See the Maple file~\cite{Maple} for detailed calculations.
\bigskip

It remains to check that the parametrization given in the theorem does indeed give a parametrization of $Q^+(\nu,t,ty)$ and not of another branch, which is also solution of~\eqref{eq:eqAlgQplus}. After replacing $t^3$, $tZ_1^+$ and $t^2Z_2^+$ by their expression in terms of $U$, \eqref{eq:eqAlgQplus} can be rewritten as: 
\[
    Q^+(\nu,t,ty)=1 +y \mathrm{Pol}_{\nu,U}(Q^+,y),
\]
where $\mathrm{Pol}_{\nu,U}$ is a polynomial in two variables whose coefficients are rational in $\nu$ and $U$, and such that $P(1,0)\neq 0$. From this expression, it is clear, that $Q^+(\nu,t,ty)$ is the unique solution of \eqref{eq:eqAlgQplus} that is a formal power series in $y$ (with rational coefficients in $\nu$ and $U$) with constant term $1$. 

On the other hand, \eqref{eq:defyV} can also be rewritten as: 
\[
    V = y\cdot\frac{Pol_{\nu,U}(V)}{1+V},
\]
where $Pol_{\nu,U}$ is a polynomial in one variable whose coefficients are rational in $\nu$ and $U$, and such that $P(0)\neq 0$. This ensures that there exists a unique $V \in\mathbb{Q}[\nu,U]\llbracket y \rrbracket$, with $V(0)=0$ and solution of~\eqref{eq:defyV}. It is then clear that $\hat Q(\nu,U,V(y))\in\mathbb{Q}[\nu,U]\llbracket y \rrbracket$ with constant term 1, which concludes the proof.
\end{proof}

\begin{rema}
 Of course, the main difficulty of Theorem \ref{th:ratPar} (which is not visible in the proof!) is guessing the rational parametrization. The general procedure is as follows. We start with the algebraic equation between $\nu$, $U$, $y$ and $Q^+(\nu, t(U) , t(U) y)$. If we specify a value for $\nu$ and for $U$, we can get (with Maple) a rational parametrization for $y$ and the series $Q^+$ specialized at these values of $\nu$ and $U$. Doing this for several values of $\nu$ and $U$, we get several rational parametrizations for several specialized instances of $(y,Q^+)$. 

For generic values of $\nu$ and $U$, the rational functions given by these parametrizations have the same degrees on both their numerator and denominator. We then perform a polynomial interpolation in $(\nu,U)$ of their coefficients. This yields a guess for an unspecialized parametrization of $(y, Q^+)$, and we verify that it is indeed solution of the equation. 

The guess we obtain is not the formula given in the theorem and is rather complicated. We simplify it further with M\"obius transforms of $V$ to get the parametrization of the theorem. See also \cite[Remark 9]{ChenTurunen2} for details on this guess and check procedure.
\end{rema}

\subsection{Critical points and singularities of the parametrization of \texorpdfstring{$Q^+$}{Q+}}\label{sub:studyV}

\subsubsection{Analytic properties of \texorpdfstring{$V$}{V} and \texorpdfstring{$Q^+$}{Q+}}
In this section, we will study the formal power series $y \mapsto V(\nu,U(\nu,t^3),y)$ of Theorem \ref{th:ratPar} as a function of a complex variable. We will see that it defines an analytic function with two real singularities defined below.

\begin{defi}\label{def:y+-}
Fix $\nu >0$ and $t \in (0,t_\nu]$. Set $V_-\big(\nu,U(\nu,t^3)\big)$ and $V_+\big(\nu,U(\nu,t^3)\big)$ to be respectively the largest negative root and the smallest positive root of the polynomial: 
\begin{equation} \label{eq:VcritU}
V^4 + 2 V^3 + 2 \frac{\Big(3 U(\nu,t^3) - 2\Big) \Big(3 U(\nu,t^3) (\nu + 1) - 2 \nu\Big)}{U(\nu,t^3) \Big(U(\nu,t^3) (\nu+1) - 2\Big) }V^2+ 2 V +1.
\end{equation}
Then, define:
\begin{equation}\label{eq:defy+-}
y_{+}(\nu,t) = \hat Y (\nu,U(\nu,t^3),V_+(\nu,U(\nu,t^3)))\quad \text{ and }\quad
y_{-}(\nu,t) = \hat Y (\nu,U(\nu,t^3),V_-(\nu,U(\nu,t^3))).
\end{equation}
\end{defi}
\begin{rema}
The existence of $V_-(\nu,U(\nu,t^3))$ and $V_+(\nu,U(\nu,t^3))$ is not clear at this point. It is established in the proof of the following proposition. 
\end{rema}

\begin{prop}\label{prop:inverseV}
For any $\nu>0$ and $t\in (0,t_\nu]$, we have $y_+(\nu,t) >0$ and $y_-(\nu,t) < 0$ and the formal power series $y \mapsto V  (\nu,U(\nu,t^3),y)$ of Theorem~\ref{th:ratPar} defines an analytic function on the domain
$\mathbb C \setminus \{ (- \infty , y_-(\nu,t)]\cup [y_+(\nu,t),+\infty) \}$ which is well-defined and singular at $y_-(\nu,t)$ and $y_+(\nu,t)$.

Moreover, $V(\nu,U(\nu,t^3),y)\neq -1$ for every $y \in \mathbb C \setminus \{ (- \infty , y_-(\nu,t)]\cup [y_+(\nu,t),+\infty)\}$.
\end{prop}

\begin{proof}
This property mainly boils down to a study of the rational function $V \mapsto \hat Y (\nu , U(\nu,t^3),V)$. Fix $\nu$ and $t$ as in the proposition. We want to study the critical points of this parametrization. Those are the poles and stationary points of the function $V\mapsto\hat Y (\nu , U(\nu,t^3),V)$, where it cannot be inverted.

We first look at the poles, they are the roots of the following polynomial of degree $3$:
\begin{equation} \label{eq:poley}
V^3+\dfrac{9(1+\nu)\cdot U^2(\nu,t^3)-2(3+10\nu)U(\nu,t^3)+8\nu}{U(\nu,t^3)\Big((1+\nu)\cdot U(\nu,t^3)-2\Big)}V^2-\dfrac{9(1+\nu)\cdot U(\nu,t^3)-2(2\nu+3)}{U(\nu,t^3)\Big((1+\nu)\cdot U(\nu,t^3)-2\Big)}V-1.
\end{equation}
If we compute the discriminant with respect to $V$ of this polynomial, we can see that it is negative for $0 \leq U(\nu,t^3) \leq U(\nu,t^3_\nu)$ (see the Maple file~\cite{Maple} for details). Therefore, $V\mapsto\hat Y (\nu , U(\nu,t^3),V)$ has a single real pole. We can check that the value of the polynomial \eqref{eq:poley} at $V=1$ is negative, meaning that the real pole of $V \mapsto \hat Y (\nu , U(\nu,t^3),V)$ is located after $V=1$.

Let us turn to the stationary points. They correspond to the zeros of the derivative $\partial_V \hat Y (\nu , U(\nu,t^3),V)$, and hence are given by the roots of the polynomial given in \eqref{eq:VcritU}.

In the entire range of our parameters $(\nu,U(\nu,t^3))$, we can determine that this polynomial is negative at $V=-1$, is equal to 1 at $V=0$, and is non positive at $V=1$. Hence, all its roots are real and each of the intervals $(-\infty ,-1)$, $(-1,0)$, $(0,1]$ and $[1,+\infty)$ contains one root. (If one of the roots is equal to 1, then it is a double root). This implies in particular $V_-(\nu,U)$ and $V_+(\nu,U)$ of Definition~\ref{def:y+-} are well-defined. Besides, $V_-(\nu,U)\in(-1,0)$ and $V_+(\nu,U)\in(0,1]$. Since $\partial_V \hat Y (\nu , U(\nu,t^3),V)\vert_{V=0}>0$ and $\hat Y (\nu , U(\nu,t^3),0)=0$, this implies that $y_+(\nu,t) >0$ and $y_-(\nu,t) < 0$.

The first part of the proposition follows by global inversion. Finally, to see that $V(\nu,U(\nu,t^3),y)$ is never equal to $-1$ in its domain of analyticity, we simply check that $\hat Y(\nu,U(\nu,t^3),-1) = 0$ is not in the domain where we inverse it.
\end{proof}

The analyticity properties of $y \mapsto V(\nu,U(\nu,t^3),y)$ can be transferred to the generating series $Q^+$:

\begin{prop} \label{prop:analyticQty}
Fix $\nu >0$ and $t \in (0,t_\nu]$. Then $Q^+(\nu,t,ty)$ seen as a function of $y$ has an analytic continuation on $\mathbb C \setminus \{ (- \infty , y_-(\nu,t)]\cup [y_+(\nu,t),+\infty) \}$. 
Moreover, $Q^+(\nu,t,ty)$ is well-defined and singular at $y_-(\nu,t)$ and $y_+(\nu,t)$.
\end{prop}
\begin{proof}
The first statement of the proposition follows from the rational parametrization given in Theorem~\ref{th:ratPar} and from Proposition~\ref{prop:inverseV}. Indeed, from Theorem~\ref{th:ratPar}, we know that 
\[\hat Q^+ \Big(\nu,U(\nu,t^3),V(\nu,U(\nu,t^3),y)\Big) = Q^+(\nu,t,ty)\]
as formal power series in $y$. In addition, from Proposition~\ref{prop:inverseV}, we deduce that
\[
y \mapsto \hat Q^+ \Big(\nu,U(\nu,t^3),V(\nu,U(\nu,t^3),y)\Big),
\]
and therefore $Q^+(\nu,t,ty)$, is analytic on~$\mathbb C \setminus \{ (- \infty , y_-(\nu,t)]\cup [y_+(\nu,t),+\infty) \}$.
\bigskip

It only remains to prove that $Q^+(\nu,t,ty)$ is singular at $y_-(\nu,t)$ and $y_+(\nu,t)$. There is indeed something to prove: even if $y\mapsto V(\nu,U(\nu,t^3),y)$ is singular for these two values, there could be some simplifications in \eqref{eq:ratParQ}, when substituting $V$ by its expansion around $y_{\pm}(\nu,t)$. We rely on an argument which appears in~\cite[Lemma 20]{ChenTurunen2}.

We first observe that the parametrization given in Theorem~\ref{th:ratPar} is \emph{proper}, meaning that $\Big(\hat Y(V),\hat Q^+(V)\Big)=(y,q)$ has a unique solution $V$ for all but finitely many $(y,q)$ which are solutions of~\eqref{eq:eqAlgQplus} (in which $Q^+(\nu,t,ty)$ is replaced by $q$). This fact follows directly from the characterization of properness via the degree of the rational parametrization given in~\cite[Theorem~4.21]{RationalAlgebraicCurves}. 

Assume that $\nu>0$ and $t\in (0,t_\nu]$ are fixed and suppose by contradiction that $Q^+(\nu,t,ty)$ is analytic at $y_+(\nu,t)$ (respectively at $y_-(\nu,t)$), then for $y$ in a neighborhood of $y_+(\nu,t)$ (respectively at $y_-(\nu,t)$), the equation $\hat Y (\nu, U(\nu,t^3), V)= y$ has several solutions. In turn, the equation $(\hat Y (\nu, U(\nu,t^3), V),\hat Q(\nu,U(\nu,t^3),V)) = (y,Q^+(\nu,U(\nu,t^3),y))$ has several solutions in the same neighborhood, contradicting the fact that the parametrization is proper.
\end{proof} 

We conclude this section by important bounds on the functions $y_+(\nu,t)$ and $y_-(\nu,t)$:

\begin{prop} \label{prop:yineq}
Fix $\nu >0$ and $t \in (0,t_\nu]$. We have $y_+(\nu,t)\geq 2$ and $y_{-}(\nu,t)<-y_+(\nu,t)$. This implies in particular that $y_+(\nu,t)$ is the unique dominant singularity of $y\mapsto V(\nu,U(\nu,t^3),y)$ and of $y \mapsto Q^+(\nu,t,ty)$. 
\end{prop}

\begin{proof}
The series $Q^+(\nu,t,ty)$ having positive coefficients and $y_+(\nu,t)$ being its radius of convergence in $y$ when $\nu$ and $t$ are fixed, it is clear that $y_+(\nu,t)$ is a decreasing function of $t$. Therefore, for $t \in(0,t_\nu]$ we have $y_+(\nu,t) \geq y_+(\nu,t_\nu)$. We will see in Lemma~\ref{lem:expanVy} that $y_+(\nu,t_\nu) \geq 2$ for every $\nu > 0$, giving the first part of the statement. 

In addition, the fact that $y_+(\nu,t)$ is the radius of convergence of $Q^+(\nu,t,ty)$ ensures that $y_-(\nu,t) \leq - y_+(\nu,t)$ since $y_-(\nu,t)$ is a singularity of the series. We cannot directly conclude that $y_-(\nu,t) \neq - y_+(\nu,t)$ and proving it by calculus, even if possible, seems tedious at best. We will see in Remark~\ref{rem:y+-} that this fact is an easy consequence of the link between $y_+(\nu,t)$, $y_-(\nu,t)$ and the universal generating series of a model of \emph{non-bipartite} Boltzmann planar maps given in \eqref{eq:cy+-}.
\end{proof}

\subsubsection{Singular expansions for \texorpdfstring{$V$}{V}}\label{sub:critV}

In this section, we study the singularities of the function $V(\nu,U(\nu,t^3),y)$. In the following, two types of expansion will be needed: first, we set $t=t_\nu$ and compute an expansion in $y$ at $y_+(\nu,t_\nu)$ and secondly, we fix $y$ and perform an expansion in $t$ at $t_\nu$.

\bigskip

\begin{lemm}\label{lem:expanVy}
Recall that $U_\nu:=U(\nu,t_\nu^3)$. The unique dominant singularity
\begin{equation} \label{eq:ynu}
y_\nu := y_+(\nu,t_\nu)
\end{equation}
of $y\mapsto V(\nu,U_\nu,y)$ satisfies
\begin{equation}\label{eq:ynuval}
y_\nu = 2 \quad\text{ for }\nu \leq \nu_c \qquad \text{and}\qquad 
y_\nu>2 \quad\text{ for }\nu > \nu_c.
\end{equation}
In addition, $y\mapsto V(\nu,U_\nu,y)$ admits the following singular expansion at $y_\nu$:
\begin{equation}\label{eq:devV1}
V(\nu,U_\nu,y) = 
\begin{cases}
1+\displaystyle\sum_{i\in\{\frac{1}{3},\frac{2}{3},1,\frac{4}{3}\}} \aleph^V_i(\nu)\left(1-\frac{y}{2}\right)^{i}+o\left(\left(1-\frac{y}{2}\right)^{4/3}\right)&\text{ for }\nu\leq \nu_c,\\
\displaystyle V_+(\nu,U_\nu)+\sum_{i\in\{\frac{1}{2},1,\frac{3}{2}\}} \aleph^V_i(\nu)\left(1-\frac{y}{y_+(\nu,t_\nu)}\right)^{i}+o\left(\left(1-\frac{y}{y_+(\nu,t_\nu)}\right)^{3/2}\right)&\text{ for }\nu> \nu_c,
\end{cases}
\end{equation}
where the coefficients $\aleph^V_i(\nu)$ are explicit functions of $\nu$, which do not vanish. 
\end{lemm}

\begin{proof}
To obtain this asymptotic behavior, we rely on the rational parametrization obtained in Theorem~\ref{th:ratPar} and on the quasi-automatic analysis of algebraic generating series described in \cite[Section VII.7]{FS}. However, because of the parameter $\nu$, this analysis must be carried out with care.
We study here the asymptotic behavior of $V(\nu,U,y)$ when $U$ is fixed and equal to $U_\nu$. To simplify the expressions, we replace $\nu$ by its expression in terms of $U_\nu$ as described in Proposition~\ref{prop:Unuval} and proceed by a case-by-case analysis, depending on the range of values for $\nu$. All calculations are done in the companion Maple file \cite{Maple}.

\paragraph{Case $\nu\leq\nu_c$.} Replacing $\nu$ by its expression~\eqref{eq:UnuSubCrit} in terms of $U_\nu$ in $\hat Y$, we obtain: 
\begin{equation}\label{eq:ysub}
\hat Y(\nu,U_\nu,V)= \frac{24(1-U_\nu )V(V+1)}{3U_\nu V^3-21U_\nu V^2-2V^3-3U_\nu V+18V^2-3U_\nu +6V+2}.
\end{equation}
We already checked in the proof of Proposition~\ref{prop:inverseV}, that for $U\in [0,U_{\nu_c}]$, the denominator of \eqref{eq:ysub} has a unique root -- say $V_0(\nu)$ --, and that $V_0(\nu)>1$.

We check that $\frac{\partial}{\partial V}\hat Y(\nu,U_\nu,V)=0$ if and only if $V\in \{-2-\sqrt{3},-2+\sqrt{3},1\}$. This implies that the set of critical values for $\hat Y(\nu,U_\nu,\cdot)$ is $\{-2-\sqrt{3},-2+\sqrt{3},1,V_0(\nu)\}$. For any $
\nu\leq \nu_c$, we have $\hat Y(\nu,U_\nu,1)=2$ and $\hat Y(\nu,U_\nu,-2+\sqrt{3})<-2$. So that, by singular inversion, $V(\nu,U_\nu,y)$ is analytic for $y\in \mathbb{C}\setminus \{(-\infty,\hat Y(\nu,U_\nu,-2+\sqrt{3})]\cup [2,\infty)\}$.

We can then compute the singular expansion of $V$ around $y=2$. There exist functions $\aleph^V_{i}(\nu)$ for $i\in \{\frac{1}{3},\frac{2}{3},1,\frac{4}{3}\}$, which do not vanish on $[0,\nu_c]$, such that: 
\begin{equation}\label{eq:devV1sub}
V(\nu,U_\nu,y) = 1+\sum_{i\in\{\frac{1}{3},\frac{2}{3},1,\frac{4}{3}\}} \aleph^V_i(\nu)\left(1-\frac{y}{2}\right)^{i}+o\left(\left(1-\frac{y}{2}\right)^{4/3}\right).
\end{equation}

\paragraph{Case $\nu > \nu_c$.} Things get slightly more complicated since $U_\nu$ cannot be expressed simply in terms of $\nu$ as in the subcritical case. Instead, we use the rational parametrization for $U_\nu$ and $\nu$ given in Proposition~\ref{prop:Unuval}.
Replacing $U_\nu$ and $\nu$ by their expression in terms of $K_\nu$ in~\eqref{eq:defyV}, we obtain the following expression for $\hat Y$:

\begin{equation}\label{eq:ysurK}
    \hat Y (\nu,U_\nu,V) = \frac{n(K_\nu)V(V+1)}{d_3(K_\nu)V^3+d_2(K_\nu)V^2+d_1(K_\nu)V+d_0(K_\nu)},
\end{equation}
where $n,d_0,d_1,d_2$ and $d_3$ are explicit polynomials.

From the proof of Proposition~\ref{prop:inverseV}, we know that the only pole of $\hat Y(\nu,U_\nu,V)$ occurs for $V>1$, so that we can restrict our attention to values of $V$ which cancels the derivative of \eqref{eq:ysurK} with respect to $V$. There are 4 values of $V$ which cancel $\frac{\partial}{\partial V}\hat Y(\nu,U_\nu,Y)$, they correspond to the roots (in $V$) of the two following polynomials: 
\[
    \begin{cases}
    P_1&= (K_\nu^2-3)V^2+4(1+K_\nu)^2V+K_\nu^2-3\\ 
    P_2&=(K_\nu^2-3)V^2+(-2K_\nu^2-8K_\nu-10)V+K_\nu^2-3.
    \end{cases}
\]
We can check by direct inspection that for $K\in[K_{\nu_c},K_\infty)$, these four roots are all real numbers. We denote $V_{1,1}(K_\nu) \leq V_{1,2}(K_\nu)$ (respectively $V_{2,1}(K_\nu) \leq V_{2,2}(K_\nu)$) the roots of $P_1$ (respectively of $P_2$). A basic analysis yields
\begin{equation} \label{eq:Vijcrit}
V_{2,1}(K_\nu) \leq V_{2,2}(K_\nu) < 0 < V_{1,1}(K_\nu) \leq 1 \leq V_{1,2}(K_\nu).
\end{equation}
So that, with the notation of Definition~\ref{def:y+-}, $V^-(\nu,U_\nu)=V_{2,2}(K_\nu)$ and $V^+(\nu,U_\nu)=V_{1,1}(K_\nu)$.
We can check that, for $\nu > \nu_c$, we have $y^+(\nu,t_\nu):=Y(\nu,U_\nu,V_{1,1}(K_\nu))>2$ and that $y^-(\nu,t_\nu):=Y(\nu,U_\nu,V_{2,2}(K_\nu)) < - Y(\nu,U_\nu,V_{1,1}(K_\nu))$.

We then compute the singular expansion of $V$ around $y_\nu=y_+(\nu,t_\nu)$. There exist functions $\aleph^V_i(\nu)$ for $i\in\{\frac{1}{2},1,\frac{3}{2}\}$, which do not vanish for $\nu \in (\nu_c,\infty)$, such that: 
\begin{equation}\label{eq:devV1sup}
V(\nu,U_\nu,y) =  V_+(\nu,U_\nu)+\sum_{i\in\{\frac{1}{2},1,\frac{3}{2}\}} \aleph^V_i(\nu)\left(1-\frac{y}{y_\nu}\right)^{i}+o\left(\left(1-\frac{y}{y_\nu}\right)^{3/2}\right),
\end{equation}
see the Maple file \cite{Maple} for their explicit expression in terms of $K_\nu$. 
\end{proof}

\bigskip

We now establish the expansion in $t$ of the series $V$ when $y$ is fixed:

\begin{lemm}\label{lem:expanVt} Fix $y$ in $\bar{D}(0,y_+(\nu,t_\nu))$, then $t^3\mapsto V(\nu,U(\nu,t^3),y)$ admits the following singular expansion at $t_\nu^3$:
\begin{equation}\label{eq:devSingV}
V(\nu,U(\nu,t^3),y) =
\begin{cases}
V_\nu(y)  +\displaystyle\sum_{i\in\{\frac{1}{2},1,\frac{3}{2}\}}\beth_i^V (\nu,y) \cdot \left(1-\frac{t^3}{t_\nu^3}\right)^{i}+o\left(\left(1-\frac{t^3}{t_\nu^3}\right)^{3/2}\right) & \text{ for $\nu \neq \nu_c$, }\\
V_{\nu_c}(y) +\displaystyle \sum_{i\in\{\frac{1}{3},\frac{2}{3},1,\frac{4}{3}\}}\beth_i^V(\nu_c,y) \cdot \left(1-\frac{t^3}{t_{\nu_c}^3}\right)^{i}+o\left(\left(1-\frac{t^3}{t_{\nu_c}^3}\right)^{4/3}\right)
& \text{ for $\nu = \nu_c$},
\end{cases}
\end{equation}
where $V_\nu(y):=V(\nu,U(\nu,t_\nu^3),y)$, and $\beth_i^V(\nu,y)$ are explicit functions of $\nu$ and $y$, which do not vanish.
\end{lemm}
\begin{proof}
Fix $\nu > 0 $ and  $y \in \bar{D}(0,y_+(\nu,t_\nu))$. In the algebraic equation  \eqref{eq:defyV} between $y$, $U$ and $V$, we replace $U$ by its development obtained in Lemma~\ref{lem:propU} to get the singular behavior of $V$. We have to proceed by a case-by-case analysis since the singular expansions for $U$ are different for $\nu < \nu_c$, $\nu = \nu_c$ and $\nu > \nu_c$. Again all computations are performed in the Maple worksheet~\cite{Maple}.

\paragraph{Case $\nu < \nu_c$.} Recall \eqref{eq:UnuSubCrit}, which expresses $\nu$ as a function of $U_\nu = U(\nu,t_\nu^3)$. The singular behavior of $U$ for $t^3$ around $t_\nu^3$ obtained in Lemma~\ref{lem:propU} is of the form
\begin{equation*}
U(\nu,t^3) = U_\nu + \sum_{i\in\{\frac{1}{2},1,\frac{3}{2}\}}\beth_i^U (\nu) \cdot \left(1-\frac{t^3}{t_\nu^3}\right)^{i}+o\left(\left(1-\frac{t^3}{t_\nu^3}\right)^{3/2}\right),
\end{equation*}
where $\beth_{1/2}^U, \beth_1^U$ and $\beth_{3/2}^U$ are explicit functions of $U_\nu$ which do not cancel in this range of values for $\nu$. 

Equation \eqref{eq:defyV} gives an algebraic equation between $\nu$, $U=U(\nu,t^3)$, $V=V(\nu,U,y)$ and $y$. Since $y$ is fixed, we could plug the expansion of $U$ in this equation and obtain the development in $t^3$ of $V$ with $y$ as a parameter. However, expressions turn out to be nicer if we replace $y$ by its value in terms of $V_\nu(y) := V(\nu,U_\nu,y)$.
The identity
\[
y = \hat Y(\nu,U_\nu,V_\nu) = \hat Y\left(\nu,U(\nu,t^3),V\big(\nu,U(\nu,t^3),y\big)\right),
\]
which is valid for any fixed $t\in [0,t_\nu^3]$, gives an algebraic equation between $\nu$, $U_\nu$, $V_\nu$, $U(\nu,t^3)$ and $V\big(\nu,U(\nu,t^3),y\big)$. We can again eliminate $\nu$ from this equation by replacing it by its value in terms of $U_\nu$. We then replace $U(\nu,t^3)$ by its singular expansion to obtain the following asymptotic behavior for $V\big(\nu,U(\nu,t^3),y\big)$: 
\begin{equation}\label{eq:devSingVsub}
V\big(\nu,U(\nu,t^3),y\big) = V_\nu + \sum_{i\in\{\frac{1}{2},1,\frac{3}{2}\}}\beth_i^V (\nu,y) \cdot \left(1-\frac{t^3}{t_\nu^3}\right)^{i}+o\left(\left(1-\frac{t^3}{t_\nu^3}\right)^{3/2}\right),
\end{equation}
where $\beth_i^V(\nu,y)$ are non-vanishing explicit functions of $\nu$ and $y$. More precisely, the functions $\beth_i^V(\nu,y)$ we obtain are explicit functions of $U_\nu$ and $V_\nu(y)$. See the Maple companion file \cite{Maple} for details.

\paragraph{Case $\nu = \nu_c$.} The strategy is similar, but since $\nu$ is fixed, computations are slightly less heavy. We refer again to the Maple worksheet.
Notice, that the expansion in $t^3$ of $U(\nu_c,t^3)$ obtained in Lemma~\ref{lem:propU} is in $\Big(1-(\frac{t}{t_{\nu_c}})^3\Big)^{1/3}$ rather than in $\Big(1-(\frac{t}{t_{\nu_c}})^3\Big)^{1/2}$ as above. By performing, the exact same line of arguments, we then obtain the desired singular development for $V(\nu_c,t,ty)$.

\paragraph{Case $\nu > \nu_c$.} Computations are slightly more elaborate in this case since $\nu$ cannot be expressed as a rational function of $U_\nu$. As in the proof of Proposition~\ref{lemm:weightsclusters}, we rely on the rational parametrization for $U_\nu$ and $\nu$ by $K_\nu$ given in \eqref{eq:defKU} and \eqref{eq:defKnu}. Apart from this, the strategy of proof is totally similar and we refer to the Maple worksheet for the details of the computation. In particular the asymptotic development for $U(\nu,t^3)$ around $t_{\nu}^3$ of Lemma~\ref{lem:propU} has the form
\begin{equation*}
U(\nu,t^3) = U_\nu + \sum_{i\in\{\frac{1}{2},1,\frac{3}{2}\}}\beth_i^U (\nu) \cdot \left(1-\frac{t^3}{t_\nu^3}\right)^{i}+o\left(\left(1-\frac{t^3}{t_\nu^3}\right)^{3/2}\right),
\end{equation*}
where $U_\nu$, $\beth_{1/2}^U, \beth_1^U$ and $\beth_{3/2}^U$ are explicit functions of $K_\nu$, defined for $\nu\in (\nu_c,\infty)$ and which do not cancel on this interval, leading to the development of $V(\nu,U(\nu,t),y)$. 
\end{proof}

\subsection{Singular expansions for the generating series of triangulations with monochromatic boundary}

In this section, we study the singularities of $Q^+(\nu,t,ty)$. As was done with $V$, we first consider the case when $t$ is fixed at $t_\nu$ and study the singularities in $y$ of $Q^+(\nu,t_\nu,t_\nu y)$. Then, we fix $y$ and study the singularities in $t$ of $Q^+(\nu,t,t y)$.

\begin{prop}\label{lemm:weightsclusters}
For $\nu >0$, the series $Q^+ (\nu,t_\nu,t_\nu y)$ has radius of convergence $y_\nu = y_+(\nu,t_\nu)$ and a single dominant singularity at $y_\nu$, where it admits the following expansion.

\begin{itemize}
\item For $\nu<\nu_c$: 
\begin{equation}
Q^+(\nu,t_\nu,t_\nu y) = Q^+(\nu,t_\nu ,t_\nu y_\nu) + \aleph^{Q^+}(\nu) \left( 1- \frac{y}{y_\nu} \right)^{\alphaa(\nu)-1} + o \left( \left( 1- \frac{y}{y_\nu} \right)^{\alphaa(\nu)-1} \right);
\end{equation}

\item For $\nu\geq\nu_c$: 
\begin{multline}
Q^+(\nu,t_\nu,t_\nu y) = Q^+(\nu,t_\nu ,t_\nu y_\nu) + \aleph_1^{Q^+}(\nu) \left( 1- \frac{y}{y_\nu} \right)\\
+  \aleph^{Q^+}(\nu) \left( 1- \frac{y}{y_\nu} \right)^{\alphaa(\nu)-1} + o \left( \left( 1- \frac{y}{y_\nu} \right)^{\alphaa(\nu)-1} \right);
\end{multline}
\end{itemize}
where $\aleph_1^{Q^+}(\nu)$ and $\aleph^{Q^+}(\nu)$ are explicit non-vanishing functions of $\nu$ and 
\begin{equation}\label{eq:alphanu}
\alphaa(\nu) =
\begin{cases}
5/3 & \text{for $\nu < \nu_c$,}\\
7/3 & \text{for $\nu = \nu_c$,}\\
5/2 & \text{for $\nu > \nu_c$.}
\end{cases}
\end{equation}
\end{prop}
\begin{proof}
The fact that $y_\nu$ is the unique dominant singularity of $Q^+ (\nu,t_\nu,t_\nu y)$ follows from Proposition~\ref{prop:yineq}. It remains to establish the asymptotic behavior at $y_\nu$. We rely on the rational parametrization obtained in Theorem~\ref{th:ratPar}. From \eqref{eq:paramQ}, we see that as long as $V\neq -1$, the asymptotic behavior of $Q^+ (\nu,t_\nu,t_\nu y)$ is driven by the asymptotic behavior of $V$ as a function of $y$. It hence suffices to plug in \eqref{eq:paramQ} the singular expansion of $y\mapsto V\Big(\nu,U(\nu,t_\nu^3),y\Big)$ obtained in Lemma~\ref{lem:expanVy} to get the desired results. All computations are available in the Maple companion file~\cite{Maple}.
\end{proof}

A direct application of the transfer theorem (see Theorem~\ref{th:transfer} in Appendix~\ref{sec:Hadamard}) then implies:
\begin{coro}\label{cor:Q}
For any $\nu>0$, we have:
\begin{equation}\label{eq:coeffQl}
Q^+_\ell(\nu,t_\nu) \stackrel[\ell\rightarrow \infty]{}{\sim} \frac{\aleph^{Q^+}(\nu)}{\Gamma\Big(1-\alphaa(\nu)\Big)}\cdot\Big(t_\nu y_\nu\Big)^{-\ell} \ell^{-\alphaa(\nu)}.
\end{equation}
\end{coro}

\bigskip

We now turn to the singular behavior in $t^3$ of $Q_+(\nu,t,ty)$ when $y$ is fixed.

\begin{prop} \label{prop:serQty}
Fix $\nu > 0$ and $y \in (0, y_\nu)$, then $Q_+(\nu,t,ty)$ seen as a series in $t^3$ has radius of convergence $t_\nu^3$ where it has the following asymptotic expansion:
\begin{multline}
Q^+(\nu,t,ty)  = Q^+(\nu,t_\nu ,t_\nu y) + \beth_1^{Q^+}(\nu,y) \left( 1- \Big(\frac{t}{t_\nu}\Big)^{\!3} \right)\\+ \beth^{Q^+}(\nu,y) \left( 1- \Big(\frac{t}{t_\nu}\Big)^{\!3} \right)^{\gamma(\nu)-1} + o \left( \left( 1- \Big(\frac{t}{t_\nu}\Big)^{\!3} \right)^{\gamma(\nu)-1} \right),
\end{multline}
where we recall from Proposition \ref{prop:asymptoZp+} that: 
\[\gamma(\nu) = 
\begin{cases}
5/2 & \text{for $\nu \neq \nu_c$,}\\
7/3 & \text{for $\nu = \nu_c$.}
\end{cases}
\]
Furthermore, for every fixed $\nu >0$, the two coefficients $\beth_1^{Q^+}$ and $\beth^{Q^+}$ are power series in $y$ with radius of convergence $y_\nu$ and their expression is given by an explicit rational function in $V(\nu,U(\nu,t_\nu),y)$.
\end{prop}

\begin{rema}
For every $p > 0$, the series $t^pQ^+_p(\nu,t)$ has a unique dominant singularity at $t_\nu^3$, so that it is reasonable to think that $t_\nu^3$ is also the unique dominant singularity of $Q_+(\nu,t,ty)$. This could be rigorously proved by a similar approach as in the proof of Proposition~\ref{prop:asymptQk}. However, we did not prove this statement as it is not needed for our work.
\end{rema}

\begin{proof}
For any fixed $\nu>0$ and any fixed $y\in (0,y_\nu)$, the function $Q_+(\nu,t,ty)$ seen as a series in $t^3$ has non negative coefficients. It is also clear from its definition that its radius of convergence cannot be larger that $t_\nu^3$. From Proposition \ref{lemm:weightsclusters}, we deduce that its radius of convergence is indeed $t_\nu^3$. By Pringsheim's theorem, $Q_+(\nu,t,ty)$ is then singular at $t_\nu^3$. It hence remains to prove that the asymptotic behavior at this singularity is the one given in the proposition. 

To get the asymptotic expansion of $Q^+(\nu,t,ty)$ around $t_\nu^3$, we once again rely on the rational parametrization obtained in Theorem~\ref{th:ratPar}. In \eqref{eq:paramQ}, we substitute $U$ and $V$ by their singular expansion obtained respectively in Lemma~\ref{lem:propU} and Lemma~\ref{lem:expanVt}
to get the desired result. 

All computations are performed in \cite{Maple}.
\end{proof}

\bigskip

We will need later the asymptotic development of $\beth^{Q^+}$ in the sub-critical and critical regime:

\begin{prop} \label{prop:asymptalephQ+}
Fix $\nu \in (0,\nu_c]$. Then $\beth^{Q^+}(\nu,y)$, seen as a power series in $y$, has a unique dominant singularity at $y_\nu = 2$, where it has the following singular expansion:
\[
\beth^{Q^+}(\nu,y) = \sum_{i \in \left\{\frac{4}{3},\frac{2}{3},\frac{1}{3}\right\}}\frac{\aleph_i^{\beth^{Q^+}} (\nu)}{(2-y)^{i}} + \mathcal O (1),
\]
with non vanishing coefficients $\aleph_i^{\beth^{Q^+}} (\nu)$ that are explicit functions of $U_\nu$.
\end{prop}
\begin{proof}
This is a direct consequence of the expression of $\beth^{Q^+}(\nu,y)$ as a rational function of $V(\nu,U_\nu,y)$. The development is obtained as usual by replacing $V$ by its singular development obtained in Lemma~\ref{lem:expanVy}. See the Maple companion file \cite{Maple} for details.
\end{proof}

\subsection{Singularities in \texorpdfstring{$y$}{y} of \texorpdfstring{$Q^+(\nu,t,ty)$}{Q+(nu,t,ty)}}

In this section, we study the value of the singularities (in $y$) of $Q^+(\nu,t,ty)$. Recall the definition of $y_-(\nu,t)$ and $y_+(\nu,t)$ given in Definition~\ref{def:y+-}. The next statement studies the singular behavior in $t$ of these functions:

\begin{prop} \label{prop:asym_t_sign_y}
Fix $\nu >0$. The functions $y_+(\nu,t)$ and $y_-(\nu,t)$ of $Q_+ (\nu,t,ty)$ defined in Definition~\ref{def:y+-} are power series in $t^3$ and have a unique dominant singularity at $t_\nu^3$. In addition, they have the following asymptotic expansions:
\begin{multline}
y_-(\nu,t) = y_-(\nu,t_\nu) +\\ \beth^-_1(\nu) \left( 1- \Big(\frac{t}{t_\nu}\Big)^{\!3} \right) + \beth^-(\nu) \left( 1- \Big(\frac{t}{t_\nu}\Big)^{\!3} \right)^{\gamma(\nu)-1} + o \left( \left( 1- \Big(\frac{t}{t_\nu}\Big)^{\!3} \right)^{\gamma(\nu)-1} \right),
\end{multline}
and
\begin{equation}
y_+(\nu,t) = y_\nu + \beth^+(\nu) \left( 1- \Big(\frac{t}{t_\nu}\Big)^{\!3} \right)^{\gamma_+(\nu)-1} + o \left( \left( 1- \Big(\frac{t}{t_\nu}\Big)^{\!3} \right)^{\gamma_+(\nu)-1} \right),
\end{equation}
where $\beth^-_1(\nu)$, $\beth^-(\nu)$ and $\beth^+(\nu)$ are explicit non-vanishing functions of $\nu$. Furthermore, $\gamma(\nu) = 5/2$ for $\nu \neq \nu_c$ or $7/3$ for $\nu = \nu_c$ as defined in Proposition~\ref{prop:asymptoZp+}, and
\[
\gamma_+(\nu) =
\begin{cases}
7/4 & \text{for $\nu < \nu_c$;}\\
3/2 & \text{for $\nu =\nu_c$;}\\
3/2 & \text{for $\nu > \nu_c$.}
\end{cases}
\]
\end{prop}

\begin{proof}
Recall from Definition~\ref{def:y+-} that $y_\pm(\nu,t)$ are given respectively by
\[
\hat Y(\nu,U(\nu,t^3),V_\pm(\nu,U(\nu,t^3))),
\]
where $V_\pm(\nu,U(\nu,t^3)))$ denote the largest negative root and the smallest positive root of the polynomial~\eqref{eq:VcritU}. This already proves that $y_\pm(\nu,t)$ are formal power series in $U(\nu,t^3)$, and therefore in $t^3$. To prove that they both have the same singularities as $U(\nu,t^3)$, and therefore a unique singularity at $t_\nu^3$, we look at the discriminant in $V$ of the polynomial \eqref{eq:VcritU}. This discriminant factorizes as:
\[
\frac{1024 (-1+2 U) (U \nu+U-\nu) (3 U^{2} \nu+3 U^{2}-3 U \nu-3 U+\nu) (15 U^{2} \nu+15 U^{2}-24 U \nu-6 U+8 \nu)^{2}}{(U \nu+U-2)^{4} U^{4}}.
\]
The factors of degree 1 in $U$ are clearly not 0 for $|U| \leq U(\nu,t_\nu^3)$. The first factor of degree 2 in $U$ is the polynomial giving $U(\nu,t_\nu^3)$ for $\nu \leq \nu_c$ and can only be $0$ at this value for $|U|\leq U(\nu,t_\nu^3)$. Finally we can check that the last factor can only be $0$ for values of $U$ with a modulus larger than $U(\nu,t_\nu^3)$ (see the Maple companion file~\cite{Maple} for details). Hence, $y_\pm(\nu,t)$ have unique dominant singularities at $t_\nu^3$.

\bigskip

It remains to prove that the asymptotic expansions at $t_\nu^3$ have the form given in the proposition. The calculations share similarities with the proof of Proposition \ref{prop:serQty}. Informally, in the equation \eqref{eq:VcritU} verified by $V_\pm$, we can replace $U(\nu,t^3)$ by its singular behavior around $t_{\nu}^3$ obtained in Lemma~\ref{lem:propU}. It then suffices to identify the right branch for $V$, to get a singular expansion for $V_\pm(\nu,U(\nu,t^3))$. Plugging back the expansions for $U(\nu,t^3)$ and $V_\pm(\nu,U(\nu,t^3))$ in $\hat Y(\nu,U(\nu,t^3),V_\pm(\nu,U(\nu,t^3)))$ then yields the desired result.

We now give some details about the computations for the different values of $\nu$. All the computations are performed in the Maple companion file~\cite{Maple}.

\paragraph{Case $\nu<\nu_c$.} From Proposition~\ref{lemm:weightsclusters}, we know that, for $\nu\leq \nu_c$, the radius of convergence in $y$ of $Q(t_\nu,t_\nu y)$ is equal to 2, which corresponds to $V_+(\nu,U(t_\nu^3))=1$. 

We can indeed check that, after plugging the expansion of $U(\nu,t^3)$ in~\eqref{eq:VcritU}, one (and only one) of the solutions has a constant term equal to 1, and we can hence compute its expansion around 1. We obtain: 
\begin{equation}
V_+(\nu,U(\nu,t^3))
= 1 + \sum_{i\in\{\frac{1}{4},\frac{1}{2},\frac{3}{4}\}}\beth_i^{V_+} (\nu) \cdot \left(1-\frac{t^3}{t_\nu^3}\right)^{i}+o\left(\left(1-\frac{t^3}{t_\nu^3}\right)^{3/4}\right),
\end{equation}
where $\beth_{\frac{1}{4}}^{V_+}$, $\beth_{\frac{1}{2}}^{V_+}$ and $\beth_{\frac{3}{4}}^{V_+}$ are explicit functions of $\nu$. Replacing $U$ and $V_+$ by their expansion in $\hat Y(\nu,U(\nu,t^3),V_+(\nu,U(\nu,t^3)))$ gives the desired expansion for $y_+(\nu,t)$.

We now move to the singular expansion of the negative singularity $y_-(\nu,t)$. From the proof of Proposition~\ref{lemm:weightsclusters}, we know that $y_-(\nu,t_\nu)$ corresponds to $V_-(\nu,U(\nu,t_\nu^3))=-2+\sqrt{3}$. We can hence perform the same type of computation around $V=-2+\sqrt{3}$ rather than $V=1$ to get the desired result.

\paragraph{Case $\nu=\nu_c$.} In this case, the value of $y_+(\nu_c,t_{\nu_c})$ also corresponds to $V_+(\nu_c,U(\nu_c,t^3_{\nu_c}))=1$. 
This time, when we compute the expansion of $V$, we obtain a different singular behavior, namely:
\begin{equation}
V_+(\nu_c,U(\nu_c,t^3)) 
= 1 + \sum_{i\in\{\frac{1}{6},\frac{1}{3},\frac{1}{2}\}}\beth_i^{V_+} (\nu) \cdot \left(1-\frac{t^3}{t_\nu^3}\right)^{i}+o\left(\left(1-\frac{t^3}{t_\nu^3}\right)^{1/2}\right),
\end{equation}
where $\beth_{\frac{1}{6}}^{V_+}$, $\beth_{\frac{1}{3}}^{V_+}$ and $\beth_{\frac{1}{2}}^{V_+}$ are explicit functions of $\nu$. Replacing $U$ and $V_+$ by their expansion in $\hat Y(\nu_c,U(\nu_c,t^3),V_+(\nu_c,U(\nu_c,t^3)))$ gives the desired expansion for $y_+(\nu_c,t)$.

The expansion around $V_-(\nu_c,U(\nu_c,t_{\nu_c}^3))=-2+\sqrt{3}$ does not yield any additional difficulties and gives the desired singular expansion. 

\paragraph{Case $\nu>\nu_c$.} As in the proofs of Proposition \ref{lemm:weightsclusters} and of Proposition \ref{prop:serQty}, we parametrize $\nu$ and $U(\nu,t_\nu^3)$ by $K_\nu$ with \eqref{eq:defKnu} and \eqref{eq:defKU}. The values of $V_\pm(\nu,U(\nu,t_\nu^3))$ can be calculated explicitly from \eqref{eq:Vijcrit} as roots of polynomials of degree 2. We obtain from this the two following asymptotic expansions
\begin{align*}
V_+(\nu,U(\nu,t^3))
&= V_+(\nu,U_\nu) + \sum_{i\in\{\frac{1}{2},1,\frac{3}{2}\}}\beth_i^{V_+} (\nu) \cdot \left(1-\frac{t^3}{t_\nu^3}\right)^{i}+o\left(\left(1-\frac{t^3}{t_\nu^3}\right)^{3/2}\right)\\
V_-(\nu,U(\nu,t^3))
&= V_-(\nu,U_\nu) + \sum_{i\in\{\frac{1}{2},1,\frac{3}{2}\}}\beth_i^{V_-} (\nu) \cdot \left(1-\frac{t^3}{t_\nu^3}\right)^{i}+o\left(\left(1-\frac{t^3}{t_\nu^3}\right)^{3/2}\right),
\end{align*}
where $\beth_{\frac{1}{2}}^{V_+}$, $\beth_{\frac{1}{2}}^{V_-}$, $\beth_{1}^{V_+}$, $\beth_{1}^{V_-}$, $\beth_{\frac{3}{2}}^{V_+}$ and $\beth_{\frac{3}{2}}^{V_-}$ are explicit functions of $\nu$.

Replacing $U$ and $V_\pm$ by their expansions in $\hat Y(\nu,U(\nu,t^3),V_\pm(\nu,U(\nu,t^3)))$ gives the desired expansions for $y_\pm(\nu,t)$ without additional difficulties, after noticing that the term of order $\left( 1- (t/t_\nu)^3 \right)^{1/2}$ vanishes in $y_-(\nu,t)$ but not in $y_+(\nu,t)$. 
\end{proof}

\subsection{Monochromatic simple boundary} \label{sec:monochromatic}

Recall that $Z^+(\nu,t,x)$ denotes the generating series of Ising-weighted  triangulations with a \emph{simple} monochromatic boundary. It was proved in \cite{IsingAMS} that for every $\nu>0$, the series $Z^+(\nu,t_\nu,x)$ is well-defined as $t_\nu$ is the radius of convergence of $Z^+_p(t)$ for every $p>0$ and $Z^+_p(t_\nu)<\infty$. We start this section with a rational parametrization for $Z^+$:

\begin{theo} \label{th:ratParZ}
Recall the definition of $U$ given in Proposition~\ref{prop:paramU}. We define $W\equiv W(\nu,U,x)$ as the unique formal power series in $\mathbb{Q}[\nu,U]\llbracket x \rrbracket$ having constant term $0$ and satisfying the following equation:
\begin{align}\label{eq:defxW}
x = \hat X (\nu,U,W),
\end{align}
where $\hat X (\nu,U,W)$ is the following rational fraction:
\[
\begin{split}
\hat X (\nu,U,W) &= 
\frac{8 \nu  \left(1-2 U\right) \left(1-\nu \right) }{ 8 \left(\nu +1\right)^{2} U^{3}-\left(11 \nu +13\right) \left(\nu +1\right) U^{2}+2 \left(\nu +3\right) \left(2 \nu +1\right) U-4 \nu }
\frac{W}{\left(W+1\right)^{2}}
\\
& \qquad \left(W^{2}+\frac{2 \left(5 \left(\nu +1\right) U^{2}-2 \left(3 \nu +2\right) U+2 \nu \right) }{U \left(U \left(\nu +1\right)-2\right)} W \right.\\
& \qquad \qquad \left.-\frac{8 \left(\nu +1\right)^{2} U^{3}-\left(11 \nu +13\right) \left(\nu +1\right) U^{2}+2 \left(\nu +3\right) \left(2 \nu +1\right) U-4 \nu}{U \left(U \left(\nu +1\right)-2\right) \left(1-\nu \right)}\right).
\end{split}
\]
The series $Z^+ (\nu,t,tx)$ is algebraic and admits the following rational parametrization in terms of $U$ and $W$:
\begin{align}\label{eq:paramZ}
Z^+(\nu,t,tx)=\hat Q^+ \Big( \nu,U(\nu,t^3),W \big(\nu,U(\nu,t^3),x \big) \Big) - 1,
\end{align}
where $\hat Q^+(\nu,U,W)$ is the rational parametrization for $Q^+$ given in Theorem \ref{th:ratPar}.
\end{theo}

\begin{proof}
As mentioned earlier, the algebraicity of $Z^+$ is established in \cite{IsingAMS}, see in particular Theorem 2.11 in this reference. We only have to check that the expressions given in the statement are solution of the algebraic equation satisfied by $Z^+$ and the theorem follows by uniqueness of the solution in a similar fashion as Theorem \ref{th:ratPar}. Computations are performed in the Maple file \cite{Maple}.
\end{proof}

\bigskip

With the rational parametrization of $Z^+$ the study of the singular behavior of $Z^+(\nu,t_\nu,t_\nu \,x)$ is very similar to the study of $Q^+(\nu,t_\nu,t_\nu y)$ performed in Proposition~\ref{lemm:weightsclusters}. We gather in the next statement the information we need on $Z^+$.

\begin{lemm} \label{lem:expZplusxc}
For $\nu >0$, the series $Z^+ (\nu,t_\nu,t_\nu x)$ has radius of convergence $x_\nu > 0$ and a single dominant singularity at $x_\nu$, where it admits the following expansion:
\begin{multline}
Z^+(\nu,t_\nu,t_\nu x) = Z^+(\nu,t_\nu ,t_\nu x_\nu) + \aleph_1^{Z^+}(\nu) \left( 1- \frac{x}{x_\nu} \right)\\
+  \aleph^{Z^+}(\nu) \left( 1- \frac{x}{x_\nu} \right)^{\gamma(\nu)-1} + o \left( \left( 1- \frac{x}{x_\nu} \right)^{\gamma(\nu)-1} \right),
\end{multline}
where $\aleph_1^{Z^+}(\nu)$ and $\aleph^{Z^+}(\nu)$ are explicit functions of $\nu$, which do not vanish, and $\gamma$ is the function of $\nu$, which already appeared in the asymptotic developments of $Q^+$ (Proposition~\ref{prop:serQty}) and is given by: 
\begin{equation*}
\gamma(\nu) =
\begin{cases}
5/2 & \text{for $\nu \neq \nu_c$,}\\
7/3 & \text{for $\nu = \nu_c$.}
\end{cases}
\end{equation*}
Furthermore, we have
\[
\begin{cases}
\dfrac{-\aleph_1^{Z^+}(\nu)}{1+ Z^+(\nu,t_\nu ,t_\nu x_\nu)} =1 & \text{if $\nu<\nu_c$,} \\
\dfrac{-\aleph_1^{Z^+}(\nu)}{1+ Z^+(\nu,t_\nu ,t_\nu x_\nu)} \leq \dfrac{\sqrt 7}{1 + \sqrt 7} & \text{if $\nu \geq\nu_c$,}
\end{cases}
\]
with equality in the last case if and only if $\nu = \nu_c$.
\end{lemm}

\begin{proof}
The proof is very similar to what was done with $Q^+$ and we only point out the differences. To study the singularities of $x \mapsto Z^+(\nu,t_\nu, t_\nu \,x)$, we can first study the singularities of $x \mapsto W(\nu,U_\nu,x)$, where we recall that $U_\nu = U(\nu,t_\nu^3)$. Again we treat separately the case $\nu \leq \nu_c$ and $\nu > \nu_c$. Detailed computations are available in the Maple worksheet \cite{Maple}. 

\paragraph{Case $\nu \leq \nu_c$.} In this regime, $\hat X(\nu,U_\nu,W)$ has a pole at $W=-1$ and its stationary points are given by the roots of the polynomial
\begin{equation} \label{eq:wcritsubc}
\left(W-1\right) \left((6 U^{2} - 6U +1) W^{2}+4 (6 U^{2} - 6 U + 1) W+6 U^{2}-10 U+3\right).
\end{equation}
We check that for $U \in (0,U_c]$, the roots of the polynomial of degree 2 are never in $(-1,1)$ and therefore $x \mapsto W(\nu,U(\nu,t_\nu^3),x)$ is bijective from $(-\infty,x_\nu]$ onto $(-1,1]$, where
\[{}
x_\nu = \hat X(\nu,U_\nu,1) >0
\]
is the radius of convergence of $x \mapsto Z^+(\nu,U(\nu,t_\nu^3),x)$ by composition.

Before computing the expansion of $ Z^+(\nu,U(\nu,t_\nu^3),x)$ at $x_\nu$, we still have to check that it has no other dominant singularity. The only candidates are the images $x_1$ and $x_2$ of the two roots of the polynomial of degree two \eqref{eq:wcritsubc} by the function $\hat X$. We compute explicitly the asymptotic expansions of $Z^+$ at these two values and see that if $|x_i| = x_\nu$, the development is non singular (we identify the right branch using the value of $Z^+(\nu,t_\nu,t_\nu x_i)$, which has to be of modulus smaller than $Z^+(\nu,t_\nu,t_\nu x_\nu$)).

Now that we know that $x \mapsto Z^+(\nu,U(\nu,t_\nu^3),x)$ has a unique dominant singularity at $x_\nu$, we compute the expansion of $x \mapsto W(\nu,U(\nu,t_\nu^3),x)$ at this value by singular inversion and plug it in $\hat Z^+$ to obtain the expansions given in the statement. The coefficients $\aleph_1^{Z^+}(\nu)$, $\aleph^{Z^+}(\nu)$ and $Z^+(\nu,U(\nu,t_\nu^3),x_\nu)$ are all explicit functions of $U_\nu$ and establishing the last statement is an easy calculation.

\paragraph{Case $\nu > \nu_c$.} As done several times before, we work in this regime by replacing $\nu$ and $U_\nu$ by their value in terms of $K_\nu$ given in equations \eqref{eq:defKnu} and \eqref{eq:defKU}. The stationary points of $\hat X( \nu, U_\nu, W)$ are given by the roots of the polynomial
\[
\left(K_\nu^{2} W-K_\nu^{2}-8 K_\nu-3 W-13\right) \left(K_\nu^{2} W^{2}+4 K_\nu^{2} W+K_\nu^{2}+8 K_\nu W-3 W^{2}+4 W-3\right).
\]
We check that the root of the factor of degree $1$ is real and smaller that $-1$, and that the factor of degree two has two real roots, one in $(0,1)$ that we will denote by $W_\nu$, and one in $(1,\infty)$. By global inversion, we see that $x \mapsto W(\nu,U(\nu,t_\nu^3),x)$ is analytic on $\mathbb C \setminus [x_\nu,\infty)$, where
\[
x_\nu = \hat X(\nu,U_\nu,W_\nu) >0
\]
is the radius of convergence of $x \mapsto Z^+(\nu,U(\nu,t_\nu^3),x)$ by composition. The rest is relatively straightforward: we compute the expansion of $\hat X$ at $W_\nu$ and plug it into $\hat Z^+$ to get the statement.
\end{proof}

\section{Boltzmann maps and BDG functions} \label{sec:BoltzmannAll}

To study the root spin cluster, we will interpret it in the context of Boltzmann planar maps associated to a given weight sequence, which has been studied intensively in the literature~\cite{MiermontInvariance,MarckertMiermontInvariance,BuddPeel,LeGallMiermont,Mar18b,CRlooptrees}. In this section, we first review the necessary material about this model of maps. This includes universal expressions for their partition functions, and the connection with a pair of bivariate functions $f^\bullet$ and $f^\diamond$, given by the Bouttier--Di Francesco--Guitter bijection. 

In a second part, we prove that the root spin cluster has the distribution of a Boltzmann planar map with an explicit weight sequence. We study the asymptotic properties of this weight sequence as well as the singular behavior of the two functions $f^\bullet$ and $f^\diamond$ associated to it.

\subsection{Background on Boltzmann maps} \label{sec:Boltzmann}

\subsubsection{Partition functions}\label{sub:partitionFunctions}

Fix any sequence of non-negative real numbers $\mathbf q = (q_k)_{k \geq 1}$ such that $q_k \neq 0$ for some odd $k \geq 3$.
Recall that $\mathcal{M}$ (respectively $\mathcal{M}^\bullet$) denotes the set of finite non-atomic rooted (respectively and pointed) planar maps. For any $\mathfrak m \in \mathcal M$, the $\mathbf q$-Boltzmann weight of $\mathfrak m$ is defined by
\begin{equation}\label{eq:defwq}
w_{\mathbf q}(\mathfrak m) = \prod_{f \in F(\mathfrak m)} q_{\mathrm{deg}(f)}.
\end{equation}
The partition function $w_{\mathbf q}(\mathcal{M})$ and the pointed partition function $w_{\mathbf q}^\bullet(\mathcal{M})$ of $\mathbf q$-Boltzmann maps are then defined by:
\begin{align*}
w_{\mathbf q} (\mathcal M) & = \sum_{\mathfrak m \in \mathcal{M} } \, \prod_{f \in F(\mathfrak m)} q_{\mathrm{deg}(f)} ,\\
w_{\mathbf q} (\mathcal M^\bullet)  &= \sum_{\mathfrak m \in \mathcal{M}^\bullet} \,  \prod_{f \in F(\mathfrak m)} q_{\mathrm{deg}(f)} = \sum_{\mathfrak m \in \mathcal{M}} \, |V(\mathfrak m)| \, \prod_{f \in F(\mathfrak m)} q_{\mathrm{deg}(f)}.
\end{align*}

\begin{defi}\label{def:admissible}
A weight sequence $\mathbf q$ is said to be \emph{admissible} if its partition function satisfies
\[
w_{\mathbf q} (\mathcal M) < \infty,
\]
which is equivalent to $w_{\mathbf q} (\mathcal M^\bullet) < \infty$, see \cite[Proposition 4.1]{BeCuMie}.

We can associate to any admissible weight sequence $\mathbf q$ a probability measure $\mathbb P^{\mathbf q}$ on $\mathcal M$ defined by,
\begin{equation}
\forall \mathfrak m \in \mathcal M , \quad \mathbb P^{\mathbf q} \left( \{ \mathfrak m \} \right)
=
\frac{w_{\mathbf q}(\mathfrak m)}{w_{\mathbf q} (\mathcal M)}.
\end{equation}
A random planar map sampled from $\mathbb P^{\mathbf q}$ is called a \emph{$\mathbf q$-Boltzmann map} and is denoted by $\mathbf M^\mathbf q$.
\end{defi}

\bigskip

For admissible weight sequences, other important quantities are the so-called disk partition functions and their pointed versions. For $l \geq 1$, let $\mathcal M_l$ denote the set of finite rooted planar maps whose root face $f_r$ is of degree $l$. For every $l \geq 1$, the disk function and the pointed disk function are respectively defined  by:
\begin{align}\label{eq:defDiskL}
W^{(l)}_{\mathbf q} &= \sum_{ \mathfrak m \in \mathcal M_l} 
\prod_{f \in F( \mathfrak m) \setminus \{ f_r\}} q_{\mathrm{deg}(f)},\\
W^{(l)}_{\mathbf q ,\bullet} &= \sum_{ \mathfrak m \in \mathcal M_l}
|V(\mathfrak m)|
\prod_{f \in F(\mathfrak m) \setminus \{ f_r\}} q_{\mathrm{deg}(f)},
\end{align}
which are all finite if $\mathbf q$ is admissible.
Note that the case $l = 2$ is of special interest since it allows to recover the partition functions of the $\mathbf q$-Boltzmann maps by the identities
\begin{equation} \label{eq:cyltobolt}
W^{(2)}_{\mathbf q} - 1 = w_{\mathbf q} (\mathcal M)
\quad \text{and} \quad
W^{(2)}_{\mathbf q ,\bullet} - 1 = w_{\mathbf q} (\mathcal M^\bullet),
\end{equation}
which come from gluing the edges of the root face of maps with root face degree $2$.

\bigskip

The generating functions of disk partition functions are usually defined formally as follows:
\begin{align}
W_{\mathbf q}(z) &=  \sum_{l \geq 0} W^{(l)}_{\mathbf q} \, z^{-(l+1)},\label{eq:defDiskPartition}\\
W_{\mathbf q ,\bullet}(z) &= \sum_{l \geq 0} W^{(l)}_{\mathbf q ,\bullet} \, z^{-(l+1)},\label{eq:defPointedDiskPartition}
\end{align}
with the convention that $ W^{(0)}_{\mathbf q} = W^{(0)}_{\mathbf q ,\bullet} = 1$. 
A striking property of Boltzmann maps is the universal form of the pointed disk function $W_{\mathbf q,\bullet}$ (see for instance Budd \cite[Proposition 2]{BuddPeel} and Borot-Bouttier-Guitter \cite[Section 6.1]{BBGb} for two different proofs). Indeed, there exist real numbers $c_+ (\mathbf q) >2$ and $c_-(\mathbf q) \in [-c_+(\mathbf q),c_+(\mathbf q))$ such that $W_{\mathbf q,\bullet}$ has an analytic continuation on $\mathbb C \setminus [c_-(\mathbf q),c_+(\mathbf q)]$ and such that, on this domain, we have:
\begin{equation}\label{eq:pointedDisk}
W_{\mathbf q,\bullet}(z) = \frac{1}{\sqrt{\left( (z- c_+(\mathbf q)) (z- c_-(\mathbf q)) \right)}}.
\end{equation}
In addition, if the maps are not bipartite (which is the case here), then $c_- (\mathbf q)\neq -c_+(\mathbf q)$. The unpointed disk function does not have a universal form, but it does have an analytic continuation on the same domain $\mathbb C \setminus [c_-(\mathbf q),c_+(\mathbf q)]$ than the pointed disk function (see for example \cite[Section 6.1]{BBGb} and \cite[Proposition 4.3]{BeCuMie}).

\subsubsection{Critical and regular critical weight sequences}

Weight sequences are classified according to their asymptotic properties.
The case where the number of vertices of a $\mathbf q$-Boltzmann map has infinite variance is of special interest, which leads to the following definition.

\begin{defi} \label{def:critique}
An admissible weight sequence $\mathbf q$ is said to be
\begin{itemize}
\item \emph{critical} if
\begin{equation} \label{eq:defcrit}
\sum_{\mathfrak m \in \mathcal{M}^\bullet} \, |V(\mathfrak m)|^2 \, \prod_{f \in F(\mathfrak m)} q_{\mathrm{deg}(f)} = + \infty,
\end{equation}
\item \emph{sub-critical} if
\begin{equation} \label{eq:defsubcrit}
\sum_{\mathfrak m \in \mathcal{M}^\bullet} \, |V(\mathfrak m)|^2 \, \prod_{f \in F(\mathfrak m)} q_{\mathrm{deg}(f)} < + \infty.
\end{equation}
\end{itemize}
\end{defi}

Critical weight sequences are further classified according to the tail distribution of the degree of the root face of the associated $\mathbf q$-Boltzmann map. This classification can be found in \cite[Definition 1 and Proposition 3]{MiermontInvariance}.

\begin{defi} \label{def:regcrit}
A critical weight sequence $\mathbf q$ is said to be
\begin{itemize}
\item \emph{regular} (or \emph{generic}) critical if there exists $r \in (0,1)$ such that, as $l \to \infty$, one has
\begin{equation} \label{eq:defregcrit}
\mathbb P^{\mathbf q} \left( |\partial \mathbf M^\mathbf q | = l \right) = 
\frac{w_\mathbf{q} (\mathcal M_l)}{w_\mathbf{q} (\mathcal M)} = \frac{q_l \cdot W_{\mathbf q}^{l}}{w_\mathbf{q} (\mathcal M)} = \mathcal O \left( r^l \right);
\end{equation}
\item \emph{non-regular} (or \emph{non-generic}) critical if it is not regular critical.
\end{itemize}
\end{defi}

\bigskip

An important feature of regular critical Boltzmann random maps is the universal behavior for the tail distribution of their number of vertices. We refer the interested reader to \cite{BeCuMie} for further details and also to \cite[Lemma 2]{BuddPeel} for a related result.

\begin{prop}[Proposition 5.1 of \cite{BeCuMie}] \label{prop:regcritvol}
Let $\mathbf q$ be a regular critical weight sequence. Then there exists a constant $c > 0$ such that
\begin{equation} \label{eq:regcritvol}
\mathbb P^{\mathbf q} \left( |V( \mathbf M^\mathbf q ) | \geq  n \right)
\underset{n \to \infty}{\sim}
c \, n^{-3/2}.
\end{equation}
\end{prop}

\subsubsection{Bouttier-Di Francesco-Guitter functions} \label{sec:BDGgen}

An essential tool to study Boltzmann planar maps is the Bouttier-Di Francesco-Guitter bijection \cite{BDFG}. Our work relies heavily on one of its consequences established in slight variations by Miermont \cite{MiermontInvariance}, Budd \cite{BuddPeel} or Bernardi, Curien and Miermont \cite{BeCuMie}. 
For a weight sequence $\mathbf{q}$, they give a criterion for $\mathbf{q}$ to be admissible, and if applicable, they give an expression for the partition function of $\mathbf{q}$-Boltzmann maps.

More precisely, let $f^\bullet_{\mathbf q}$ and $f^\diamond_{\mathbf q}$ be the two bivariate functions defined by: 
\begin{align}
f^\bullet_{\mathbf q} (z_1,z_2) &= \sum_{k,k' \geq 0} z_1^k z_2^{k'} \binom{2k+k' +1}{k+1,k,k'} q_{2+2k+k'}, \label{eq:deffbullet}\\
f^\diamond_{\mathbf q} (z_1,z_2) &= \sum_{k,k' \geq 0} z_1^k z_2^{k'} \binom{2k+k'}{k,k,k'} q_{1+2k+k'}.\label{eq:deffdiamond}
\end{align}
It is proved in these three articles that the weight sequence $\mathbf q$ is admissible if and only if there is a unique solution $(z^+(\mathbf q), z^\diamond(\mathbf q)) \in \mathbb R_+^2$ to the system
\begin{equation}\label{eq:fixpoint}
\begin{cases}
f^\bullet_{\mathbf q} (z^+(\mathbf q),z^\diamond(\mathbf q)) = 1 - \frac{1}{z^+(\mathbf q)},\\
f^\diamond_{\mathbf q} (z^+(\mathbf q),z^\diamond(\mathbf q)) = z^\diamond(\mathbf q),\\
(\partial_{z_2} + \sqrt{z_1} \partial_{z_1}) f^\diamond_{\mathbf q} (z^+(\mathbf q),z^\diamond(\mathbf q)) \leq 1.
\end{cases}
\end{equation}
In addition, if $\mathbf{q}$ is admissible, the solution of the system has the following combinatorial interpretation
\begin{equation} \label{eq:zpzdcombi}
\begin{cases}
z^+ (\mathbf q) &= 1 + \sum_{\mathfrak m \in \mathcal{M}^+} \, \prod_{f \in F(\mathfrak m)} q_{\mathrm{deg}(f)} = 1+ w_{\mathbf q} (\mathcal{M}^+), \\
z^\diamond (\mathbf q) &= \sqrt{ \sum_{\mathfrak m \in \mathcal{M}^0}  \, \prod_{f \in F(\mathfrak m)} q_{\mathrm{deg}(f)}} = \sqrt{w_{\mathbf q} (\mathcal{M}^0)},
\end{cases}
\end{equation}
where $\mathcal M^+$ (respectively $\mathcal M^0$) denotes the subset of $\mathcal{M}^\bullet$ made of maps whose root edge points to a vertex that is closer (respectively at equal distance) to the marked vertex than is the root vertex. See for example \cite[Lemma 4.4]{BeCuMie}. As a consequence, we have
\begin{equation} \label{eq:wdzpzd}
w_{\mathbf q} (\mathcal M^\bullet) = 2z^+ (\mathbf q) + (z^\diamond (\mathbf q))^2 -1,
\end{equation}
or, equivalently,
\begin{equation} \label{eq:cpmz}
c_+ (\mathbf q) = z^\diamond (\mathbf q) + 2 \sqrt{z^+} (\mathbf q) \quad \text{and} \quad  \quad c_- (\mathbf q) = z^\diamond (\mathbf q) - 2 \sqrt{z^+} (\mathbf q).
\end{equation}
Based on these results, the following characterization of criticality is given in 
\cite{BeCuMie}:
\begin{prop}[Proposition~4.3 of~\cite{BeCuMie}]\label{prop:criterionCriticality}
For $\mathbf q$ an admissible weight sequence, the three following statements are equivalent: 
\begin{itemize}
\item $\mathbf{q}$ is sub-critical (in the sense given in Definition~\ref{def:critique}),
\item the inequality in the third equation of~\eqref{eq:fixpoint} is strict, 
\item its disk partition functions have the following asymptotic behavior: there exist constants $c,R>0$ such that: 
\begin{equation} \label{eq:subcritweight}
W_\mathbf q ^{(l)} \underset{l \to \infty}{\sim} c \, R^l \, l^{-3/2}.
\end{equation}
\end{itemize}
\end{prop}

\bigskip

Finally, the functions $f^\bullet_{\mathbf q}$ and $f^\diamond_{\mathbf q}$ also give indirect access to the \emph{volume generating functions} of $\mathbf q$-Boltzmann planar maps. Indeed, for $0 < g \leq 1$ consider the weight sequence $\mathbf q_g$ defined by $\mathbf q_g(k) = g^{(k-2)/2} q_k$ for $k \geq 1$. Euler's formula ensures that
\[
g \, w_{\mathbf q_g} (\mathcal M^\bullet)  = \sum_{ \mathfrak m \in \mathcal M } \, |V( \mathfrak m)| \, g^{|V( \mathfrak m)| -1} \, \prod_{f \in F( \mathfrak m)} q_{\mathrm{deg}(f)}.
\]
Following~\cite{BeCuMie}, and by analogy with \eqref{eq:zpzdcombi}, we define
\begin{equation} \label{eq:zpzdcombig}
\begin{cases}
z^+ (\mathbf q,g) &= g \left( 1+ w_{\mathbf q_g} (\mathcal{M}^+) \right), \\
z^\diamond (\mathbf q,g) &= \sqrt{g \, w_{\mathbf q_g} (\mathcal{M}^0)},
\end{cases}
\end{equation}
so that we have: 
\begin{equation} \label{eq:Zg}
\sum_{ \mathfrak m \in \mathcal M } \, |V( \mathfrak m)| \, g^{|V( \mathfrak m)| -1} \, \prod_{f \in F( \mathfrak m)} q_{\mathrm{deg}(f)} 
= 2 (z^+(\mathbf q,g) - g) + (z^\diamond(\mathbf q,g))^2,
\end{equation}
and
\begin{equation} \label{eq:cpmqg}
c_\pm(\mathbf q_g) = \frac{z^\diamond(\mathbf q,g) \pm 2 \sqrt{z^+(\mathbf q,g)}}{\sqrt{g}}.
\end{equation}

\bigskip

From their definition, $z^+(\mathbf q,g)$ and $z^\diamond(\mathbf q,g)$ are easily seen to be increasing in $g$. Besides, by monotone convergence their limit as $g \to 1^-$ are given by $z^+(\mathbf q,1) = z^+(\mathbf q)$ and $z^\diamond(\mathbf q,1) = z^\diamond(\mathbf q)$.
On the other hand, we can easily check that for $g \in (0,1]$ and $x,y >0$ we have
\begin{equation}\label{eq:fqFixed}
f^\bullet_{\mathbf q_g} (x,y) = f^\bullet_{\mathbf q} (g \, x, \sqrt g  \, y) \quad \text{and} \quad f^\diamond_{\mathbf q_g} (x,y) = f^\diamond_{\mathbf q} (g \, x, \sqrt g  \, y) / \sqrt{g}.
\end{equation}
Therefore, by first applying~\eqref{eq:fixpoint} and~\eqref{eq:zpzdcombi} to $f^\bullet_{\mathbf q_g}$ and $f^\circ_{\mathbf q_g}$, then by replacing $w_{\mathbf q_g}(\mathcal{M}^+)$ and $w_{\mathbf q_g}(\mathcal{M}^0)$ by their expression in terms of $z^\diamond(\mathbf q,g)$ and of $z^+(\mathbf q,g)$, and by using~\eqref{eq:fqFixed}, we obtain that for $0 < g \leq 1$ the quantities $0 < z^+(\mathbf q,g) \leq z^+(\mathbf q)$ and $0 < z^\diamond(\mathbf q,g) \leq z^\diamond( \mathbf q)$ are the unique solutions of: 
\begin{equation} \label{eq:fixpointg}
\begin{cases}
f^\bullet_\mathbf q (z^+(\mathbf q,g),z^\diamond(\mathbf q,g)) &= 1 - \frac{g}{z^+(\mathbf q,g))},\\
f^\diamond_\mathbf q (z^+(\mathbf q,g),z^\diamond(\mathbf q,g)) &= z^\diamond(\mathbf q,g).
\end{cases}
\end{equation}

\subsubsection{Maps of the cylinder}

We now turn our attention to the generating function of Boltzmann maps of the cylinder, that is maps with two marked faces. More precisely, for $l_1, l_2 >0$, let $\mathcal M_{(l_1,l_2)}$ be the set of all finite maps with two boundary faces $f_1$ and $f_2$ of respective perimeter $l_1$ and $l_2$, and rooted at each boundary in such a way that $f_1$ and $f_2$ lie on the right side of the corresponding root. The partition function of $\mathbf q$-Boltzmann maps in $\mathcal M_{(l_1,l_2)}$ is defined by:
\begin{equation} \label{eq:cylinder}
W_{\mathbf q, 2}^{(l_1,l_2)} 
=
\sum_{ \mathfrak m \in \mathcal M_{(l_1,l_2)}} 
\prod_{f \in F( \mathfrak m) \setminus \{ f_1,f_2\}} q_{\mathrm{deg}(f)}
\end{equation}
and their generating series by
\[
W_{\mathbf q, 2}(z_1,z_2) = \sum_{l_1,l_2 \geq 1} \frac{W_{\mathbf q,2}^{(l_1,l_2)}}{z_1^{l_1+1} z_2^{l_2+1}}.
\]
The expression of $W_{\mathbf q , 2}$ is well known and is given by:
\begin{prop} \label{prop:gencylinders} The generating series $W_{\mathbf q , 2} (z_1,z_2)$  of $\mathbf q$-Boltzmann planar maps of the cylinder admits the following expression: 
\begin{multline}\label{eq:cylindre}
W_{\mathbf q , 2} (z_1,z_2) \\
= \frac{1}{2(z_1-z_2)^2} \left(W_{\mathbf q, \bullet} (z_1) W_{\mathbf q ,\bullet} (z_2) \left( z_1z_2 -\frac{c_+(\mathbf q) + c_-(\mathbf q)}{2} (z_1 + z_2 ) + c_+(\mathbf q) c_- (\mathbf q)\right) -1\right).
\end{multline}
\end{prop}
This formula was originally obtained in the physics literature via some connections with matrix integrals, see~\cite{AJM90}. It also appears in \cite[Theorem 3.3]{borot2018nesting} and in the recent book by Eynard~\cite[Theorem~3.2.1]{EynardBook}, but is given via a rational parametrization. We give in Appendix~\ref{appendix:cylinder}, the connection between the expression given in~\eqref{eq:cylindre} and Eynard's formula. For the sake of completeness, we also give a fully combinatorial derivation of this proposition based on the theory of slice decomposition introduced by Bouttier and Guitter~\cite{BouttierGuitterIrreducible,BouttierHDR}.

\subsection{The root spin cluster is a \texorpdfstring{$\mathbf q(\nu,t)$}{q(nu,t)}-Boltzmann map}

The main point of this section is to prove that the root spin cluster of an Ising-weighted triangulation is a Boltzmann map, whose weight sequence is the sequence $\mathbf{q}(\nu,t)$ defined in the introduction. As a byproduct, we will see that this weight sequence is always admissible when $\nu > 0$ and $0 < t \leq t_\nu$. We will then study the singular behavior in $t$ and $y$ of these weights.

\subsubsection{Decomposition of finite triangulations into the root spin cluster and islands}\label{sub:gasket}

Recall that $\mathcal T^\ps$ is the set of Ising-weighted triangulations with a positive monochromatic root edge. For $\mathfrak t \in \mathcal T^\ps$, we also recall that its root spin cluster -- denoted by $\mathfrak C (\mathfrak t)$ -- is the submap of $\mathfrak t$ spanned by all the vertices with spin $\ps$ connected to the root by a sequence of monochromatic edges, with the same root edge as $\mathfrak t$. The root spin cluster of a triangulation of $\mathcal T^\ps$ can thus be seen as a finite rooted non atomic map without spins. 
\bigskip

To characterize the probability distribution of the root spin cluster of a triangulation with spins $\mathbf{T}^\nu$ sampled from $\mathbb{P}^\nu$, we decompose $\mathbf{T}^\nu$ into its root spin cluster and some submaps filling its faces. As mentioned in the introduction, such decompositions are now classical and appear in the literature under several aliases: gasket decomposition in \cite{BBD,BBGa,BBGc,BBGb,borot2018nesting} or reef-island decomposition in \cite{BeCuMie}. We start by introducing the weight sequence:

\begin{defi}\label{def:qk}
For every $\nu > 0$ and every positive integer $k$, we define for $| t | \leq t_\nu$
\begin{equation}\label{eq:qk}
q_k (\nu, t) = 
\left( \nu \, t \right)^{k/2} \mathbf{1}_{\{k=3\}}
+ \left( \nu \, t^{3}  \right)^{k/2} \cdot \sum_{l \geq 0} \binom{k +l-1}{k-1} \, t^{l} \, Q_l^+ (\nu,t),
\end{equation}
and set $\mathbf q(\nu,t) = \left( q_k (\nu, t) \right)_{k \geq 1}$. 
\end{defi}
Notice that the generating series of this weight sequence is the series $F( \nu , t ,z )$ introduced in~\eqref{eq:genqk} in the introduction. Indeed, we have: 
\begin{align*}
\sum_{k \geq 1} q_k (\nu,t)z^k
& = (\nu t)^{3/2} z^3 +
\sum_{l \geq 0}
t^{l} \, Q_l^+ (\nu,t) \,
\sum_{k \geq 1} \binom{k +l-1}{k-1} \left( \nu \, t^{3} \right)^{k/2} z^k, \notag\\
&= (\nu t)^{3/2} z^3 +
\sqrt{\nu t^3} z \,
\sum_{l \geq 0}
t^{l} \, Q_l^+ (\nu,t) \cdot (1 - \sqrt{\nu t^3} z)^{-(l+1)}, \notag\\
&= (\nu t)^{3/2} z^3 +
\frac{\sqrt{\nu t^3} z}{1- \sqrt{\nu t^3} z} \,
Q^+ \left( \nu,t , \frac{t}{1-\sqrt{\nu t^3}z} \right ), \notag\\
&= F(\nu ,t ,z).
\end{align*}

The next statement relates the weight sequence $\mathbf{q}(\nu,t)$ with the partition function of the root spin cluster: 
\begin{prop} \label{prop:q_k}
Fix $\nu >0$ and $t \in (0,t_\nu]$.
For any $\mathfrak m \in \mathcal M$, we have:
\[
 \sum_{(\mathfrak t,\sigma) \in \mathcal T^\ps \, : \, \mathfrak C (\mathfrak t) = \mathfrak m}
t^{|\mathfrak t|} \nu^{m(\mathfrak t,\sigma)}
= \prod_{f \in F(\mathfrak m)} q_{\mathrm{deg}(f)}(\nu,t).
\]
As a consequence, we have the following identity:
\[
\mathcal Z ^\ps (\nu , t) = \sum_{\mathfrak m \in \mathcal M }
\prod_{f \in F(\mathfrak m)} q_{\mathrm{deg}(f)}(\nu,t) < \infty, 
\]
where $\mathcal Z^{\ps}$ is defined in \eqref{eq:condplus}. In particular, the weight sequence $\mathbf q (\nu ,t )$ is admissible.
\end{prop}

Before giving the proof of this proposition, we state the following characterization of the root spin cluster as a $\mathbf{q}(\nu,t)$-Boltzmann map, which is at the heart of our analysis and is one of its direct consequences:

\begin{prop} \label{prop:bolt}
Let $\mathfrak m \in \mathcal{M}$ and fix $\nu >0$.
\begin{itemize}
\item Unconditioned volume case: The law of the root spin cluster $\mathfrak C(\mathbf T^\nu)$ is given by
\begin{equation}  \label{eq:bolt}
\mathbb P^\nu \left( \mathfrak C(\mathbf T^\nu) = \mathfrak m \right) = \frac{\prod_{f \in F(\mathfrak m)} q_{\mathrm{deg}(f)} (\nu ,t_\nu)}{\mathcal Z ^\ps ( \nu ,t_\nu)}.
\end{equation}
In other words, \emph{$\mathfrak{C}(\mathbf{T}^\nu)$ is distributed as a $\mathbf{q}(\nu,t_\nu)$-Boltzmann map}.

\item Conditioned volume case: The law of the root spin cluster $\mathfrak C(\mathbf T^\nu_n)$ is given by
\begin{equation} \label{eq:boltn}
\mathbb P_n^\nu \left( \mathfrak C(\mathbf T_n^\nu) = \mathfrak m \right) = \frac{[t^{3n}] \prod_{f \in F(\mathfrak m)} q_{\mathrm{deg}(f)} (\nu ,t)}{[t^{3n}] \mathcal Z ^\ps ( \nu ,t)}.
\end{equation}
\end{itemize}
\end{prop}
\begin{figure}[t!]
\begin{center}
 \subfloat[An element of $\mathcal{T}^\ps,$]{
      \includegraphics[width=0.45\textwidth,page=1]{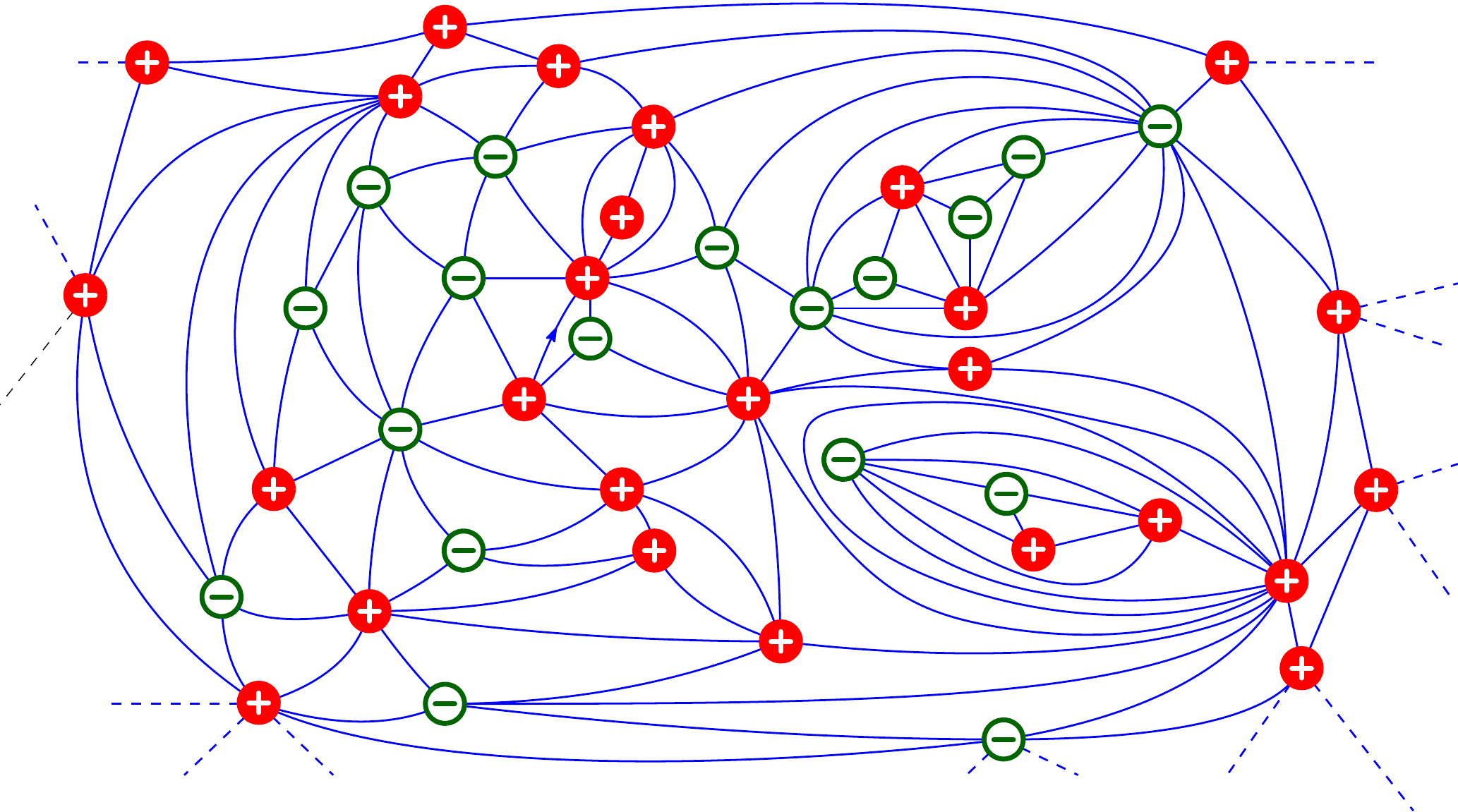}
      \label{subfig:config}} \qquad
    \subfloat[and its root spin cluster. Letters $A$ and $B$ refer to Figures~\ref{subfig:decA} and~\ref{subfig:decB}.]{
      \includegraphics[width=0.45\textwidth,page=2]{Clusters.pdf}
      \label{subfig:cluster}}\\

     \subfloat[Island (left) and necklace (right) corresponding to the face $A$ in \ref{subfig:cluster}.]{
      \includegraphics[width=0.45\textwidth,page=5]{Clusters.pdf}
      \label{subfig:decA}}\qquad  
    \subfloat[Island (left) and necklace (right) corresponding to the face $B$ in \ref{subfig:cluster}.]{
      \includegraphics[width=0.45\textwidth,page=6]{Clusters.pdf}
      \label{subfig:decB}
                     }

    \caption{Example of the root spin cluster of an Ising configuration, see \ref{subfig:cluster}, and of the necklace/island decomposition, see \ref{subfig:decA} and \ref{subfig:decB}}
    \label{fig:ExampleCluster}
\end{center}
\end{figure}

\begin{proof}[Proof of Proposition~\ref{prop:q_k}]
Following \cite{BBGa} and \cite{BeCuMie}, for any $\mathfrak{t}\in \mathcal{T}^\ps$, we consider the following decomposition of $\mathfrak t$. Let $C$ be the cut-set of $\mathfrak{C}(\mathfrak t)$, that is the set of edges of $\mathfrak t$ that have one endpoint in $\mathfrak{C}(\mathfrak t)$ and the other endpoint not in $\mathfrak{C}(\mathfrak t)$. It follows directly from the definition of $\mathfrak{C}(\mathfrak t)$, that every edge in $C$ is frustrated (\emph{i.e.} not monochromatic). Next, the connected components of $\mathfrak t \backslash C$ are $\mathfrak{C}(\mathfrak t)$ and a collection of maps (eventually atomic) -- called \emph{islands} -- with non simple boundaries, and such that their boundary vertices have all spin $\ns$, see Figure~\ref{fig:ExampleCluster}. Furthermore, each face of $\mathfrak C (\mathfrak t)$ that is not of degree $3$ contains exactly one of these islands. Faces of $\mathfrak C (\mathfrak t)$ with degree $3$ can contain either an island or nothing if it is also a face of the original triangulation $\mathfrak t$.

\bigskip

Reciprocally, suppose that we are given a cluster $\mathfrak C(\mathfrak t)$ and a collection of maps with spins, indexed by the faces of $\mathfrak C(\mathfrak t)$ (except possibly for some faces of degree 3). Assume that these maps have non simple boundaries (eventually atomic) with monochromatic $\ns$ boundary condition. 

The only missing information to recover the full triangulation is the edges between vertices of $\mathfrak C(\mathfrak t)$ and the maps inside each of its faces. These parts are called reefs in \cite{BeCuMie} or necklaces in \cite{BBGa}. The part between the boundary of a face of $\mathfrak C(\mathfrak t)$ of degree $k$ and the corresponding island with boundary length $l$ is therefore a map with a root face of degree $k$, a marked inner face of degree $l$ such that these two faces do not share vertices. Moreover, vertices in this map are only incident either to the root face or to the marked inner face. 
 In addition, every other face of this map is a triangle. Thus, there are $k+l$ such triangles, among them $k$ share an edge with the root face and their third vertex with the marked face (if it is non empty), and $l$ share a edge with the marked face and their third vertex with the root face. There are $\binom{k+l-1}{k-1}$ such necklaces (this is the number of different orderings of the two types of triangles). See Figures~\ref{fig:ExampleCluster}\subref{subfig:decA} and \ref{fig:ExampleCluster}\subref{subfig:decB} for an illustration.

Summing over every possible island and reef, this decomposition yields the following identity for any fixed non atomic map $\mathfrak m$:
\[
\sum_{(\mathfrak t,\sigma) \in \mathcal T^\ps \, : \, \mathfrak C(\mathfrak t) = \mathfrak m}
t^{|\mathfrak t|} \nu^{m(\mathfrak t,\sigma)}
=
t^{|\mathfrak m|} \nu^{|\mathfrak m|}
\prod_{f \in F(\mathfrak m)} \left( \mathbf 1_{\mathrm{deg}(f)=3} + \sum_{l \geq 0} \binom{\mathrm{deg}(f) +l-1}{\mathrm{deg}(f)-1} \, t^{\mathrm{deg}(f) +l} \, Q_l^+ (\nu,t) \right)
\]
where the term $\mathbf 1_{k=3}$ takes into account the fact that a face of degree $3$ in $\mathfrak m$ could be a face of the original triangulation $\mathfrak t$. In view of the definition of $\mathbf q (\nu,t)$ given in Definition~\ref{def:qk}, this identity can immediately be rewritten as: 
\begin{align} \label{eq:weightcluster}
\sum_{(\mathfrak t,\sigma) \in \mathcal T^\ps \, : \, \mathfrak C(\mathfrak t) = \mathfrak m}
t^{|\mathfrak t|} \nu^{m(\mathfrak t,\sigma)}
=
\prod_{f \in F(\mathfrak m)} q_{\mathrm{deg}(f)} (\nu, t)
\end{align}
The last statement of the proposition follows from the definition of $\mathcal Z ^\ps (\nu,t)$ and the fact that $\mathcal Z^\ps(\nu,t_\nu)<\infty$ as stated in Proposition~\ref{prop:asymptoZmono}.
\end{proof}

\subsubsection{Disk generating functions of \texorpdfstring{$\mathbf{q} (\nu,t)$}{q(nu,t)}-Boltzmann maps}

As mentioned in Section~\ref{sub:partitionFunctions}, the unpointed disk function defined in~\eqref{eq:defDiskPartition} does not have a universal form. Fortunately, for $\mathbf q(\nu,t)$-Boltzmann maps it can be computed explicitly:

\begin{prop} \label{prop:Wnut}
Fix $\nu > 0$ and $t \in (0,t_\nu]$, the disk generating function associated to $\mathbf q(\nu,t)$ is given by
\[
W_{\mathbf q(\nu,t)}(z) = \frac{1}{z} Q^+ \left(\nu,t,\frac{t}{z \sqrt{\nu t^3}} \right).
\]
As a consequence, for every $\nu >0$, when $t = t_\nu$, the weight sequence $\mathbf q(\nu,t_\nu)$ is critical in the sense given in~\eqref{eq:defcrit}.
\end{prop}

\begin{proof}
Consider a triangulation $\mathfrak{t}$ endowed with an Ising configuration and with a \emph{monochromatic} boundary of length $l$. When we apply the gasket decomposition presented in Section~\ref{sub:gasket} to $\mathfrak{t}$, we obtain a map of $\mathcal M_{l}$.
By following the same line of arguments as in Proposition~\ref{prop:q_k}, we see that: 
\begin{equation} \label{eq:Wl}
W^{(l)}_{\mathbf q(\nu,t)} = Q_l^+(\nu,t) \cdot (\nu t )^{-l/2},
\end{equation}
where the corrective term $(\nu t )^{-l/2}$ comes from the fact that the root face does not contribute to $W^{(l)}_{\mathbf q(\nu,t)}$. Therefore each edge incident to the root face contributes only a weight $\sqrt{\nu t}$ to $W^{(l)}$, whereas it contributes a weight $\nu t$ to $Q_l^+(\nu,t)$.

The generating series is then easy to compute:
\begin{align*}
W_{\mathbf q(\nu,t)}(z) &= \sum_{l \geq 0} Q_l^+(\nu,t) \cdot (\nu t )^{-l/2} z^{-(l+1)},\\
& = \sum_{l \geq 0} t^l Q_l^+(\nu,t) \cdot (\nu t^3 )^{-l/2} z^{-(l+1)},\\
& = \frac{1}{z} Q^+ \left(\nu,t,\frac{t}{z \sqrt{\nu t^3}} \right).
\end{align*}

To see that the weight sequence $\mathbf q (\nu,t_\nu)$ is critical, we apply the criticality criterion given in Proposition~\ref{prop:criterionCriticality}, where we rely on the asymptotic behavior of $Q^+_l(\nu,t_\nu)$, as $l$ grows to infinity, given in Corollary \ref{cor:Q}.
\end{proof}

Since the function $W_{\mathbf q(\nu,t)}$ always has an analytic continuation on the same domain as the pointed function $W_{\mathbf q(\nu,t),\bullet}$ defined in~\eqref{eq:defPointedDiskPartition}, this expression allows us to compute $c_+ (\nu,t) := c_+(\mathbf q(\nu,t))$ and $c_-(\nu,t) := c_-(\mathbf q(\nu,t))$ in terms of the singularities in y of $Q^+(t,ty)$ appearing in Proposition \ref{prop:asym_t_sign_y}. We obtain: 

\begin{coro}
The pointed disk function $W_{\mathbf q(\nu,t),\bullet}$ defined in~\eqref{eq:defPointedDiskPartition} admits the following expression: 
\begin{equation}\label{eq:pointedExplicit}
W_{\mathbf q(\nu,t),\bullet}(z) = \frac{1}{\sqrt{\Big( z- c_+(\nu,t)\Big) \Big(z- c_-(\nu,t) \Big)}},
\end{equation}
where
\begin{equation}\label{eq:cy+-}
c_+(\nu,t) = \frac{1}{\sqrt{\nu t^3} \, y_+(\nu,t)}\quad \text{ and }\quad 
c_-(\nu,t) = \frac{1}{\sqrt{\nu t^3} \, y_-(\nu,t)}. 
\end{equation}
\end{coro}

\begin{rema}\label{rem:y+-}
Since the maps we consider are not bipartite, we have $c_+(\nu,t)>-c_-(\nu,t)$ (this is a general property of non-bipartite Boltzmann maps recalled in Section~\ref{sub:partitionFunctions} and which is a direct consequence of~\eqref{eq:cpmz}). The expression of $c_+(\nu,t)$ and $c_-(\nu,t)$ given above then imply that 
\begin{equation} \label{eq:ymcompareyp}
y_-(\nu,t) < - y_+(\nu,t),
\end{equation}
as was announced in Proposition \ref{prop:yineq}.
\end{rema}

The expressions for $c_+(\nu,t)$ and $c_-(\nu,t)$ in terms of $y_+(\nu,t)$ and $y_-(\nu,t)$ also give expressions for the solution $\left( z^+(\nu,t),z^\diamond(\nu,t) \right) := \left( z^+(\mathbf q(\nu,t)),z^\diamond(\mathbf q(\nu,t)) \right)$ of the system of equations \eqref{eq:fixpoint}:
\begin{align}
z^+ (\nu,t) &= \frac{1}{2\sqrt{\nu t^3}} \left( \frac{1}{y_+(\nu,t)} + \frac{1}{y_-(\nu,t)} \right), \label{eq:z+y+-}\\
z^\diamond ( \nu,t) &= \frac{1}{4 \, \nu t^3} \left( \frac{1}{y_+(\nu,t)} - \frac{1}{y_-(\nu,t)} \right)^2. \label{eq:zdy+-}
\end{align}

\subsubsection{Asymptotic properties of the weights \texorpdfstring{$q_k(\nu,t)$}{qk(nu,t)}}

We gather in this section some asymptotic properties of the weight sequence $\mathbf q (\nu,t)$. We start with the special case where $t = t_\nu$:

\begin{prop} \label{prop:asymptoq}
The weight sequence $(q_k(\nu,t_\nu))_{k \geq 1}$ has the following asymptotic behavior:
\[
q_k (\nu,t_\nu) \underset{k\to \infty}{\sim} \frac{(y_\nu -1)^{\alphaa(\nu)} \, \aleph^{Q^+} (\nu)}{\Gamma (1-\alphaa(\nu))} \, \left(\frac{y_\nu}{y_\nu-1} \, \sqrt{\nu t_\nu^{3} } \right)^k k^{-\alphaa(\nu)},
\]
where we recall that $y_\nu$ and $\alphaa(\nu)$ were first defined in Proposition~\ref{lemm:weightsclusters} and satisfy:
\begin{equation*}
\begin{cases}
y_\nu = 2 \text{ and }\alphaa(\nu) = 5/3& \text{for }\nu < \nu_c\\
y_{\nu_c} = 2 \text{ and }\alphaa(\nu_c) = 7/3& \text{for }\nu = \nu_c\\
y_\nu > 2 \text{ and }\alphaa(\nu) = 5/2& \text{for }\nu > \nu_c
\end{cases}
\end{equation*}

\end{prop}
\begin{proof}
From the expression of the generating series $F(\nu,t_\nu,z)$ of the weights $q_k(\nu,t_\nu)$, we can see that the asymptotic behavior of $q_k(\nu,t_\nu)$ as $k$ grows to infinity is driven (up to a change of variable) by the singular behavior (in y) of $Q^+(\nu,t_\nu,y)$. 

We established in Proposition~\ref{prop:analyticQty} that, $y\mapsto Q^+(\nu,t_\nu,ty)$ has two singularities, namely $y_\nu=y_+(\nu,t_\nu)$ and $y_-(\nu,t_\nu)$. The possible singularities $z_\nu$ and $z_-(\nu)$ of $z\mapsto F(\nu,t_\nu,z)$ are then given by:  
\begin{equation} \label{eq:znu}
\begin{cases}
\dfrac{1}{1 - \sqrt{\nu t_\nu^3} z_\nu} = y_\nu \, &\Leftrightarrow \, z_\nu = \dfrac{y_\nu - 1}{\sqrt{\nu t_\nu^3} \, y_\nu}\\
\dfrac{1}{1 - \sqrt{\nu t_\nu^3} z_-(\nu)} = y_-(\nu,t_\nu) \, &\Leftrightarrow \, z_-(\nu) = \dfrac{y_-(\nu,t_\nu) - 1}{\sqrt{\nu t_\nu^3} \, y_-(\nu,t_\nu)}.
\end{cases}
\end{equation}
Since for any $\nu>0$, $y_\nu >0$ and $y_-(\nu,t_\nu)<0$, we have $0<z_\nu<z_-(\nu)$, so that $z_\nu$ is the unique dominant singularity of $F(\nu,t_\nu,z)$.
In addition, from the singular expansion of $y \mapsto Q^+(\nu,t_\nu, t_\nu \, y)$ at $y_\nu$, we see that the dominant singular term in the expansion of $z \mapsto F(\nu,t_\nu ,z)$ at $z_\nu$ is of the form
\begin{align*}
(y_\nu -1)^{\alphaa(\nu)} \, \aleph^{Q^+} (\nu) \, \left(1 - \frac{z}{z_\nu} \right)^{\alphaa(\nu) -1},
\end{align*}
and the classical transfer theorem of Appendix~\ref{sec:algebraic} gives the statement.
\end{proof}

\bigskip

Next we establish the asymptotic behavior of the coefficients in $t$ of the weight $q_k(\nu,t)$ for fixed $k$:

\begin{prop}\label{prop:asymptQk}
For every $k\geq 1$, we have: 
\begin{equation}\label{eq:asymQk}
\frac{[t^{3n}]\left((\nu t^3)^{-k/2}q_{k} (\nu ,t)\right)}{[t^{3n}] \mathcal Z ^\ps ( \nu ,t)}
\underset{n \to \infty}{\rightarrow} \delta_k(\nu),
\end{equation}
where the generating series of the numbers $\delta_k(\nu)$ is given by
\[
\Delta(\nu,z) = \sum_k \delta_k(\nu) \, z^k = \frac{1}{ \beth^{\mathcal Z^\ps}(\nu) } \, \frac{ z}{1 -  z} \beth^{Q^+} \left( \nu, \frac{1}{1-  z}\right).
\]
Moreover, we recall that $\beth^{Q^+} (\nu,y)$, defined in Proposition \ref{prop:serQty}, seen as a series in $y$, has an explicit rational parametrization in terms of $V(\nu,U_\nu,y)$.
\end{prop}
\begin{proof}
For $k\geq 1$, we set: 
\[
    \tilde q_k(\nu,t) = (\nu t^3)^{-k/2}\left(q_{k} (\nu ,t)-(\nu t)^{3/2}\mathbf{1}_{k=3}\right)
\]
It follows directly from the definition of the weights $q_k(\nu,t)$ that, for every $k$, the modified weight $\tilde q_k(\nu , t)$ is a series in $t^3$ with nonnegative coefficients and radius of convergence $t_\nu^3$. By Pringsheim's theorem, $\tilde q_k(\nu , t)$ is hence singular at $t_\nu^3$ and we will prove that this is its unique dominant singularity.

From Proposition~\ref{prop:q_k} and the expression of the generating series $F(\nu,t,z)$ given in~\eqref{eq:genqk}, we can write: 
\begin{equation}\label{eq:Ftilde}
\tilde F( \nu , t ,z ) = \sum_{k \geq 1} \tilde q_k(\nu,t) z^{k-1}
= 
\frac{1}{1-  z} \,
Q^+ \left(\nu, t , \frac{t}{1-z} \right).
\end{equation}
From this expression and the algebraic equation~\eqref{eq:eqAlgQplus} satisfied by $Q^+$, $\tilde F$ can be expressed as the solution of an algebraic equation involving $\nu$, $t^3$ and $tZ_1^+(\nu,t)$, see the Maple file~\cite{Maple} for details. Note that, by definition, the constant coefficient of $\tilde F$ (as a formal power series in $z$) is equal to $Q^+(\nu,t,t)$. The algebraic equation satisfied by $\tilde F$ takes the following form: There exists two explicit polynomials $\mathrm{Pol}_1$ and $\mathrm{Pol}_2$ such that:
\begin{align} 
\left(\tilde F( \nu , t ,z ) -  Q^+(\nu, t,t) \right) & \cdot \mathrm{Pol}_1 \left(\tilde F( \nu , t ,z ) , Q^+(\nu,t,t) , \nu ,t^3 , t Z_1^+(\nu,t) \right) \notag\\
& \qquad = 
z \cdot \mathrm{Pol}_2 \left(z, \tilde F( \nu , t ,z ) , \nu ,t^3 , t Z_1^+(\nu,t) ) \right). \label{eq:tildeFeq}
\end{align}
Furthermore, $\mathrm{Pol}_1 \left(Q^+(\nu, t,t) , Q^+(\nu, t,t) , \nu ,t^3 , t Z_1^+(\nu,t) \right)$ is the derivative with respect to $Q^+(\nu,t,t)$ of the algebraic equation satisfied by $Q^+(\nu,t,t)$ and can be $0$ only at the singularities (in $t^3$) of $Q^+(\nu,t,t)$, namely only for $t^3 = t_\nu^3$. This implies that the coefficients $\tilde q_k(\nu,t)$ can be iteratively calculated from \eqref{eq:tildeFeq}. In addition, each $\tilde q_k(\nu,t)$ has the following expression:
\[
\tilde q_k(\nu,t) = \frac{\mathrm{P}_k \left(Q^+(\nu,t,t) , \nu ,t^3 , t Z_1^+(\nu,t) , (\tilde q_l(\nu,t))_{l<k}\right)}{\left(\mathrm{Pol}_1 \left(Q^+(\nu,t,t) , Q^+(\nu,t,t) , \nu ,t^3 , t Z_1^+(\nu,t) \right) \right)^k}
\]
where $(\mathrm P_k)$ is a sequence of polynomials, which can be computed by induction. Since the denominator of this expression cannot be $0$ for $|t^3| \leq t_\nu^3$ and $t^3 \neq t_\nu^3$, this implies that for every $k$, $\tilde q_k(\nu,t)$ seen as a series in $t^3$ is algebraic and has a unique dominant singularity at $t^3_\nu$ (recall that we already established that $t^3_\nu$ is a singularity).

By the standard singular behavior of algebraic functions near their singularities (see for instance \cite[Theorem VII.7 page 198]{FS}), we know that $[t^{3n}] \tilde q_k(\nu,t)$ has an asymptotic behavior of the form $\mathrm{Cst}(k) \, t_\nu^{-3n} \, n^{r_k(\nu)}$, where $r_k(\nu)\in \mathbb{Q}$ for any $\nu$. We can easily squeeze $[t^{3n}] \tilde q_k(\nu,t)$ between two bounds of the same order to identify $r_k$. Indeed from the definition of the weights $q_k(\nu,t)$:
\[
[t^{3n}] t Z_1^+(\nu,t) \leq [t^{3n}] \tilde q_k(\nu,t) \leq [t^{3(n-k)}] \mathcal{Z}^{\ps}(\nu,t).
\]
This implies that for every $k$ we have $r_k(\nu) = -5/2$ for any $\nu \neq \nu_c$ and $r_k(\nu_c) = -7/3$, i.e. for any $k$, $r_k$ is equal to the function $\gamma$ defined in \eqref{eq:gamma} of Proposition~\ref{prop:asymptoZp+}.
This finishes to establish the asymptotic behavior of $[t^{3n}] \tilde q_k(\nu,t)$ given in \eqref{eq:asymQk}. It remains to identify the generating series of the numbers $\delta_k(\nu)$.

\bigskip

For every $k$, the asymptotic expansion of $\tilde q_k(\nu, t)$ is of the form:
\begin{equation}\label{eq:devtildeq}
\tilde q_k(\nu,t) = \tilde q_k(\nu,t_\nu) + \beth^{\tilde q_k}_1(\nu) \left( 1- (t/t_\nu)^3 \right) +  \beth^{\tilde q_k}(\nu) \left( 1- (t/t_\nu)^3 \right)^{\gamma(\nu)-1} +  o\left( 1- (t/t_\nu)^3 \right)^{\gamma(\nu)-1}.
\end{equation}
Recall from Proposition~\ref{prop:asymptoZmono} that we also have 
\[
\mathcal Z^\ps (\nu,t) = \mathcal Z^\ps (\nu,t_\nu) + \beth^{\mathcal Z^\ps}_1(\nu) \left( 1- (t/t_\nu)^3 \right) +  \beth^{\mathcal Z^\ps}(\nu) \left( 1- (t/t_\nu)^3 \right)^{\gamma(\nu)-1} +  o\left( 1- (t/t_\nu)^3 \right)^{\gamma(\nu)-1}.
\]
Therefore, for every $k \geq 1$:
\[
\delta_k (\nu) = \frac{\beth^{\tilde q_k}(\nu)}{\beth^{\mathcal Z^\ps}(\nu)}.
\]
We define the formal power series in $z$:
\[
\beth^{\tilde q}_1(\nu,z) = \sum_{k \geq 1} \beth^{\tilde q_k}_1(\nu) \, z^k \quad \text{and} \quad \beth^{\tilde q} (\nu,z) = \sum_{k \geq 1} \beth^{\tilde q_k}(\nu) \, z^k.
\]
By definition, the generating series $\Delta (\nu,z)$ of the lemma is the formal power series $\beth^{\tilde q} (\nu,z) / \beth^{\mathcal Z^\ps}(\nu)$. Therefore, we want to show that, as a generating series in z, we have
\[
{\beth^{\tilde q} (\nu,z) } =  
\frac{ z}{1 -  z} \, \beth^{Q^+} \left( \nu, \frac{1}{1-  z}\right) .
\]
Informally, it boils down to proving that the development of $Q^+$ coincides with the sum of the developments of the $\tilde q_k$. We cannot sum them directly since the error term in~\eqref{eq:devtildeq} is \emph{a priori} not uniform in $k$. However, by an approach similar to the proof of \cite[Lemma 5.1]{IsingAMS}, we can establish this result.

Indeed, for every $t$ such that $|t| \leq |t_\nu|$, the radius of convergence in $y$ of $Q^+(\nu, t,ty)$ is at least $y_\nu \geq 2$. Fix $z$ such that $|1/(1-z)| < y_\nu$, or equivalently $|z|<1 - 1/y_\nu$. From Proposition \ref{prop:serQty}, the series $\tilde F(\nu,t,z)$ seen as a series in $t^3$ is analytic, has radius of convergence $t_\nu^3$, and has an asymptotic expansion at $t_\nu^3$ given by:
\begin{align*}
\tilde F(\nu,t,z) &= \frac{1}{1-z} Q^+ \left(t_\nu,\frac{t_\nu}{1-z} \right) + \frac{1}{1-z} \, \beth^{Q^+}_1 \left(\nu,\frac{1}{1-z} \right) \left( 1- (t/t_\nu)^3 \right)\\
& \quad +  \frac{1}{1-z} \, \beth^{Q^+} \left(\nu,\frac{1}{1-z} \right) \left( 1- (t/t_\nu)^3 \right)^{\gamma(\nu)-1} +  o\left( 1- (t/t_\nu)^3 \right)^{\gamma(\nu)-1},
\end{align*}
where $\beth^{Q^+}_1 (\nu,y)$ and $\beth^{Q^+} (\nu,y)$ are explicit rational functions of $\nu$ and $ V(\nu,U(t^3_\nu),y)$. Since $V(\nu, U(t^3_\nu, 1/(1-z))$ is analytic for $|z|<1 - 1/y_\nu$, so are $\beth^{Q^+}_1(\nu,1/(1-z))$ and $\beth^{Q^+}(\nu,1/(1-z))$. In particular for $|z|<1 - 1/y_\nu$, the applications $z\mapsto \beth^{Q^+}_1(\nu,1/(1-z))$ and $z \mapsto\beth^{Q^+}(\nu,1/(1-z))$ are equal to their respective developments in power series.
Even if the error term in the expansion of $\tilde F(\nu,t,z))$ is not uniform in $z$, the power functions $\beth^{Q^+}_1$ and $\beth^{Q^+}$ can be seen as the following limits as formal power series:
\begin{align*}
\frac{1}{1-z} \, \beth^{Q^+}_1 \left( \nu,\frac{1}{1-z} \right) &= \lim_{t \to t_\nu^-} \left(\tilde F(\nu,t,z) - \frac{1}{1-z} Q^+ \left(t_\nu,\frac{t_\nu}{1-z} \right) \right) \cdot \left( 1- (t/t_\nu)^3 \right)^{-1}, \\
\frac{1}{1-z} \, \beth^{Q^+} \left( \nu,\frac{1}{1-z} \right) &= \lim_{t \to t_\nu^-} \left(\tilde F(\nu,t,z) - \frac{1}{1-z} Q^+\left(t_\nu,\frac{t_\nu}{1-z} \right) - \beth^{Q^+}_1 \left( \nu,\frac{1}{1-z} \right) \left( 1- (t/t_\nu)^3 \right) \right) \\
& \qquad \qquad \qquad  \times \left( 1- (t/t_\nu)^3 \right)^{1-\gamma(\nu)}.
\end{align*}
As a consequence,
\[
\beth^{\tilde q}_1 (\nu,z)  =  \frac{ z}{1 -  z} \, \beth^{Q^+}_1 \left( \nu, \frac{1}{1-  z}\right)
\quad \text{and} \quad 
\beth^{\tilde q} (\nu,z)  =  \frac{ z}{1 -  z} \, \beth^{Q^+} \left( \nu, \frac{1}{1-  z}\right)
\]
in their respective domains of analycity, which includes the open disk $|z|<1 - 1/y_\nu$.
This finishes the proof.
\end{proof}

\bigskip

We get as an immediate corollary the following asymptotic behavior of the numbers $\delta_k(\nu)$ in the regime $\nu \leq \nu_c$.

\begin{coro} \label{coro:asymptodeltak}
For $\nu \leq \nu_c$ one has
\[
\delta_k (\nu) \underset{k \to \infty}{\sim}
\frac{1}{\beth^{\mathcal Z^\ps}(\nu)}
\frac{\aleph_{4/3}^{\beth^{Q^+}}(\nu)}{2^{4/3} \, \Gamma(4/3)} \,
2^{k} k^{1/3}.
\]
\end{coro}
\begin{proof}
This is a direct consequence of the expression of $\Delta$ in terms of $\beth^{Q^+}$, of Proposition~\ref{prop:asymptalephQ+} giving the singular behavior of $\beth^{Q^+}$ and of the transfer theorem of Appendix~\ref{sec:algebraic}.
\end{proof}

\subsection{Computations around the BDG functions \texorpdfstring{$f^\bullet_{\mathbf{q}(\nu,t)}$}{fdot} and \texorpdfstring{$f^\diamond_{\mathbf q(\nu,t)}$}{fdiamond}}

The purpose of this section is to establish some preliminary and technical results that will be useful to compute the volume generating function of the root spin cluster of random Ising-weighted triangulations. 

The main tool at our disposal to compute this generating function is its expression in terms of $z^+(\mathbf q, g)$ and $z^\diamond(\mathbf q, g)$ given in \eqref{eq:Zg}. In order to be able to use this expression, we need to calculate the development of $z^+(\mathbf q (\nu,t_\nu), g)$ and $z^\diamond(\mathbf q(\nu,t_\nu), g)$ around $z^+_\nu$ and $z^\diamond_{\nu}$ respectively (or, in other words, when $g$ goes to 1).

We will proceed in three steps. In a first section, we obtain an explicit integral formula for $f^\bullet_{\mathbf{q}(\nu,t)}$ and $f^\diamond_{\mathbf{q}(\nu,t)}$. Thanks to this expression, we can compute the singular development of these functions around $(z^+_\nu,z^\diamond_\nu)$. Finally, we use this development together with the fact that $z^+(\mathbf q (\nu,t_\nu), g)$ and $z^\diamond(\mathbf q(\nu,t_\nu), g)$ are solutions of a system of equations involving $f^\bullet_{\mathbf{q}(\nu,t)}$ and $f^\diamond_{\mathbf{q}(\nu,t)}$, to get the desired expansions.

\subsubsection{Integral expressions for \texorpdfstring{$f^\bullet_{\mathbf{q}(\nu,t)}$}{fdot} and \texorpdfstring{$f^\diamond_{\mathbf q(\nu,t)}$}{fdiamond}}

The functions $f^\bullet_{\mathbf q}$ and $f^\diamond_{\mathbf q}$ defined in~\eqref{eq:deffbullet} and~\eqref{eq:deffdiamond} of Section \ref{sec:BDGgen} are usually hard to compute for non bipartite maps. However, in our setting, we are able to obtain an integral expression involving $Q^+$ for $f^\bullet_{\mathbf q(\nu,t)}$ and $f^\diamond_{\mathbf q(\nu,t)}$ :
\begin{prop} \label{prop:fbulletdiamond}
Fix $\nu >0$ and $t \in (0, t_\nu]$. For any $z_1 \in [0,z^+(\nu,t)]$ and $z_2 \in [0, z^\diamond(\nu,t)]$ we have
\begin{align}
f^\bullet_{\mathbf q (\nu,t)} (z_1,z_2) &= (\nu t)^{3/2} \, 2z_2 + \frac{\nu t^3}{\pi}
\int_{{1}/{a_-}}^{{1}/{a_+}} \,
\frac{1 - \sqrt{\nu t^3} z_2 - \frac{1}{z}}{2 \nu t^3 z_1}
\,
\frac{Q^+(\nu,t,tz)}
{\sqrt{
\left(1-z a_+ \right)\left(z a_- -1 \right)
}} \,  \mathrm{d}z,\\
f^\diamond_{\mathbf q (\nu,t)} (z_1,z_2) &= (\nu t)^{3/2} \, (2z_1 + z_2^2) + \frac{\sqrt{\nu t^3}}{\pi}
\int_{1/a_-}^{1/a_+} \,
\frac{Q^+(\nu,t,tz) }
{\sqrt{
\left(1-z a_+ \right)\left(z a_- - 1 \right)
}}\,  \mathrm{d}z ,
\end{align}
where we set:
\begin{align}
a_+ := 1-\sqrt{\nu t^3} (z_2 + 2 \sqrt{z_1}) \quad \text{and} \quad
a_- := 1-\sqrt{\nu t^3} (z_2 - 2 \sqrt{z_1}).
\end{align}
\end{prop}

\begin{proof}
First, plugging the value of the weights $q_l(\nu,t)$ in the definition of the two functions gives the following identities:
\begin{align}
f^\bullet_{\mathbf q (\nu,t)} (z_1,z_2) &= 2 (\nu t)^{3/2} z_2 + \nu t^3 \, \sum_{k,k',l \geq 0}
\binom{2k+k'+l +1}{k+1,k,k',l} \cdot
t^l Q_l^+(\nu,t) \cdot 
(\nu t^3 \, z_1)^k ( \sqrt{\nu t^3} z_2)^{k'}, \\
f^\diamond_{\mathbf q (\nu,t)} (z_1,z_2) &=  (\nu t)^{3/2} (2z_1+z_2^2) + \sqrt{\nu t^3} \, \sum_{k,k',l \geq 0}
\binom{2k+k'+l}{k,k,k',l} \cdot
t^l Q_l^+ (\nu,t) \cdot
(\nu t^3 \, z_1)^k ( \sqrt{\nu t^3} z_2)^{k'}.
\end{align}

Fix $\nu >0$, $t \in (0,t_\nu]$, and then $z_1 \in (0,z^+(\nu,t)]$ and $z_2 \in (0,z^\diamond(\nu,t)]$. The sums appearing in the previous displays can be expressed as Hadamard products, see Appendix~\ref{sec:Hadamard}. Indeed, define the two formal power series $g^\bullet(z)$ and $g^\circ(z)$ as:
\begin{align*}
g^\bullet(z) &= \sum_{k,k',l \geq 0}
\binom{2k+k'+l +1}{k+1,k,k',l}
(\nu t^3 \, z_1)^k ( \sqrt{\nu t^3} z_2)^{k'} \cdot z^l,\\
g^\diamond(z) &= \sum_{k,k',l \geq 0}
\binom{2k+k'+l}{k,k,k',l}
(\nu t^3 \, z_1)^k ( \sqrt{\nu t^3} z_2)^{k'} \cdot z^l.
\end{align*}
Then
\begin{align}
f^\bullet_{\mathbf q (\nu,t)} (z_1,z_2) &= 2 (\nu t)^{3/2} z_2 + \nu t^3 \, 
Q^+(\nu,t,tz) \odot g^\bullet(z) \vert_{z=1}
, \\
f^\diamond_{\mathbf q (\nu,t)} (z_1,z_2) &=  (\nu t)^{3/2} (2z_1+z_2^2) + \sqrt{\nu t^3} \,
Q^+(\nu,t,tz) \odot g^\diamond(z) \vert_{z=1}.
\end{align}
Now let us compute the functions $g^\bullet$ and $g^\diamond$. 
We first rewrite: 
\[
g^\bullet(z) = \sum_{k,k' \geq 0}
\binom{2k+k'+1}{k+1,k,k'}
(\nu t^3 \, z_1)^k ( \sqrt{\nu t^3} z_2)^{k'} \sum_{l\geq 0}\binom{2k+k'+l+1}{l}z^l.
\]
Then, if $|z|<1$, summing over $l$ yields
\[
g^\bullet(z) = \sum_{k,k' \geq 0}
\binom{2k+k'+1}{k+1,k,k'}
(\nu t^3 \, z_1)^k ( \sqrt{\nu t^3} z_2)^{k'} \cdot (1-z)^{-(2k+k'+2)}.
\]
Similarly, if $|z| < 1 - \sqrt{\nu t^3} z_2$, which guarantees $|\sqrt{\nu t^3} z_2/(1-z)|<1$, summing over $k'$ yields
\[
g^\bullet(z) = \sum_{k \geq 0}
\binom{2k+1}{k+1}
(\nu t^3 \, z_1)^k \left(1-z-\sqrt{\nu t^3} z_2 \right)^{-(2k+2)}.
\]
Finally, for $|z| < 1 - \sqrt{\nu t^3} (z_2 + 2 \sqrt{z_1})$, which guarantees that $|\nu t^3 \, z_1/(1-z-\sqrt{\nu t^3} z_2)^2|<1/4$, we get
\begin{align*}
g^\bullet(z) &= \frac{1}{2 \nu t^3 \, z_1} \left( \left(1 - 4 \frac{\nu t^3 \,z_1 }{(1-z-\sqrt{\nu t^3} \, z_2)^2}\right)^{-1/2} - 1\right),\\
&= \frac{1}{2 \nu t^3 \, z_1} \left( \frac{1-z-\sqrt{\nu t^3}z_2}{\sqrt{\left(1-z-\sqrt{\nu t^3} (z_2 + 2 \sqrt{z_1})\right) \left(1-z-\sqrt{\nu t^3} (z_2 - 2 \sqrt{z_1})\right) }} - 1\right).
\end{align*}
A similar sequence of computations implies that for $|z| < 1 - \sqrt{\nu t^3} (z_2 + 2 \sqrt{z_1})$:
\[
g^\diamond(z) = \frac{1}{\sqrt{\left(1-z-\sqrt{\nu t^3} (z_2 + 2 \sqrt{z_1})\right) \left(1-z-\sqrt{\nu t^3} (z_2 - 2 \sqrt{z_1})\right) }}.
\]
Note that the region $|z| < 1 - \sqrt{\nu t^3} (z_2 + 2 \sqrt{z_1})$ is not empty since $\sqrt{\nu t^3} (z_2 + 2 \sqrt{z_1}) \leq \sqrt{\nu t^3} c_+(\nu,t) = \frac{1}{y_+(\nu,t)} \leq \frac{1}{2}$.

\bigskip

The Hadamard product can be computed with an integral formula (see for instance~\cite[Section VI.10.2]{FS} or Appendix~\ref{sec:Hadamard}). This yields:
\[
Q^+(\nu,t,tz) \odot g^\bullet(z)\vert_{z=1} = \frac{1}{2 i \pi} \oint_{\gamma} Q^+(\nu,t,tw) \,  g^\bullet \left( \frac{1}{w} \right) \frac{\mathrm{d} w}{w},
\]
and
\[
Q^+(\nu,t,tz) \odot g^\diamond(z)\vert_{z=1} = \frac{1}{2 i \pi} \oint_{\gamma} Q^+(\nu,t,tw) \,  g^\diamond \left( \frac{1}{w} \right) \frac{\mathrm{d} w}{w},
\]
where the integration contour $\gamma$ in the $w$-plane should be chosen in such a way that both factors $Q^+(\nu,t,tw)$ and $g^\bullet \left( \frac{1}{w} \right)$ (or $g^\diamond \left( \frac{1}{w} \right)$) are analytic in an open neighborhood of $\gamma$.

\begin{figure}[!t]
\centering
\includegraphics[width=0.8\linewidth]{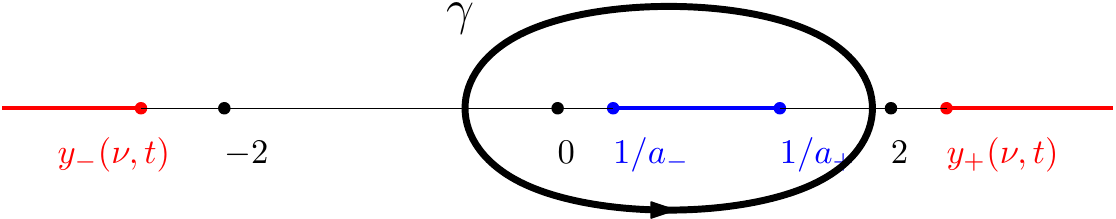}
\caption{\label{fig:contour}Integration contour for the Hadamard product.}
\end{figure}
By the above computations, $g^\bullet \left( \frac{1}{w} \right)$ and $g^\diamond \left( \frac{1}{w} \right)$ are analytic for 
\[
w \in \mathbb C \setminus \Big[ \frac{1}{1-\sqrt{\nu t^3} (z_2 - 2 \sqrt{z_1})} , \frac{1}{1-\sqrt{\nu t^3} (z_2 + 2 \sqrt{z_1})}\Big].\] Furthermore, Proposition \ref{prop:analyticQty} ensures that $Q^+(\nu,t,tw)$ is analytic for $w \in \mathbb C \setminus \{(-\infty, y_-(\nu,t)] \cup [y_+(\nu,t) , + \infty) \}$. For $z_2+2\sqrt{z_1} < c_+(\nu,t)$, a suitable contour of integration then exists. Indeed, recalling that $y_+(\nu,t) \geq 2$, we have in this case
\[
\frac{1}{1-\sqrt{\nu t^3} (z_2 + 2 \sqrt{z_1})} < \frac{1}{1-\sqrt{\nu t^3} c_+(\nu,t)} = \frac{y_+(\nu,t)}{y_+(\nu,t) -1 } \leq 2,
\]
and we also always have
\[
\frac{1}{1-\sqrt{\nu t^3} (z_2 - 2 \sqrt{z_1})} > \frac{1}{1+\sqrt{\nu t^3} c_+(\nu,t)} = \frac{y_+(\nu,t)}{y_+(\nu,t) + 1 } \geq \frac{2}{3}.
\]
We can thus take $\gamma$ inside the disk of convergence of $Q^+(\nu,t,tw)$ and enclosing a cut $[1/(1-\sqrt{\nu t^3} (z_2 - 2 \sqrt{z_1})), 1/(1-\sqrt{\nu t^3} (z_2 + 2 \sqrt{z_1}))]$, see Figure~\ref{fig:contour}. Taking the limit when the curve $\gamma$ approaches the cut then gives the expressions of the proposition for $z_1 \in [0,z^+(\nu,t))$ and $z_2 \in [0, z^\diamond (\nu,t))$. The equality for $(z_1,z_2) = (z^+(\nu,t),z^\diamond (\nu,t))$ follows by taking the limit.
\end{proof}

\subsubsection{Singular developments of \texorpdfstring{$f^\bullet_\nu$}{fdot} and \texorpdfstring{$f^\diamond_\nu$}{fdiamond} around \texorpdfstring{$(z^+_\nu , z^\diamond_\nu)$}{(z+,zd)}}\label{sub:devfBDG}

Denote by $(z^+_\nu , z^\diamond_\nu) := (z^+(\mathbf q (\nu,t_\nu)) ,z^\diamond(\mathbf q (\nu,t_\nu)))$ the solution of the system of equations \eqref{eq:fixpoint}. Recall from~\eqref{eq:z+y+-} and~\eqref{eq:zdy+-} that there exist explicit expressions for $z^+_\nu , z^\diamond_\nu$ in terms of the singularities $y_+(\nu,t_\nu)$ and $y_-(\nu,t_\nu)$ of the function $y \mapsto Q^+(\nu,t_\nu,t_\nu \, y)$, which are themselves explicit functions of $\nu$. 

Using the expressions for $f^\bullet$ and $f^\diamond$ given in Proposition \ref{prop:fbulletdiamond} and the singular behavior of $y \mapsto Q^+(\nu,t_\nu,t_\nu \, y)$ obtained in Proposition~\ref{lemm:weightsclusters}, we can compute explicitly an asymptotic expansion of the BDG functions $f^\bullet$ and $f^\diamond$ at $(z^+_\nu , z^\diamond_\nu)$:

\begin{lemm} \label{lem:asymptfbfg}
Fix $\nu \in (0,\nu_c]$ and write $f^\bullet_\nu = f^\bullet_{\mathbf q (\nu,t_\nu)}$, $f^\diamond_\nu = f^\diamond_{\mathbf q (\nu,t_\nu)}$, $z^+_\nu = z^+(\mathbf q (\nu,t_\nu))$ and $z^\diamond_\nu = z^\diamond(\mathbf q (\nu,t_\nu))$. We have the following asymptotic expansions at $(z^+_\nu,z^\diamond_\nu)^-$:
\begin{align*}
f^\bullet_\nu(z_1,z_2) &= 1 - \frac{1}{z^+_\nu} 
- \frac{1}{z^+_\nu} \left(\frac{1}{z^+_\nu} - \sqrt{z^+_\nu} 
\partial_{z_1}f^\diamond_\nu(z^+_\nu,z^\diamond_\nu) \right)
\left( z^+_\nu - z_1 \right) 
- \partial_{z_1}f^\diamond_\nu(z^+_\nu,z^\diamond_\nu) \left( z^\diamond_\nu - z_2 \right) \\
& \qquad + \aleph^\bullet(\nu) \left( (z^\diamond_\nu - z_2) + \frac{1}{\sqrt{z^+_\nu}} (z^+_\nu - z_1) \right)^{\alphaa(\nu) - 1/2}
+ \mathcal O \left( (z^\diamond_\nu - z_2)^2 + (z^+_\nu - z_1)^2 \right),\\
f^\diamond_\nu(z_1,z_2) &= z^\diamond_\nu - \partial_{z_1}f^\diamond_\nu(z^+_\nu,z^\diamond_\nu) \left( z^+_\nu - z_1 \right) - \left( 1 - \sqrt{z^+_\nu} \, \partial_{z_1}f^\diamond_\nu(z^+_\nu,z^\diamond_\nu)\right) \left( z^\diamond_\nu - z_2 \right) \\
& \qquad + \aleph^\diamond(\nu) \left( (z^\diamond_\nu - z_2) + \frac{1}{\sqrt{z^+_\nu}} (z^+_\nu - z_1) \right)^{\alphaa(\nu) - 1/2}
+ \mathcal O \left( (z^\diamond_\nu - z_2)^2 + (z^+_\nu - z_1)^2 \right),
\end{align*}
where $\partial_{z_1}f^\diamond_\nu(z^+_\nu,z^\diamond_\nu) \in (0,1)$, $\aleph^\diamond(\nu),\aleph^\bullet(\nu) > 0$ are explicit functions of $\nu$ and where we recall that
\begin{equation}
\alphaa(\nu) =
\begin{cases}
5/3 & \text{for $\nu < \nu_c$,}\\
7/3 & \text{for $\nu = \nu_c$.}
\end{cases}
\end{equation}
\end{lemm}

\begin{rema}
Our proof of Lemma \ref{lem:asymptfbfg} can also be used to show that the functions $f^\bullet_\nu$ and $f^\diamond_\nu$ are analytic in a neighborhood of $(z^+_\nu,z^\diamond_\nu)$ when $\nu > \nu_c$. We do not include this part here as we do not use that fact, and also because it is already a consequence (see for example Miermont \cite{MiermontInvariance}) of the fact that the weight sequence $\mathbf q (\nu,t_\nu)$ is regular critical when $\nu >\nu_c$ since the root face of the associated random map has exponential tails (see Proposition~\ref{prop:regNonReg}).
\end{rema}

\begin{proof}[Proof of Lemma \ref{lem:asymptfbfg}]
Fix $\nu \in (0,\nu_c]$ and $(z_1,z_2) \in (0,z^+_\nu] \times (0,z^\diamond_\nu]$.
We start with the expressions for $f^\bullet$ and $f^\diamond$ given in Proposition \ref{prop:fbulletdiamond}. Recall the notation $a_+ = 1  - \sqrt{\nu t_\nu^3} (z_2 + 2 \sqrt{z_1})$ and $a_- = 1  - \sqrt{\nu t_\nu^3} (z_2 - 2 \sqrt{z_1})$.
The change of variable
\[
z= \phi(z_1,z_2;\xi) := \frac{1}{a_-} + \left(\frac{1}{a_+} -  \frac{1}{a_-} \right) \xi
\]
in the expressions of $f_{\nu}^\bullet$ and $f_{\nu}^\diamond$ give:
\begin{align*}
f_{\nu}^\bullet &(z_1,z_2) \\
&= (\nu t_\nu)^{3/2} \, 2 z_2 + \frac{1}{\pi \sqrt{a_+a_-}} \, \int_0^1 \,
\frac{1 - \sqrt{\nu t^3_\nu} - \frac{1}{\phi(z_1,z_2;\xi)}}{2 z_1}
\, \xi^{-1/2} (1-\xi)^{-1/2} Q^+ \left(\nu, t_\nu,t_\nu
\phi(z_1,z_2;\xi) \right) d\xi,\\
f_{\nu}^\diamond &(z_1,z_2) \\
&= (\nu t_\nu)^{3/2} (2z_1 + z_2^2) + \frac{\sqrt{\nu t_\nu^3}}{\pi \sqrt{a_+a_-}} \, \int_0^1 \, \xi^{-1/2} (1-\xi)^{-1/2} Q^+ \left(\nu, t_\nu,t_\nu
\phi(z_1,z_2;\xi) \right) d\xi.
\end{align*}
\bigskip

Note that for $(z_1,z_2)\in (0,z_\nu^+]\times(0,z_\nu^\diamond)$ and $\xi\in [0,1]$, we have: 
\begin{equation} \label{eq:rangephi}
\phi(z_1,z_2;\xi) \in \left[ \frac{1}{a_-} ,  \frac{1}{a_+}  \right] \subset \left[ \frac{1}{1  + 2\sqrt{\nu t_\nu^3} z^+_\nu} ,  \frac{1}{1-1/y_+(\nu,t_\nu)}  \right] \subset (0,2],
\end{equation}
and $\phi(z_1,z_2;\xi)=2$, if and only if $(z_1,z_2,\xi)=(z_\nu^+,z_\nu^\diamond,1)$. 
Moreover recall that, since $\nu \leq \nu_c$, the dominant singularity $y_+(\nu,t_\nu)$ of $y\mapsto Q^+(\nu,t_\nu,t_\nu y)$ is equal to 2. 

The general strategy to compute the singular expansion of these integrals is then the following. 
We are first going to prove that, if we remove from $Q^+$ its first singular terms in its development at $y=2$, the first terms of the asymptotic development of the integral are linear and quadratic terms.
Then, we prove that the singular part in the development is driven by the singular part in the asymptotic behavior of $y\mapsto Q^+(\nu,t_\nu,t_\nu y)$ around 2 and calculate it. 

\bigskip

We start with $f^\diamond$ as the calculations are slightly simpler in this case. From Proposition~\ref{lemm:weightsclusters}, we know that $y \mapsto Q^+(\nu,t_\nu,t_\nu y)$ is analytic for $|y| < 2$ and we can compute the following expansion as $y \to 2^-$:
\begin{equation}\label{eq:devABC}
\begin{split}
Q^+(\nu,t_\nu,t_\nu y)
= &
Q^+(\nu,t_\nu, 2 t_\nu)\\
&+ \aleph_{2/3}^{Q^+}(\nu) \cdot (2-y)^{2/3} + \aleph_{1}^{Q^+}(\nu) \cdot (2-y) + \aleph_{4/3}^{Q^+}(\nu) \cdot (2-y)^{4/3}\\
&+ \aleph_{5/3}^{Q^+}(\nu) \cdot (2-y)^{5/3} + \aleph_{2}^{Q^+}(\nu) \cdot (2-y)^{2} + o \left( (2-y)^{2} \right)
\end{split}
\end{equation}
where we recall that the coefficients $\aleph_{i}^{Q^+}$ are explicit functions of $\nu$ such that $\aleph_{2/3}^{Q^+}(\nu) = 0$ if and only if $\nu = \nu_c$ and $\aleph_{4/3}^{Q^+}(\nu) \neq 0$ for $\nu \leq \nu_c$. 

Set
\[
\varphi(y) = Q^+(t_\nu,t_\nu y) -  \aleph_{2/3}^{Q^+}(\nu) \cdot (2-y)^{2/3} - \aleph_{4/3}^{Q^+}(\nu) \cdot (2-y)^{4/3} - \aleph_{5/3}^{Q^+}(\nu) \cdot (2-y)^{5/3}.
\]
In view of~\eqref{eq:rangephi} and \eqref{eq:devABC}, the function $\varphi$ is twice differentiable on $(0,2)$ and we have:
\[
        \int_0^1\xi^{-1/2} (1-\xi)^{-1/2} \left|\varphi'\left( \phi(z_1,z_2;\xi) \right)\right|d\xi <\infty
\]
and
\[
    \int_0^1\xi^{-1/2} (1-\xi)^{-1/2} \left|\varphi''\left( \phi(z_1,z_2;\xi) \right)\right|d\xi <\infty
\] 
Together with the fact that $ \frac{1}{a_-}$ and $\frac{1}{a_+}$ are twice differentiable on $[0,z^+_\nu] \times [0,z^\diamond_\nu]$, this implies that the function $I_\varphi$ defined by:
\[
I_\varphi (z_1,z_2) =  \frac{1}{\pi} \int_0^1 \, \xi^{-1/2} (1-\xi)^{-1/2} \varphi 
\left( \frac{1}{a_-} + \left(\frac{1}{a_+} -  \frac{1}{a_-} \right) \xi \right) d\xi
\]
is twice differentiable on $[0,z^+_\nu] \times [0,z^\diamond_\nu]$. Therefore, as $z_1 \to (z^+_\nu)^-$ and $z_2 \to (z^\diamond_\nu)^-$ we get the non singular expansion:
\[
I_\varphi (z_1,z_2) = I_\varphi(z^+_\nu,z^\diamond_\nu) + \nabla I_\varphi(z^+_\nu,z^\diamond_\nu) \cdot (z_1-z^+_\nu,z_2 - z^\diamond_\nu) + \mathcal O \left((z_1-z^+_\nu)^2 +(z_2 - z^\diamond_\nu)^2 \right).
\]

\bigskip

We now move to the singular part in the development of the function $f^\diamond_{\nu}$. It comes from the terms of the form $(2-y)^{p}$ with $p=2/3,4/3$ and $5/3$ that we subtracted. Indeed, for these values of $p$, their contribution to the integral is equal to:
\begin{align*}
I_p(z_1,z_2) &:=
\frac{1}{\pi} \int_0^1 \, \xi^{-1/2} (1-\xi)^{-1/2}  
\left(2 - \left(\frac{1}{a_-} + \left(\frac{1}{a_+} -  \frac{1}{a_-} \right) \xi \right)\right)^p d\xi, \\
&= \left({2 - \frac{1}{a_-}} \right)^p \, \frac{1}{\pi} \int_0^1
\, \xi^{-1/2} (1-\xi)^{-1/2}  
\left(1 - \left(\frac{\frac{1}{a_+} -  \frac{1}{a_-}}{{2 - \frac{1}{a_-}} } \right) \xi \right)^p d\xi.
\end{align*}
To compute this integral, we rewrite it as the Euler integral representation of a hypergeometric function. Indeed, when $z_1 \to (z^+_\nu)^-$ and $z_2 \to (z^\diamond_\nu)^-$, we have $\frac{1}{a_+} \to 2^-$ and $\frac{\frac{1}{a_+} - \frac{1}{a_-}}{{2 - \frac{1}{a_-}} } \to 1^-$, so that:
\begin{align*}
I_p(z_1,z_2) &=
\left({2 - \frac{1}{a_-}} \right)^p \,
_2 F_1 \left( -p, \frac{1}{2}; 1;
\frac{\frac{1}{a_+} -  \frac{1}{a_-}}{{2 - \frac{1}{a_-}} } \right).
\end{align*}
This allows us to use the singular behavior at $1^-$ of $\,
_2 F_1$, which is of the form: 
\begin{equation}\label{eq:devHypergeom}
_2 F_1 \left( -p, \frac{1}{2}; 1;
u \right) = a_p + b_p \, (1-u) + c_p \left(1- u\right)^{p+1/2} + \mathcal O \left(1- u\right)^2
\end{equation}
as $u \to 1^-$, where $a_p$, $b_p$ and $c_p$ are explicit (see the Maple file \cite{Maple}).

We have $1-\frac{\frac{1}{a_+} - \frac{1}{a_-}}{{2 - \frac{1}{a_-}} }= \frac{2 - \frac{1}{a_+}}{2 - \frac{1}{a_-}}$, with $2 - \frac{1}{a_-} \rightarrow 2-c_-(\nu,t_\nu)>0$ and : 
\begin{align*}
2 - \frac{1}{a_+} = 4 \sqrt{\nu t_\nu^3} \left( (z^\diamond_\nu - z_2) + \frac{1}{\sqrt{z^+}} (z^+_\nu - z_1) \right) + \mathcal O \left( (z^\diamond_\nu - z_2)^2 + (z^+_\nu - z_1)^2 \right),
\end{align*}
as $z_1\rightarrow z_\nu^+$ and $z_2\rightarrow z^\diamond_\nu$. So, that plugging this expression in~\eqref{eq:devHypergeom}, we get:
\begin{multline}
I_p(z_1,z_2) = a'_p + b'_p \, (z^+_\nu - z_1) + b''_p (z^\diamond_\nu - z_2)\\
 + c_p \left(4 \sqrt{\nu t_\nu^3} \left( (z^\diamond_\nu - z_2) + \frac{1}{\sqrt{z^+_\nu}} (z^+_\nu - z_1) \right) \right)^{p+1/2}\\ + \mathcal O \left( (z^\diamond_\nu - z_2)^2 + (z^+_\nu - z_1)^2 \right)
\end{multline}
as $z_1 \to (z^+)^-$ and $z_2 \to (z^\diamond)^-$ and where the constants are explicit. Notice that when $p=5/3$, the singular term of the development is of order higher than $2$ and will not contribute in the final development.

\bigskip

Finally, we can sum the different contributions to the integral and write:
\begin{multline*}
f^\diamond_\nu (z_1,z_2) = 
(\nu t_\nu)^{3/2} (2z_1 + z_2^2) \\
+ \frac{\sqrt{\nu t_\nu^3}}{\sqrt{a_+a_-}} \,
\left(
I_\varphi(z_1,z_2) + \aleph_{2/3}^{Q^+} (\nu) I_{2/3}(z_1,z_2) + \aleph_{4/3}^{Q^+} (\nu) I_{4/3}(z_1,z_2)+ \aleph_{5/3}^{Q^+} (\nu) I_{5/3}(z_1,z_2)
\right).
\end{multline*}
Plugging into this expression the expansions of $I_\varphi$, $I_{2/3}$, $I_{4/3}$, $I_{5/3}$, and recalling that $f^\diamond$ is differentiable at $(z^+,z^\diamond)$, we get
\begin{align*}
f^\diamond_\nu(z_1,z_2) &= f^\diamond(z^+_\nu,z^\diamond_\nu) - \nabla f^\diamond(z^+_\nu,z^\diamond_\nu) \cdot (z^+_\nu - z_1,z^\diamond_\nu - z_2) \\
& \qquad + \aleph^\diamond_{2/3} (\nu) \left( (z^\diamond_\nu - z_2) + \frac{1}{\sqrt{z^+_\nu}} (z^+_\nu - z_1) \right)^{2/3+1/2}
\\
& \qquad + \aleph^\diamond_{4/3} (\nu) \left( (z^\diamond_\nu - z_2) + \frac{1}{\sqrt{z^+_\nu}} (z^+_\nu - z_1) \right)^{4/3+1/2}
\\
& \qquad + \mathcal O \left( (z^\diamond_\nu - z_2)^2 + (z^+_\nu - z_1)^2 \right),
\end{align*}
with
\[
\aleph^\diamond_p(\nu) = \frac{c_p \left(4 \sqrt{\nu t_\nu^3}\right)^{p+1/2} \, \sqrt{\nu t_\nu^3} \, \aleph_p^{Q^+}(\nu)}{\sqrt{\frac{1}{2} \left(1-\frac{1}{y_- \left(\nu,t_\nu \right)}\right)}}.
\]
Note that we have $\aleph^\diamond_{2/3} (\nu_c) = 0$, so we set
\[
\aleph^\diamond(\nu) =
\begin{cases}
\aleph_{2/3}^\diamond(\nu) & \text{for $\nu < \nu_c$,}\\
\aleph_{4/3}^\diamond(\nu_c) & \text{for $\nu = \nu_c$.}
\end{cases}
\]
The expansion for $f^\diamond_\nu$ given in the lemma then follows from the fact that $f^\diamond(z^+_\nu,z^\diamond_\nu) = z^\diamond_\nu$ and from the criticality of the weight sequence $\mathbf q(\nu, t_\nu)$ established in Proposition \ref{prop:Wnut}, which ensures the equality
\[
\partial_{z_2} f^\diamond(z^+_\nu,z^\diamond_\nu) + \sqrt{z^+_\nu} \partial_{z_1} f^\diamond (z^+_\nu,z^\diamond_\nu) = 1.\]

\bigskip

We now move to the development for $f^\bullet$. It is obtained with the same strategy by replacing $\varphi$ by the function
\[
\varphi^\bullet (y) = \frac{1 - \sqrt{\nu t^3_\nu} - \frac{1}{y}}{2 z_1} Q^+(t_\nu,t_\nu y) -  \frac{1 - \sqrt{\nu t^3_\nu} - \frac{1}{2}}{2 z_1} \left( \aleph_{2/3}^{Q^+}(\nu) \cdot (2-y)^{2/3} + \aleph_{4/3}^{Q^+}(\nu) \cdot (2-y)^{4/3} \right).
\]
We can prove as above that the function
\[
I_{\varphi^\bullet} (z_1,z_2) =  \frac{1}{\pi} \int_0^1 \, \xi^{-1/2} (1-\xi)^{-1/2} \varphi^\bullet 
\left( \frac{1}{a_-} + \left(\frac{1}{a_+} -  \frac{1}{a_-} \right) \xi \right) d\xi
\]
is twice differentiable on $[0,z^+_\nu] \times [0,z^\diamond_\nu]$. We can then use the identity
\[
f_{\nu}^\bullet (z_1,z_2)
= (\nu t_\nu)^{3/2} \, 2 z_2 + \frac{1}{\sqrt{a_+a_-}} \, 
I_\varphi(z_1,z_2) + \frac{\frac 1 2 - \sqrt{\nu t^3_\nu}}{2 z_1 \sqrt{a_+a_-}}\,
\left(
  \aleph_{2/3}^{Q^+} (\nu) I_{2/3}(z_1,z_2) + \aleph_{4/3}^{Q^+} (\nu) I_{4/3}(z_1,z_2)
\right)
\]
and the developments of $I_{2/3}$ and $I_{4/3}$ to establish that 
\begin{align*}
f^\bullet_\nu(z_1,z_2) &= f^\bullet(z^+_\nu,z^\diamond_\nu) - \nabla f^\bullet(z^+_\nu,z^\diamond_\nu) \cdot (z^+_\nu - z_1,z^\diamond_\nu - z_2) \\
& \qquad + \aleph^\bullet_{2/3} (\nu) \left( (z^\diamond_\nu - z_2) + \frac{1}{\sqrt{z^+_\nu}} (z^+_\nu - z_1) \right)^{2/3+1/2}
\\
& \qquad + \aleph^\bullet_{4/3} (\nu) \left( (z^\diamond_\nu - z_2) + \frac{1}{\sqrt{z^+_\nu}} (z^+_\nu - z_1) \right)^{4/3+1/2}
\\
& \qquad + \mathcal O \left( (z^\diamond_\nu - z_2)^2 + (z^+_\nu - z_1)^2 \right),
\end{align*}
 with
\[
\aleph^\bullet_p(\nu) = \frac{\frac{1}{2} - \sqrt{\nu t_\nu^3} z^\diamond_\nu}{2 z^+_\nu} \, 
\frac{c_p \left(4 \sqrt{\nu t_\nu^3}\right)^{p+1/2} \, \aleph_p^{Q^+}(\nu)}{\sqrt{\frac{1}{2} \left(1-\frac{1}{y_- \left(\nu,t_\nu \right)}\right)}}.
\]
Similarly as before, we set
\[
\aleph^\bullet(\nu) =
\begin{cases}
\aleph_{2/3}^\bullet(\nu) & \text{for $\nu < \nu_c$,}\\
\aleph_{4/3}^\bullet(\nu_c) & \text{for $\nu = \nu_c$.}
\end{cases}
\]
The statement for $f^\bullet_\nu$ then follows from the criticality of the weight sequence $\mathbf q (\nu,t_\nu)$ that ensures the equality $f^\bullet_\nu (z^+_\nu,z^\diamond_\nu) = 1 - 1/z^+_\nu$, and from the two following generic identities for BDG functions:
\[
\forall (z_1,z_2), \quad \partial_{z_2} f^\bullet  (z_1,z_2)= \partial_{z_1} f^\diamond  (z_1,z_2) \quad \text{and} \quad z_1 \partial{z_1} f^\bullet  (z_1,z_2) + f^\bullet  (z_1,z_2) =  \partial_{z_2} f^\diamond  (z_1,z_2).
\]
\end{proof}

\subsubsection{Singular developments of \texorpdfstring{$z^+_\nu (g)$}{z+(g)} and \texorpdfstring{$z^\diamond_\nu (g)$}{zd(g)}}

For $\nu > 0$ and $g \in (0,1]$, let us denote by $(z^+_\nu (g),z^\diamond_\nu (g))$ the solution in $(0,z^+_\nu] \times (0,z^\diamond_\nu]$ of the system \eqref{eq:fixpointg} when the weight sequence is $\mathbf q (\nu,t_\nu)$. In other words, $(z^+_\nu (g),z^\diamond_\nu (g))$ is the solution of the system of equations:
\begin{equation} \label{eq:zgtcrit}
\begin{cases}
f^\bullet_\nu \Big(z^+_\nu(g),z^\diamond_\nu(g)\Big) &= 1 - \dfrac{g}{z^+_\nu(g))},\\
f^\diamond_\nu \Big(z^+_\nu(g),z^\diamond_\nu(g)\Big) &= z^\diamond_\nu(g),
\end{cases}
\end{equation}
where we recall that $f^\bullet_{\nu}:= f^\bullet_{\mathbf q(\nu,t_\nu)}$ and $f^\diamond_{\nu}:= f^\diamond_{\mathbf q(\nu,t_\nu)}$. 

Using the asymptotic expansions of $f^\bullet_\nu$ and $f^\diamond_\nu$ established in Lemma \ref{lem:asymptfbfg}, we can obtain the asymptotic behavior of $z^+_\nu (g)$ and $z^\diamond_\nu (g)$ as $g$ approaches $1$:

\begin{lemm} \label{lem:asymzgtcrit}
For $\nu \in (0, \nu_c]$, we have as $g \to 1^-$:
\begin{align} 
z^\diamond_\nu(g) &= z^\diamond_\nu - \frac{1}{2} \left( {\sqrt{z^+_\nu}} \aleph^\diamond (\nu) + z^+_\nu \aleph^\bullet(\nu) \right)^{-\frac{1}{\alphaa(\nu)-1/2}} \, \left( 1-g \right)^{\frac{1}{\alphaa(\nu)-1/2}} + o \left(\left( 1-g \right)^{\frac{1}{\alphaa(\nu)-1/2}} \right), \label{eq:zdgexp} \\
z^+_\nu(g) &= z^+_\nu - \frac{\sqrt{z^+_\nu}}{2} \left( {\sqrt{z^+_\nu}} \aleph^\diamond (\nu) + z^+_\nu \aleph^\bullet(\nu) \right)^{-\frac{1}{\alphaa(\nu)-1/2}} \, \left( 1-g \right)^{\frac{1}{\alphaa(\nu)-1/2}} + o \left(\left( 1-g \right)^{\frac{1}{\alphaa(\nu)-1/2}} \right), \label{eq:zpgexp}
\end{align}
where we recall that
\begin{equation}
\alphaa(\nu) =
\begin{cases}
5/3 & \text{for $\nu < \nu_c$,}\\
7/3 & \text{for $\nu = \nu_c$.}
\end{cases}
\quad \text{ and hence }\quad  \frac{1}{\alphaa(\nu)-1/2}=
\begin{cases}
6/7 & \text{for $\nu < \nu_c$,}\\
6/11 & \text{for $\nu = \nu_c$.}
\end{cases}
\end{equation}
\end{lemm}

\begin{rema}
As for Lemma \ref{lem:asymptfbfg}, we could establish a similar development for $z^+_\nu(g)$ and $z^\diamond_\nu (g)$ for $\nu > \nu_c$, with a singular term of order $(1-g)^{1/2}$. Since it could also be obtained directly with the regular criticality of the weight sequence in this regime (and is not needed for the rest of the article), we do not include it here. 
\end{rema}

\begin{proof}[Proof of Lemma \ref{lem:asymzgtcrit}]
As explained at the end of Section \ref{sec:BDGgen}, the functions $g \mapsto z^+_\nu(g)$ and $g \mapsto z^\diamond_\nu(g)$ are positive and increasing in $g$, and we have $\Big(z^+_\nu(g),z^\diamond_\nu(g)\Big) \to (z^+_\nu,z^\diamond_\nu)$ as $g \to 1^-$.
\bigskip

Fix $\nu \in (0, \nu_c]$. We start from the expansion for $f^\diamond_\nu$ established in Lemma \ref{lem:asymptfbfg}, where we take $z_1=z_\nu^+(g)$ and $z_2=z^\diamond_\nu(g)$. Since, by definition $f^\diamond_\nu \Big(z^+_\nu(g),z^\diamond_\nu(g)\Big) = z^\diamond_\nu(g)$ (see~\eqref{eq:zgtcrit}), this gives:
	\begin{multline*}
z^\diamond_\nu - z^\diamond_\nu(g) = \partial_{z_1} f^\diamond_\nu (z^+_\nu,z^\diamond_\nu) \cdot \Big(z^+_\nu - z^+_\nu(g)\Big) + \left( 1 - \sqrt{z^+_\nu} \partial_{z_1} f^\diamond_\nu(z^+_\nu,z^\diamond_\nu) \right)  \cdot \Big(z^\diamond_\nu - z^\diamond_\nu(g)\Big) \\
 - \aleph^\diamond(\nu) \left( \Big(z^\diamond_\nu - z^\diamond_\nu(g)\Big) + \frac{1}{\sqrt{z^+_\nu}} \Big(z^+_\nu - z^+_\nu(g)\Big) \right)^{\alphaa(\nu)-1/2} \\
+ \mathcal O \left( \Big(z^\diamond_\nu - z^\diamond_\nu(g)\Big)^2 + \Big(z^+_\nu - z^+_\nu(g)\Big)^2 \right).
\end{multline*}
As a consequence, we have: 
\begin{equation} \label{eq:zpg1}
z^+_\nu - z^+_\nu(g) = {\sqrt{z^+_\nu}} \, \Big(z^\diamond_\nu - z^\diamond_\nu(g)\Big)
+ \frac{ \aleph^\diamond (\nu)}{\partial_{z_1} f^\diamond_\nu(z^+_\nu,z^\diamond_\nu)} \left( 2 \Big(z^\diamond_\nu - z^\diamond_\nu(g)\Big) \right)^{\alphaa(\nu)-1/2} + o \left(\Big(z^\diamond_\nu - z^\diamond_\nu(g)\Big)^{\alphaa(\nu)-1/2} \right).
\end{equation}

Now, we take $z_1=z_\nu^+(g)$ and $z_2=z^\diamond_\nu(g)$ in the development for $f^\bullet_\nu$ obtained in Lemma \ref{lem:asymptfbfg}. Combined with the fact that $f^\bullet_\nu \Big(z^+_\nu(g),z^\diamond_\nu(g)\Big) = 1 - \frac{g}{z^+_\nu(g)}$, we get:
\begin{align}
& \frac{1}{z^+_\nu} - \frac{g}{z^+_\nu(g)} \notag\\
&= - \frac{1}{z^+_\nu} \left( \frac{1}{z^+_\nu} -\sqrt{z^+_\nu} \partial_{z_1} f^\diamond_\nu (z^+_\nu,z^\diamond_\nu) \right) \left( z^+_\nu - z^+_\nu(g) \right)
- \partial_{z_1} f^\diamond_\nu (z^+_\nu,z^\diamond_\nu) \left( z^\diamond_\nu - z^\diamond_\nu (g) \right) \notag\\
& \qquad + \aleph^\bullet(\nu) \left( \Big(z^\diamond_\nu - z^\diamond_\nu(g)\Big) + \frac{1}{\sqrt{z^+_\nu}} \Big(z^+_\nu - z^+_\nu(g)\Big) \right)^{\alphaa(\nu) - 1/2} 
 + \mathcal O \left( \Big(z^\diamond_\nu - z^\diamond_\nu(g)\Big)^2 + \Big(z^+_\nu - z^+_\nu(g)\Big)^2 \right), \notag
\end{align}
Then, replacing $z^+_\nu - z^+_\nu(g)$ by its expansion given in~\eqref{eq:zpg1} yields:
\begin{multline}\label{eq:zpg2}
\frac{1}{z^+_\nu} - \frac{g}{z^+_\nu(g)} =  -\frac{1}{(z^+_\nu)^{3/2}} \, \Big(z^\diamond_\nu - z^\diamond_\nu(g)\Big)\\
- \frac{1}{z^+_\nu} \left( \frac{1}{z^+_\nu} -\sqrt{z^+_\nu} \partial_{z_1} f^\diamond_\nu (z^+_\nu,z^\diamond_\nu) \right)  \frac{ \aleph^\diamond (\nu)}{\partial_{z_1} f^\diamond_\nu(z^+_\nu,z^\diamond_\nu)} \left( 2 \Big(z^\diamond_\nu - z^\diamond_\nu(g)\Big) \right)^{\alphaa(\nu) - 1/2}\\
+\aleph^\bullet(\nu) \left( 2 \Big(z^\diamond_\nu - z^\diamond_\nu(g)\Big) \right)^{\alphaa(\nu) -1/2}
 + o \left( \Big(z^\diamond_\nu - z^\diamond_\nu(g)\Big)^{\alphaa(\nu) -1/2} \right).
\end{multline}
On the other hand, we may also write: 
\begin{align*}
\frac{1}{z^+_\nu} - \frac{g}{z^+_\nu(g)}&=\frac{1-g}{z^+_\nu} + \frac{g}{z^+_\nu}\left(1-\frac{z^+\nu}{z^+_\nu(g)}\right)\\&= \frac{1-g}{z^+_\nu} - \frac{g}{(z^+_\nu)^2} \left(z^+_\nu - z^+_\nu(g) \right) +  \mathcal O \left( \Big(z^+_\nu - z^+_\nu(g)\Big)^2 \right) \notag \\
\end{align*}
Using again the expansion of $z^+_\nu - z^+_\nu(g)$ obtained in~\eqref{eq:zpg1}, we get:
\begin{multline}\label{eq:zpg3}
\frac{1}{z^+_\nu} - \frac{g}{z^+_\nu(g)} = \frac{1-g}{z^+_\nu}
 - \frac{g}{(z^+_\nu)^{3/2}}\Big(z^\diamond_\nu - z^\diamond_\nu(g)\Big)\\
- \frac{g}{(z^+_\nu)^{2}} \frac{ \aleph^\diamond(\nu)}{\partial_{z_1} f^\diamond_\nu(z^+_\nu,z^\diamond_\nu)} \left( 2 \Big(z^\diamond_\nu - z^\diamond_\nu(g)\Big) \right)^{\alphaa(\nu) - 1/2}  +o \left(\Big(z^\diamond_\nu - z^\diamond_\nu(g)\Big)^{\alphaa(\nu) - 1/2} \right).
\end{multline}
Identifying the right-hand sides of both~\eqref{eq:zpg2} and~\eqref{eq:zpg3} gives the following equality:
\begin{multline*}
\frac{1-g}{z^+_\nu} + \frac{1-g}{(z^+_\nu)^{3/2}}
 \Big(z^\diamond_\nu - z^\diamond_\nu(g)\Big)
+ \frac{1-g}{(z^+_\nu)^{2}} \frac{ \aleph^\diamond (\nu)}{\partial_{z_1} f^\diamond_\nu(z^+_\nu,z^\diamond_\nu)} \left( 2 \Big(z^\diamond_\nu - z^\diamond_\nu(g)\Big) \right)^{\alphaa(\nu) - 1/2}\\
=\left( \frac{1}{\sqrt{z^+_\nu}} \aleph^\diamond (\nu) + \aleph^\bullet(\nu) \right) \left( 2 \Big(z^\diamond_\nu - z^\diamond_\nu(g)\Big) \right)^{\alphaa(\nu) - 1/2} 
 + o \left( \Big(z^\diamond_\nu - z^\diamond_\nu(g)\Big)^{\alphaa(\nu) - 1/2} \right),
\end{multline*}
from which the singular development for $z^\diamond_\nu(g)$ follows.

The expansion for $z^+_\nu (g)$ can then be obtained by plugging the expansion for $z^\diamond_\nu(g)$ in \eqref{eq:zpg1}.
\end{proof}

\section{root spin cluster in finite random triangulations}

\subsection{root spin cluster in Ising Boltzmann triangulations (proof of Theorem \ref{th:mainBoltzmann})}

\subsubsection{Perimeter distribution and regular criticality of \texorpdfstring{$\mathbf q (\nu,t_\nu)$}{q(nu,tnu)}}

We start by proving the statements of Theorem~\ref{th:mainBoltzmann} about the perimeter of the root spin cluster. 
The proof relies on the fact that under $\mathbb P^\nu$, the root spin cluster is a Boltzmann random map with weight sequence $\mathbf q(\nu,t_\nu)$, as stated in Proposition \ref{prop:bolt}. We have:
\[
\mathbb P^\nu \left( | \partial \mathfrak C(\mathbf{T}^\nu) | = l \right)
= \frac{w_{\mathbf q(\nu,t_\nu)} (\mathcal M_{l})}{\mathcal Z^\ps (\nu,t_\nu)},
\]
where we recall that $\mathcal{M}_l$ is the set of planar maps with a root face of degree $l$. 
By definition of the disk generating series $W^{(l)}_{\mathbf q(\nu,t_\nu)}$, given in \eqref{eq:defDiskL}, the Boltzmann weight of clusters with perimeter $l \geq 1$ is given by
\[
w_{\mathbf q(\nu,t_\nu)} (\mathcal M_{l}) = W^{(l)}_{\mathbf q(\nu,t_\nu)} \cdot q_l(\nu,t_\nu).
\]
Therefore, we have 
\[
\mathbb P^\nu \left( | \partial \mathfrak C(\mathbf{T}^\nu) | = l \right)
=
\frac{W^{(l)}_{\mathbf q(\nu,t_\nu)} \cdot q_l(\nu,t_\nu)}{\mathcal Z^\ps (\nu,t_\nu)} = \frac{(t_\nu \nu)^{-l /2}
Q_l^+ (\nu, t_\nu) \cdot q_l(\nu,t_\nu)}{\mathcal Z^\ps (\nu,t_\nu)},
\]
where the last equality comes from \eqref{eq:Wl}. Define $\mathrm{C}^{\mathrm p}_{\mathrm{bol}}(\nu)$ by: 
 \begin{equation}\label{eq:CperBol}
 \mathrm{C}^{\mathrm p}_{\mathrm{bol}}(\nu):=\frac{(y_\nu-1)^{\alphaa(\nu)}\Big(\aleph^{Q^+}(\nu)\Big)^2}{\mathcal Z^\ps(\nu,t_\nu)\left(\Gamma\Big(1-\alphaa(\nu)\Big)\right)^2}.
 \end{equation}
Then, the asymptotics given in Corollary~\ref{cor:Q} and Proposition \ref{prop:asymptoq} yield: 
\begin{align}\label{eq:perimeterTail}
\mathbb P^\nu \left( \partial \mathfrak C(\mathbf{T}^\nu) = l \right) 
&\underset{l \to \infty}{\sim} \mathrm{C}^{\mathrm p}_{\mathrm{bol}}(\nu) (t_\nu \nu)^{-l /2} 
\times
\left({t_\nu} \right)^{-l} \left({y(\nu,t_\nu)} \right)^{-l} l^{-\alphaa(\nu)}
\times
\left(\frac{y(\nu,t_\nu)}{y(\nu,t_\nu)-1} \, t_\nu^{3/2} \nu^{1/2} \right)^l l^{-\alphaa(\nu)} \notag \\
&\underset{l \to \infty}{\sim} \mathrm{C}^{\mathrm p}_{\mathrm{bol}}(\nu)
\,
\left(\frac{1}{y(\nu,t_\nu)-1} \, \right)^l l^{-2\alphaa(\nu)},
\end{align}
The statements about the perimeter tail distribution from Theorem~\ref{th:mainBoltzmann} follow since $y(\nu,t_\nu) = 2$ for $\nu \leq \nu_c$ and $y(\nu,t_\nu) > 2$ for $\nu > \nu_c$.
\bigskip

A direct consequence of this result is the following classification of $\mathbf{q} (\nu,t_\nu)$ according to Definition~\ref{def:regcrit}:
\begin{prop}\label{prop:regNonReg}
The weight sequence $\mathbf{q}(\nu,t_\nu)$ is:
\begin{itemize}
\item non-regular critical for $\nu\leq\nu_c$. 
\item regular critical for $\nu>\nu_c$.
\end{itemize}
\end{prop}

\subsubsection{Volume distribution}\label{sub:volumeBoltzmann}

We now move to the volume tail distribution and start with the (much easier) case $\nu > \nu_c$ where the cluster is a regular critical Boltzmann map, and finally we establish the volume tail behavior in the case $\nu \leq \nu_c$ where the root spin cluster is a non generic critical Boltzmann map.

\paragraph{Volume tail when $\nu > \nu_c$.}
In this range of values for $\nu$, Proposition~\ref{prop:regNonReg} states that $\mathfrak C(\mathbf T ^\nu)$ is a regular critical Boltzmann map and the statement follows by Proposition~\ref{prop:regcritvol}.

\paragraph{Volume tail when $\nu \leq \nu_c$.} In this regime, the map $\mathfrak C(\mathbf T^\nu)$ is non-regular critical and we have to compute the generating function of its volume.

For $g \leq 1$, recall the definition of $z^+(g)$ and $z^\diamond(g)$ given in \eqref{eq:zgtcrit}.
Equation \eqref{eq:Zg} states that:
\begin{equation}\label{eq:Ebolg}
\mathbb E^\nu \left[|V(\mathfrak C (\mathbf T^\nu))| \,  g^{|V(\mathfrak C (\mathbf T^\nu))|}\right]
=
g \, \frac{ 2 (z^+(g) - g) + (z^\diamond(g))^2}{\mathcal Z^\ps(\nu,t_\nu)}.
\end{equation}
The expansions of $z^+(g)$ and $z^\diamond(g)$ obtained in~Lemma \ref{lem:asymzgtcrit} imply that $\mathbb E^\nu \left[|V(\mathfrak C (\mathbf T^\nu))| \,  g^{|V(\mathfrak C (\mathbf T^\nu))|}\right]$ has a singular expansion as $g \to 1^-$ of the form:
\begin{equation}\label{eq:expE}
\mathbb E^\nu \left[|V(\mathfrak C (\mathbf T^\nu))| \,  g^{|V(\mathfrak C (\mathbf T^\nu))|}\right] = \mathbb E^\nu \left[|V(\mathfrak C (\mathbf T^\nu))| \right] - \mathrm{C}(\nu) (1-g)^{\frac{1}{\alphaa(\nu) - 1/2}} + o \left( (1-g)^{\frac{1}{\alphaa(\nu) - 1/2}} \right),
\end{equation}
with $\mathrm{C}(\nu) >0$.

Unfortunately, this expansion is not enough to extract the asymptotic behavior of the coefficients in $g$ of $\mathbb E^\nu \left[ |V(\mathfrak C (\mathbf T^\nu)| \, g^{|V(\mathfrak C (\mathbf T^\nu)|}\right]$ since it could have complex singularities of modulus $1$ other than $1$. To work around this problem, we rely on a related series in $g$ with decreasing coefficients and Tauberian theorems, as we now explain.
\bigskip

It follows directly from~\eqref{eq:expE} that: 
\[
\frac{\mathbb E^\nu \left[|V(\mathfrak C (\mathbf T^\nu))| \right] - \mathbb E^\nu \left[|V(\mathfrak C (\mathbf T^\nu))| \,  g^{|V(\mathfrak C (\mathbf T^\nu))|}\right]}{1-g}
\underset{g \to 1^-}{\sim}
\mathrm{C}(\nu) (1-g)^{\frac{1}{\alphaa(\nu) - 1/2}-1},
\]
where the exponent $\frac{1}{\alphaa(\nu) - 1/2}-1$ is negative for every $\nu \leq \nu_c$, since $\alphaa(\nu) = 5/3$ for $\nu <\nu_c$ and $\alphaa(\nu_c) = 7/3$. 

On the other hand, denoting $p_n = \mathbb P^\nu \left(|V(\mathfrak C (\mathbf T^\nu))| = n \right)$, we also have
\begin{align*}
\frac{\mathbb E^\nu \left[|V(\mathfrak C (\mathbf T^\nu))| \right] - \mathbb E^\nu \left[|V(\mathfrak C (\mathbf T^\nu))| \,  g^{|V(\mathfrak C (\mathbf T^\nu))|}\right]}{1-g}
&=
\sum_{n \geq 0} \left( \mathbb E^\nu \left[|V(\mathfrak C (\mathbf T^\nu))| \right] - \sum_{k=0}^n k p_k\right) \, g^n, \\
& = \sum_{n \geq 0} \sigma_n \, g^n,
\end{align*}
with $\sigma_n := \sum_{k >n} k p_k$ positive and decreasing with $n$. From there, a classical Tauberian theorem (see e.g. Theorem VI.13 of \cite{FS} and the following discussion) implies that:
\[
\sigma_n \underset{n\to \infty}{\sim} \mathrm{C}(\nu) \frac{n^{-\frac{1}{\alphaa(\nu) - 1/2}}}{\Gamma\left(1-\frac{1}{\alphaa(\nu) - 1/2} \right)}.
\]

To complete the proof of Theorem~\ref{th:mainBoltzmann}, we want to relate this last asymptotic behavior to the behavior of $s_n = \sum_{k>n} p_k$. Since
\[
\sum_{k \geq n} s_k = \sum_{l>n} p_l (l-n) = \sigma_n - n s_n,
\]
we can again use Tauberian methods to establish the asymptotic behavior of $s_n$. Indeed, for $\varepsilon >0$, we have: 
\[
	ns_n - n(1+\varepsilon ) s_{n(1+\varepsilon )} + \sum_{k = n}^{n (1+ \varepsilon)-1} s_n = \sigma_n-\sigma_{n(1+\varepsilon)},
\]
which implies
\begin{multline} \label{eq:snintermediaire}
ns_n - n(1+\varepsilon ) s_{n(1+\varepsilon )} + \sum_{k = n}^{n (1+ \varepsilon)-1} s_n \\
 = 
\frac{\mathrm{C}(\nu)}{\Gamma\left(1-\frac{1}{\alphaa(\nu) - 1/2} \right)} \, n^{-\frac{1}{\alphaa(\nu) - 1/2}} \, \left(1 - (1+\varepsilon)^{-\frac{1}{\alphaa(\nu) - 1/2}} \right) + o \left( n^{-\frac{1}{\alphaa(\nu) - 1/2}} \right).
\end{multline}
Using the fact that $s_n$ is decreasing, we get
\[
n (1+ \varepsilon) s_n \geq n(1+\varepsilon ) s_{n(1+\varepsilon )} +
\frac{\mathrm{C}(\nu)}{\Gamma\left(1-\frac{1}{\alphaa(\nu) - 1/2} \right)} \, n^{-\frac{1}{\alphaa(\nu) - 1/2}} \, \left(1 - (1+\varepsilon)^{-\frac{1}{\alphaa(\nu) - 1/2}} \right) + o \left( n^{-\frac{1}{\alphaa(\nu) - 1/2}} \right).
\]
Multiplying by $n^{ \frac{1}{\alphaa(\nu) - 1/2}}$ and taking the $\varliminf$ as $n \to \infty$ on both sides we get:
\begin{align*}
&(1+ \varepsilon) \varliminf \left(n^{ 1+  \frac{1}{\alphaa(\nu) - 1/2}}  s_n \right)
\\ \geq\quad &
(1+\varepsilon )^{- \frac{1}{\alphaa(\nu) - 1/2}}
\varliminf \left(
\Big(n(1+\varepsilon )\Big)^{ 1+  \frac{1}{\alphaa(\nu) - 1/2}} s_{n(1+\varepsilon )} \right) \ + \frac{\mathrm{C}(\nu)}{\Gamma\left(1-\frac{1}{\alphaa(\nu) - 1/2} \right)} \left(1 - (1+\varepsilon)^{-\frac{1}{\alphaa(\nu) - 1/2}} \right),\\
 \geq\quad&
(1+\varepsilon )^{- \frac{1}{\alphaa(\nu) - 1/2}}
\varliminf \left(n^{ 1+  \frac{1}{\alphaa(\nu) - 1/2}}  s_n \right)
 + \frac{\mathrm{C}(\nu)}{\Gamma\left(1-\frac{1}{\alphaa(\nu) - 1/2} \right)} \left(1 - (1+\varepsilon)^{-\frac{1}{\alphaa(\nu) - 1/2}} \right).
\end{align*}
Letting $\varepsilon \to 0$ in the last display yields
\begin{equation} \label{eq:snlliminf}
\varliminf \left(n^{ 1+  \frac{1}{\alphaa(\nu) - 1/2}}  s_n \right)
\geq
\frac{\frac{1}{\alphaa(\nu) - 1/2}}{1+\frac{1}{\alphaa(\nu) - 1/2}}
\frac{\mathrm{C}(\nu)}{\Gamma\left(1-\frac{1}{\alphaa(\nu) - 1/2} \right)}
=\frac{1}{\alphaa(\nu) + 1/2}
\frac{\mathrm{C}(\nu)}{\Gamma\left(1-\frac{1}{\alphaa(\nu) - 1/2} \right)}.
\end{equation}

We can establish an inequality for the $\varlimsup$ in a similar fashion. Indeed, starting from~\eqref{eq:snintermediaire} and using the fact that $s_n$ decreases with $n$, we get
\[
ns_n - n s_{n(1+\varepsilon)}
\leq
\frac{\mathrm{C}(\nu)}{\Gamma\left(1-\frac{1}{\alphaa(\nu) - 1/2} \right)} \, n^{-\frac{1}{\alphaa(\nu) - 1/2}} \, \left(1 - (1+\varepsilon)^{-\frac{1}{\alphaa(\nu) - 1/2}} \right)
+ o \left( n^{-\frac{1}{\alphaa(\nu) - 1/2}} \right).
\]
Multiplying by $n^{ \frac{1}{\alphaa(\nu) - 1/2}}$ and taking the $\varlimsup$ as $n \to \infty$ on both sides we get
\begin{align*}
&\varlimsup \left( n^{1 + \frac{1}{\alphaa(\nu) - 1/2}} s_n \right)\\
 \leq\quad &
\varlimsup \left( n^{1 + \frac{1}{\alphaa(\nu) - 1/2}} s_{n(1+\varepsilon)} \right) 
 + \frac{\mathrm{C}(\nu)}{\Gamma\left(1-\frac{1}{\alphaa(\nu) - 1/2} \right)} \, n^{-\frac{1}{\alphaa(\nu) - 1/2}} \, \left(1 - (1+\varepsilon)^{-\frac{1}{\alphaa(\nu) - 1/2}} \right),\\
\leq\quad &
(1+\varepsilon)^{-1 - \frac{1}{\alphaa(\nu) - 1/2}}
\varlimsup \left( n^{1 + \frac{1}{\alphaa(\nu) - 1/2}} s_n \right)  + \frac{\mathrm{C}(\nu)}{\Gamma\left(1-\frac{1}{\alphaa(\nu) - 1/2} \right)} \, n^{-\frac{1}{\alphaa(\nu) - 1/2}} \, \left(1 - (1+\varepsilon)^{-\frac{1}{\alphaa(\nu) - 1/2}} \right).
\end{align*}
Letting $\varepsilon \to 0$ in the last display finally gives
\begin{equation} \label{eq:snlimsup}
\varlimsup \left( n^{1 + \frac{1}{\alphaa(\nu) - 1/2}} s_n \right)
\leq
\frac{1}{\alphaa(\nu) + 1/2}
\frac{\mathrm{C}(\nu)}{\Gamma\left(1-\frac{1}{\alphaa(\nu) - 1/2} \right)}.
\end{equation}
By combining inequalities \eqref{eq:snlliminf} and \eqref{eq:snlimsup}, we get:
\begin{equation} \label{eq:snequiv}
s_n = \mathbb P^\nu \left(|V(\mathfrak C (\mathbf T^\nu))| > n \right)
\underset{n \to \infty}{\sim}
\frac{1}{\alphaa(\nu) + 1/2}
\frac{\mathrm{C}(\nu)}{\Gamma\left(1-\frac{1}{\alphaa(\nu) - 1/2} \right)} \,
n^{- \frac{\alphaa(\nu) + 1/2}{\alphaa(\nu) - 1/2}}.
\end{equation}
This finishes the proof of Theorem~\ref{th:mainBoltzmann} since the exponent in the previous expansion is $-13/7$ in the high temperature case when $\alphaa(\nu) = 5/3$, and $-17/11$ in the critical temperature case when $\alphaa(\nu) = 7/3$.

\subsection{root spin cluster in large random Ising triangulations (proof of Theorem \ref{th:mainsizen})}

Recall from~\eqref{eq:Pnundef} that $\mathbf T_{n}^\nu$ denotes a random triangulation with law $\mathbb P^\nu$ conditioned to have $3n$ edges, meaning:
\[
\forall (\mathfrak t,\sigma) \in \mathcal T^\ps \, , \quad \mathbb P_{n}^\nu \left( \mathbf T^\nu_n = (\mathfrak t,\sigma)\right) = \frac{\nu^{m(\mathfrak t,\sigma)} }{ [t^{3n}] \, \mathcal Z^\ps (\nu,t)} \mathbf{1}_{\{ | \mathfrak t | = 3n\}}.
\]
It follows from Proposition~\ref{prop:q_k}, that the law of $\mathfrak{C}\left( \mathbf T^\nu_n\right)$ is given by: 
\begin{equation*}
\forall \mathfrak m \in \mathcal M   \, , \quad
 \mathbb P_n^\nu \left( \mathfrak C(\mathbf T_n^\nu) = \mathfrak m \right)
=
\frac{[t^{3n}] \prod_{f \in F(\mathfrak m)} q_{\mathrm{deg}(f)}(\nu,t)}{[t^{3n}] \mathcal Z^\ps (\nu,t)},
\end{equation*}
and therefore
\begin{align*}
\mathbb E_n^\nu \left[ |V(\mathfrak C (\mathbf T_n^\nu))| \right]
&=
\frac{[t^{3n}] \sum_{\mathfrak m \in \mathcal M} |V( \mathfrak m )| \, \prod_{f \in F(\mathfrak m)} q_{\mathrm{deg}(f)}(\nu,t)}{[t^{3n}] \mathcal Z^\ps (\nu,t)},\\
& = \frac{[t^{3n}] w_{\mathbf q(\nu,t)} \left( \mathcal M ^\bullet \right) }{[t^{3n}] \mathcal Z^\ps (\nu,t)}.
\end{align*}
Equations~\eqref{eq:cyltobolt} and \eqref{eq:pointedDisk} then give
\begin{align*}
\mathbb E_n^\nu \left[ |V(\mathfrak C (\mathbf T_n^\nu))| \right]
&=
\frac{[t^{3n}] \left( W_{\mathbf q(\nu,t),\bullet}^{(2)} - 1 \right)}{[t^{3n}] \mathcal Z^\ps (\nu,t)}\\
& = \frac{[t^{3n}] \left( \frac{3}{8} c_+^2 (\nu,t) + \frac{3}{8} c_-^2(\nu,t) + \frac{1}{4} c_+(\nu,t)c_-(\nu,t) \right) }{[t^{3n}] \mathcal Z^\ps (\nu,t)}.
\end{align*}
We can then use the link between the quantities $c_\pm(\nu,t)$ and the singularities $y_\pm (\nu,t)$ of $Q^+(\nu,t,ty)$ given in \eqref{eq:cy+-} to get
\begin{equation} \label{eq:expectCy+-}
\mathbb E_n^\nu \left[ |V(\mathfrak C (\mathbf T_n^\nu))| \right] = \frac{[t^{3n}] \frac{1}{\nu t^3} \left(\frac{3}{8} \frac{1}{y_+(\nu,t)^2} + \frac{3}{8} \frac{1}{y_-(\nu,t)^2} + \frac{1}{4} \frac{1}{y_+(\nu,t)y_-(\nu,t)} \right)}{[t^{3n}] \mathcal Z^\ps (\nu,t)}.
\end{equation}

Proposition \ref{prop:asym_t_sign_y} ensures that the numerator of \eqref{eq:expectCy+-} has asymptotics of the form $\mathrm{Cst} \, t_\nu^{-3n} \, n^{-\gamma_+(\nu)}$, with $\gamma_+ (\nu) = 7/4$ for $\nu < \nu_c$ and  $\gamma_+ (\nu) = 3/2$ for $\nu > \nu_c$. Moreover, the asymptotics for $[t^{3n}] \mathcal Z^\ps (\nu,t)$ are of the form $\mathrm{Cst} \, (t_\nu)^{-3n} \, n^{-\gamma(\nu)}$, with $\gamma (\nu) = 5/2$ for $\nu \neq \nu_c$ and  $\gamma (\nu) = 7/3$ for $\nu = \nu_c$. Therefore we have
\[
\mathbb E_n^\nu \left[ |V(\mathfrak C(\mathbf T_n^\nu))| \right] \underset{n \to \infty}{\sim} \mathrm{Cst} \, n^{\gamma(\nu) - \gamma_+(\nu)},
\]
finishing the proof of Theorem~\ref{th:mainsizen}.

\section{root spin cluster in the IIPT}

\subsection{Cluster volume generating function and percolation probability}

We start by identifying the law of the root spin cluster in the IIPT on the event where it is finite.

\begin{prop} \label{prop:boltIIPT}
For any non atomic finite rooted map $\mathfrak m$ and every $\nu >0$ one has
\begin{equation} \label{eq:boltIIPT}
\mathbb P_\infty^\nu \left( \mathfrak C(\mathbf T_\infty^\nu) = \mathfrak m \right) =  \left( \prod_{f \in F(\mathfrak m)} q_{\mathrm{deg}(f)} (\nu ,t_\nu) \right)
\cdot \sum_{f \in F(\mathfrak m)} \frac{(\nu t_\nu^3)^{\mathrm{deg}(f)/2} \, \delta_{\mathrm{deg}(f)} (\nu)}{q_{\mathrm{deg}(f)} (\nu ,t_\nu)},
\end{equation}
where the numbers $\left( \delta_k(\nu) \right)_{k \geq 1}$ are defined in Proposition \ref{prop:asymptQk}.
\end{prop}
\begin{proof}
We start from~\eqref{eq:boltn}, which can be written as:
\[
\mathbb P_n^\nu \left( \mathfrak C(\mathbf T_n^\nu) = \mathfrak m \right) = \frac{[t^{3n}]\left( (\nu t^3)^{|\mathfrak m|} \cdot \prod_{f \in F(\mathfrak m)} (\nu t^3)^{-\mathrm{deg}(f)/2} q_{\mathrm{deg}(f)} (\nu ,t) \right)}{[t^{3n}] \mathcal Z ^\ps ( \nu ,t)}.
\]
Proposition \ref{prop:asymptQk} then states that, for every $k$, the asymptotic behavior of $[t^{3n}] \left( (\nu t^3)^{-k/2} q_{k} (\nu ,t) \right)$ are of the form $\mathrm{Cst}_k\cdot t_\nu^{-3n} n ^{-\gamma(\nu)}$ with $\gamma(\nu) >2$, which is also the asymptotic behavior of $[t^{3n}] \mathcal Z^\ps(\nu,t)$. We also proved in the same lemma that the ratio of these asymptotics converge towards $\delta_k(\nu)$ for every $k$. From there, classical calculations give the announced formula.
\end{proof}

\bigskip

Our starting point to prove Theorem~\ref{th:expo} and Theorem~\ref{th:mainIIPT} will be the next statement, which establishes integral formulae for the probability that the root spin cluster is finite and for the generating function of the volume of the root spin cluster.

\begin{prop} \label{prop:formulaperco}
Recall the definition of the function $\Delta$ given in Proposition~\ref{prop:asymptQk} and set $c_\pm^\nu := c_\pm(\nu,t_\nu$). For every $\nu >0$ one has
\begin{align} \label{eq:probaCfinite}
\mathbb P_\infty^{\nu} & (|\mathfrak C(\mathbf T_\infty^\nu))| < \infty)= -\frac{1}{2 \pi} \int_{c_-^\nu}^{c_+^\nu} \frac{dz}{z} \Delta  \left(\nu,\sqrt{\nu t_\nu^3} \,  z \right) 
\left(z + \frac{c_+^\nu + c_-^\nu}{2} \right) \sqrt{(c_+^\nu - z)(z-c_-^\nu)}.
\end{align}
In addition, for $g \in (0,1)$ one has
\begin{align}
\mathbb E_\infty^\nu & \left[ g^{|V(\mathfrak C (\mathbf T_\infty^\nu))|}\right] \notag \\
& =
-\frac{g}{2 \pi} \int_{c_-^\nu(g)}^{c_+^\nu(g)} \frac{dz}{z} \Delta \left(\nu,\sqrt{g \, \nu t_\nu ^3} \, z \right)
\left(z + \frac{c_+^\nu(g) + c_-^\nu(g)}{2} \right) \sqrt{(c_+^\nu(g)-z)(z-c_-^\nu(g))},
\end{align}
where
\begin{equation}
c_\pm^\nu(g) := c_\pm(\mathbf q_g(\nu,t_\nu)).
\end{equation}
\end{prop}

\begin{proof}
By Proposition \ref{prop:boltIIPT}, for any $\nu>0$ and any non atomic map $\mathfrak m$ one has
\begin{equation*}
\mathbb P_\infty^\nu  \left( \mathfrak C(\mathbf T^\nu_\infty) = \mathfrak m \right) =
\sum_{f \in F(\mathfrak m)} (\nu t_\nu^3)^{\mathrm{deg}(f)/2} \,  \delta_{\mathrm{deg}(f)}
\prod_{f' \in F(\mathfrak m), f' \neq f} q_{\mathrm{deg}(f')} (\nu ,t_\nu).
\end{equation*}
We want to sum this probability over finite rooted non atomic maps $\mathfrak m \in \mathcal M$. This will be easier to manage if we sum maps with a marked face (the face denoted $f'$ in the above expression).

Let $\mathfrak m_\square:=(\mathfrak m, f_\square)$ be a non-atomic rooted map $\mathfrak m$, with an additional marked face $f_\square$, such that $f_\square$ is different from the root face $f_r$ of $\mathfrak m$. The Boltzmann weight $w_\square(\mathfrak m_\square)$ of $\mathfrak{m}_\square$ is defined as:
\[
w_{\square}(\mathfrak m_\square) = \prod_{f \in F(\mathfrak m)\setminus\{f_r,f_\square\}} q_{\mathrm{deg}(f)} (\nu ,t_\nu).
\]
For $l_1,l_2 \geq 1$, let us denote by $\mathcal M_\square^{(l_1,l_2)}$ the set all maps with root face degree $l_1$ and a marked face of degree $l_2$, and set
\[
W_\square^{(l_1,l_2)} = \sum_{\mathfrak m_\square \in \mathcal M_\square^{(l_1,l_2)}} w_{\square}(\mathfrak m_\square).
\]
By rooting the marked face so that it lies on the left side of this new root, we see that
\[
W_\square^{(l_1,l_2)} = \frac{1}{l_2} W^{(l_1,l_2)}_{2},
\]
where we recall that $W^{(l_1,l_2)}_{2}$ is the generating function of $\mathbf q(\nu,t_\nu)$-Boltzmann maps of the cylinder with perimeters $l_1$ and $l_2$ defined in~\eqref{eq:cylinder}.
\bigskip

We then have:
\begin{align*}
\mathbb P_\infty^{\nu} (|\mathfrak C(\mathbf T_\infty^\nu)| < \infty) &=
\sum_{\mathfrak m \in \mathcal M} \mathbb P_\infty^{\nu}(\mathfrak C = \mathfrak m),\\
&= \sum_{\mathfrak m \in \mathcal M_f^{(2)}} \, 
\sum_{f \in F(\mathfrak m)\setminus\{f_r\}} \,  (\nu t_\nu^3)^{\mathrm{deg}(f)/2} \, \delta_{\mathrm{deg}(f)}
\prod_{f' \in F(\mathfrak m), f' \neq f,f_r} q_{\mathrm{deg}(f')} (\nu ,t_\nu),\\
&=\sum_{\mathfrak m \in \mathcal M_{f,\square}^{(2)}} \, 
 (\nu t_\nu^3)^{\mathrm{deg}(f_\square)/2} \, 
\delta_{\mathrm{deg}(f_\square)} \,
w_\square(\mathfrak m_\square),\\
&=
\sum_{k \geq 1}  (\nu t_\nu^3)^{k/2} \,  \, \delta_k \cdot W_\square^{(2,k)}, \\
&=
\sum_{k \geq 1}   (\nu t_\nu^3)^{k/2} \,  \, \delta_k \cdot \frac{1}{k} W^{(2,k)}_2.
\end{align*}
The last display is the Hadamard product of two generating series evaluated at $z=1$. Similarly to what was done in the proof of Proposition \ref{prop:fbulletdiamond} (see also Appendix~\ref{sec:Hadamard}), we can write:
\begin{align} \label{eq:probafinitecontour}
\mathbb P_\infty^{\nu} (|\mathfrak C(\mathbf T_\infty^\nu)| < \infty)
&= \frac{1}{2 i \pi} \oint_\gamma \frac{dz}{z} \Delta \left( \nu,\sqrt{\nu t_\nu^3} \, z \right) \cdot \left(\sum_{k \geq 1}  \frac{1}{k} W_2^{(2,k)} z^{-k}\right),
\end{align}
where the integration contour $\gamma$ is inside a domain where both functions $\Delta (\nu,\sqrt{\nu t_\nu^3} \, z)$ and $\sum_{k \geq 1}  \frac{1}{k} W_2^{(2,k)} z^{-k}$ are analytic. 

To check that such a contour exists and to be able to evaluate the integral,
we will compute $\sum_{k \geq 1}  \frac{1}{k} W_2^{(2,k)} z^{-k}$ below in~\eqref{eq:1/kW2z}. The expression we obtain shows that this function is analytic on $\mathbb C \setminus [c_-^\nu,c_+^\nu]$. On the other hand, from Proposition~\ref{prop:asymptQk}, we have that:
\[
\Delta \left( \nu,\sqrt{\nu t_\nu^3} \,z \right) = \frac{1}{ \beth^{\mathcal Z^\ps}(\nu) } \, \frac{ \sqrt{\nu t_\nu^3} \,z}{1 -  \sqrt{\nu t_\nu^3} \,z} \beth^{Q^+} \left( \nu, \frac{1}{1-  \sqrt{\nu t_\nu^3} \,z}\right).
\]
The parametric expression of $\beth^{Q^+} \left( \nu, y\right)$ in terms of $V(\nu,U_\nu,y)$ ensures that it is analytic on $\mathbb C \setminus \left( (- \infty , y_-(\nu,t_\nu] \cup [y_\nu , +\infty) \right)$. Therefore, the function $\Delta \left( \nu,\sqrt{\nu t_\nu^3} \,z \right)$ is analytic on $\mathbb C \setminus [z_\nu, \infty)$, where $z_\nu$ is computed in~\eqref{eq:znu} and satisfies
\[
z_\nu = c_+^\nu \cdot (y_\nu -1) \geq c_+^\nu
\]
since $y_\nu \geq 2$. 

As in the proof of Proposition~\ref{prop:fbulletdiamond}, when $z_\nu = c_+^\nu$ (which corresponds to $\nu\leq \nu_c$), we cannot directly find a suitable contour $\gamma$. However, we can fortunately first compute
\begin{align*}
\frac{1}{2 i \pi} \oint_\gamma \frac{dz}{z} \Delta \left( \nu,\sqrt{\nu t_\nu^3} \, w \, z \right) \cdot \left(\sum_{k \geq 1}  \frac{1}{k} W_2^{(2,k)} z^{-k}\right),
\end{align*}
for $w < 1$ where a contour exists and take the limit as $w \to 1^-$.

\bigskip

To evaluate $\sum_{k \geq 1}  \frac{1}{k} W_2^{(2,k)} z^{-k}$, we first apply Proposition \ref{prop:gencylinders} (with weight sequence $\mathbf q(\nu,t_\nu)$) to get:
\begin{align*}
\sum_{k \geq 1} W_2^{(2,k)} z^{-k-1}&=
[z_1^{-3}] W_2(z_1,z) \\
&= - z + W_\bullet(z) \, \left( z^2 - \frac{c_+^\nu + c_-^\nu}{2} z + \frac{1}{4} c_+^\nu c_-^\nu - \frac{1}{8} \left((c_+^\nu)^2 + (c_-^\nu)^2\right)\right).
\end{align*}
This can be rewritten as:
\[
\frac{d}{dz} \left(\sum_{k \geq 1} \frac{1}{k} W_2^{(2,k)} z^{-k}\right)
= z - W_\bullet(z) \, \left( z^2 - \frac{c_+^\nu + c_-^\nu}{2} z + \frac{1}{4} c_+^\nu c_-^\nu - \frac{1}{8} \left((c_+^\nu)^2 + (c_-^\nu)^2 \right)\right).
\]
The antiderivative has a simple expression:
\begin{equation} \label{eq:1/kW2z}
\sum_{k \geq 1} \frac{1}{k} W_2^{(2,k)} z^{-k}
=
\frac{z^2}{2} - \frac{z + \frac{c_+^\nu + c_-^\nu}{2}}{2 W_\bullet(z)}
=
\frac{z^2}{2} - \frac{z + \frac{c_+^\nu + c_-^\nu}{2}}{2} \sqrt{(z-c_+^\nu)(z-c_-^\nu)},
\end{equation}
which gives: 
\begin{align*}
\mathbb P_\infty^{\nu} (|\mathfrak C| < \infty) &=
\frac{1}{2 i \pi} \oint_\gamma \frac{dz}{z} \Delta \left( \nu,\sqrt{\nu t_\nu^3} \,  z \right) \cdot \left(\frac{z^2}{2} - \frac{z + \frac{c_+^\nu + c_-^\nu}{2}}{2} \sqrt{(z-c_+^\nu)(z-c_-^\nu)}\right), \notag\\
&= -\frac{1}{4 i \pi} \oint_\gamma \frac{dz}{z} \Delta  \left(\nu, \sqrt{\nu t_\nu^3} \,  z \right) \cdot  \left(z + \frac{c_+^\nu + c_-^\nu}{2} \right) \sqrt{(z-c_+^\nu)(z-c_-^\nu)}. 
\end{align*}
By deforming the contour $\gamma$ so that it surrounds the cut $[c_-^\nu,c_+^\nu]$ where the root is not analytic, we finally obtain:
\begin{equation*}
\mathbb P_\infty^{\nu} (|\mathfrak C| < \infty)= -\frac{1}{2 \pi} \int_{c_-^\nu}^{c_+^\nu} \frac{dz}{z} \Delta  \left( \nu, \sqrt{\nu t_\nu^3} \,  z \right) \cdot  \left(z + \frac{c_+^\nu + c_-^\nu}{2} \right) \sqrt{(c_+^\nu - z)(z-c_-^\nu)}.
\end{equation*}

\bigskip

To establish the second formula of the proposition, we also start from~\eqref{eq:boltIIPT}. Euler's formula implies that, for every $\mathfrak m \in\mathcal{M}$ and $g\in (0,1)$:
\begin{multline}
g^{|V(\mathfrak m)|} \mathbb P_\infty^{\nu}\Big(\mathfrak C (\mathbf T_\infty^\nu) = \mathfrak m\Big)
=
g \, \sum_{f \in F(\mathfrak m)} g^{\mathrm{deg}(f)/2} \, (\nu t_\nu^3)^{\mathrm{deg}(f)/2}\delta_{\mathrm{deg}(f)}\\ \cdot
\left(\prod_{f' \in F(\mathfrak m), f' \neq f} g^{(\mathrm{deg}(f')-2)/2}q_{\mathrm{deg}(f')} (\nu ,t_\nu)\right).
\end{multline}
We can then perform the exact same sequence of computations as above, but with the weight sequence $\mathbf q_g(\nu,t_\nu)$ (introduced in~\eqref{eq:zpzdcombig} of Section \ref{sec:BDGgen}) instead of the weight sequence $\mathbf q(\nu,t_\nu)$. 
We obtain in this way: 
\begin{multline*}
\mathbb E_\infty^\nu  \left[ g^{|V(\mathfrak C (\mathbf T_\infty^\nu))|}\right]\\
 =
-\frac{g}{2 \pi} \int_{c_-^\nu(g)}^{c_+^\nu(g)} \frac{dz}{z} \Delta \left(\nu,\sqrt{g \, \nu t_\nu ^3} \, z \right) \cdot  \left(z + \frac{c_+^\nu(g) + c_-^\nu(g)}{2} \right) \sqrt{(c_+^\nu(g)-z)(z-c_-^\nu(g))},
\end{multline*}
which finishes the proof.
\end{proof}

\subsection{Percolation probability and critical exponent (proof of Theorem \ref{th:expo})}

\paragraph{The root spin cluster is finite almost surely when $\nu \leq \nu_c$.}
We prove the first statement of the theorem by computing explicitly the probability that the root spin cluster in the IIPT is finite for $\nu \leq \nu_c$. To do so, we rely on the integral formula obtained in~Proposition~\ref{prop:formulaperco}.
\bigskip

Recall that $c_\pm^\nu = 1/(\sqrt{\nu t_\nu^3} y_\pm^\nu)$.
The change of variables $z = (1 - 1/y) (\nu t_\nu^3)^{-1/2}$ in \eqref{eq:probaCfinite} gives
\begin{align*}
\mathbb P_\infty^{\nu} (|\mathfrak C| < \infty) &= 
\frac{1}{2 \pi \, \nu t_\nu^3} \int_{\frac{y_-^\nu}{y_-^\nu-1}}^{\frac{y_+^\nu}{y_+^\nu-1}} \frac{dy}{y(y-1)} 
\Delta \left(\nu, \frac{y-1}{y} \right)
\left(\frac{y -1}{y} + \frac{\frac{1}{y_+^\nu} + \frac{1}{y_-^\nu}}{2} \right)\\
& \qquad \qquad\qquad\qquad\qquad\qquad\times
\sqrt{ \left( \frac{1}{y_+^\nu} - \frac{y-1}{y} \right) \left( \frac{y-1}{y} - \frac{1}{y_-^\nu}\right)},\\
& =\frac{1}{2 \pi \, \nu t_\nu^3 \, \beth^{\mathcal Z^\ps}(\nu)} \int_{\frac{y_-^\nu}{y_-^\nu-1}}^{\frac{y_+^\nu}{y_+^\nu-1}} \frac{dy}{y} \,
 \beth^{Q^+} \left(\nu, y \right)
\left(\frac{y -1}{y} + \frac{\frac{1}{y_+^\nu} + \frac{1}{y_-^\nu}}{2} \right)\\
& \qquad \qquad\qquad\qquad\qquad\qquad\times
\sqrt{ \left( \frac{1}{y_+^\nu} - \frac{y-1}{y} \right) \left( \frac{y-1}{y} - \frac{1}{y_-^\nu}\right)},
\end{align*}
where we used the expression of $\Delta$ in terms of $\beth^{Q^+}$ given in Proposition \ref{prop:asymptQk}.

To compute this integral, we perform the change of variable $y = \hat Y(\nu,U_\nu,V)$, which allows in particular to use the explicit formula for $\beth^{Q^+}\!\!(\nu ,\hat Y (\nu,U_\nu,V))$. We start by computing the new bounds of the integral, which amounts to solving for $V$ the two equations:
\[
\hat Y (\nu,U_\nu,V) = \frac{y_\pm^\nu}{y_\pm^\nu-1} \Leftrightarrow
\frac{\hat Y (\nu,U_\nu,V) - 1}{\hat Y (\nu,U_\nu,V)} = \frac{1}{y_\pm^\nu}.
\]
When $\nu \leq \nu_c$, it is easily done since
\[
\frac{1}{y_+^\nu} - \frac{\hat Y (\nu,U_\nu,V) - 1}{\hat Y (\nu,U_\nu,V)}
= \frac{2U_\nu-3}{24(1-U_\nu)} \, \frac{(1-V)^3}{V(V+1)},
\]
and
\[
\frac{\hat Y (\nu,U_\nu,V) - 1}{\hat Y (\nu,U_\nu,V)} - \frac{1}{y_-^\nu}
= \frac{2U_\nu-3}{24(1-U_\nu)} \, \frac{(V+2+\sqrt{3})^2 (V-7+4\sqrt{3})}{V(V+1)}.
\]
Since $y_+^\nu$ and $y_-^\nu$ correspond respectively to $V$ equal to $1$ and to $-2+\sqrt{3}$ in the parametrization $y = \hat Y (\nu,U_\nu,V)$, the only legitimate solution for the second equation is $V=7-4\sqrt{3}$. We then obtain:
\begin{align} \label{eq:probaHT}
\mathbb P_\infty^{\nu} (|\mathfrak C| < \infty) &= 
\frac{1}{2 \pi \, \nu t_\nu^3 \, \beth^{\mathcal Z^\ps}(\nu)} \int_{7-4\sqrt{3}}^{1} dV \, \frac{\partial_V \hat Y (\nu,U_\nu,V)}{\hat Y (\nu,U_\nu,V)} \,
\beth^{Q^+} \left(\nu, \hat Y (\nu,U_\nu,V) \right) \notag \\
& \qquad \times \left(\frac{\hat Y (\nu,U_\nu,V) -1}{\hat Y (\nu,U_\nu,V)} + \frac{\frac{1}{y_+^\nu} + \frac{1}{y_-^\nu}}{2} \right) \notag \\
& \qquad \times
\sqrt{ \left( \frac{1}{y_+^\nu} - \frac{\hat Y (\nu,U_\nu,V)-1}{\hat Y (\nu,U_\nu,V)} \right) \left( \frac{\hat Y (\nu,U_\nu,V)-1}{\hat Y (\nu,U_\nu,V)} - \frac{1}{y_-^\nu}\right)},
\end{align}
with as a bonus a fully factorized rational fraction under the square root. There are a lot of cancellations in the other terms and the integral is easy to compute with Maple. It turns out that the value of this probability is $1$ and we refer to the Maple file \cite{Maple} for the explicit computation.

\paragraph{The root spin cluster is infinite with positive probability when $\nu >\nu_c$.}
The explicit formula~\eqref{eq:ProbaPerco} that we obtain for the probability that the root spin cluster is infinite when $\nu >\nu_c$ is awful and we prefer to give a simpler argument that does not depend on this formula.
Indeed, in this temperature regime, a general argument already found in \cite[proof of Theorem 1.2]{BeCuMie} can be readily applied to our setting as follows.

We established in Proposition~\ref{prop:regNonReg} that, for $\nu>\nu_c$, the root spin cluster under $\mathbb P^\nu$ is a regular critical Boltzmann map. Proposition~\ref{eq:regcritvol} then states that for any $\nu > \nu_c$, there exists $c$ such that 
\[\mathbb P^\nu \left( |V(\mathfrak C (\mathbf T ^\nu))| \geq n \right) \underset{n\to \infty}{\sim} c \, n^{-3/2}.\]
However, for $\nu \neq \nu_c$, Corollary~\ref{cor:equivZ+} gives
\[\mathbb P^\nu \left( |V(\mathbf{T}^\nu)| \geq n \right) \underset{n\to \infty}{\sim} C \, n^{-3/2},\]
for some $C>0$ depending on $\nu$. In addition, the event $\{ |V(\mathfrak C(\mathbf T ^\nu))| \geq n \}$ is the same as the event $\{ |V(\mathfrak C(\mathbf T ^\nu )| \geq n , |V(\mathbf{T}^\nu )| \geq n \}$. Hence, we have for $n$ large enough
\[
\mathbb P^\nu \left( |V(\mathfrak C (\mathbf T ^\nu))| \geq n \, \Big| \, |V(\mathbf{T}^\nu)| \geq n \right) \geq \frac{c}{2C} > 0.
\]
Then, for every $n$ large enough and any $N\leq n$, we have $\mathbb P^\nu \left( |V(\mathfrak C (\mathbf T ^\nu))| \geq N \, \Big| \, |V(\mathbf{T}^\nu)| \geq n \right) \geq \frac{c}{2C}$. Since the event $\{ |V(\mathfrak C (\mathbf T ^\nu))| \geq N \}$ is a local event, taking the limit as $n\to \infty$ yields $\mathbb P_\infty^\nu \left( |V(\mathfrak C (\mathbf T ^\nu_\infty))| \geq N \right) \geq \frac{c}{2C}$. This last statement being true for any $N$, the root spin cluster is therefore infinite with positive probability when $\nu > \nu_c$.

\paragraph{Explicit formula and critical exponent $\beta$.}

We start from the formula for the probability that the root spin cluster is finite obtained in Proposition~\ref{prop:formulaperco}:
\begin{align*}
\mathbb P_\infty^{\nu} (|\mathfrak C| < \infty)
& =\frac{1}{2 \pi \, \nu t_\nu^3 \, \beth^{\mathcal Z^\ps}(\nu)} \int_{\frac{y_-^\nu}{y_-^\nu-1}}^{\frac{y_+^\nu}{y_+^\nu-1}} \frac{dy}{y} \,
 \beth^{Q^+} \left(\nu, y \right)
\left(\frac{y -1}{y} + \frac{\frac{1}{y_+^\nu} + \frac{1}{y_-^\nu}}{2} \right)\\
& \qquad \qquad\qquad\qquad\times
\sqrt{ \left( \frac{1}{y_+^\nu} - \frac{y-1}{y} \right) \left( \frac{y-1}{y} - \frac{1}{y_-^\nu}\right)}.
\end{align*}

To compute this integral and then do an asymptotic expansion as $\nu \to \nu_c^+$, we will use the change of variable $y = \hat Y(\nu,U_\nu,V)$ with $\nu$ and $U_\nu$ parametrized by $K_\nu$ according to~\eqref{eq:defKnu} and~\eqref{eq:defKU}.
This relation also gives that as $K_\nu \to K_{\nu_c}^+$ we have
\begin{equation} \label{eq:Knucdev}
K_\nu-K_{\nu_c} = \frac{56 +14 \sqrt 7}{81} (\nu - \nu_c) + o \left( \nu - \nu_c \right),
\end{equation}
so we can do an asymptotic expansion in $K$.
The main difficulty will come from the fact that several roots and poles of the rational functions in $V$ that are in the integral converge to $1$ at the same time and we have to be careful to analyze each contribution separately.
\bigskip

We proceed similarly as for $\nu\leq \nu_c$ and start by computing the bounds in $V$ for the integral after the change of variables. This means solving for $V$ the equation
\[
\hat Y (\nu, U_\nu,V) = \frac{y_{\pm}^\nu}{y_{\pm}^\nu-1},
\]
where we recall that $y_{\pm}^\nu$ are the positive and negative singularities of $Q_+(\nu, t_\nu, t_\nu y)$. Via the rational parametrization, they correspond respectively to $V_+(K_\nu) \in (0,1)$ and $V_-(K_\nu) < 0$, for which we have the explicit expressions from \eqref{eq:Vijcrit}:
\begin{align*}
V_+(K_\nu) &= \frac{2(K_\nu+1)^2 - \sqrt{(K_\nu^2 + 4K_\nu + 5)(3K_\nu^2 + 4K_\nu - 1)}}{3-K_\nu^2},\\
V_-(K_\nu) &= \frac{-(K_\nu^2+4K_\nu+5) + 2\sqrt{2} \sqrt{(K_\nu + 2)(K_\nu + 1)^2}}{3-K_\nu^2}.
\end{align*}
We thus have to solve for $V$ the two equations of degree three
\[
\hat Y (\nu,U_\nu,V)  = \frac{\hat Y (\nu,U_\nu,V_+(K_\nu))}{\hat Y (\nu,U_\nu,V_+(K_\nu))-1}
\quad \text{and} \quad
\hat Y (\nu,U_\nu,V)  = \frac{\hat Y (\nu,U_\nu,V_-(K_\nu))}{\hat Y (\nu,U_\nu,V_-(K_\nu))-1}.
\]
Fortunately we have the relation
\[
\hat Y (\nu,U_\nu,1/V)  = \frac{\hat Y (\nu,U_\nu,V)}{\hat Y (\nu,U_\nu,V)-1},
\]
and therefore we have to solve
\[
\hat Y (\nu,U_\nu,V)  = \hat Y (\nu,U_\nu,1/V_+(K_\nu)) 
\quad \text{and} \quad
\hat Y (\nu,U_\nu,V)  = \hat Y (\nu,U_\nu,1/V_-(K_\nu)).
\]
We obtain readily that $1/V_{+}(K_\nu)$ and $1/V_{-}(K_\nu)$ are respective solutions of the two equations. But, we can moreover check that they are in fact double roots of each equation. Furthermore, the third solution of each equation is given respectively by $V_+(K_\nu)^2$ and $V_-(K_\nu)^2$. See the Maple file \cite{Maple} for the detailed calculations.

Since $0<V_+(K_\nu)<1$, we have $0<V_+(K_\nu)^2 < V_+(K_\nu)$ and $1/V_+(K_\nu) > V_+(K_\nu)$. Similarly, we also have $1/V_-(K_\nu) < V_-(K_\nu)$ and $0 < -V_-(K_\nu) < V_+(K_\nu)$. The bounds for the integral after the change of variable $y =  \hat Y(\nu,U_\nu,V)$ are then $V_-(K_\nu)^2$ and $V_+(K_\nu)^2$. We thus have:
\begin{align*}
\mathbb P_\infty^{\nu} (|\mathfrak C| < \infty)
& = \frac{1}{2 \pi \, \nu t_\nu^3 \, \beth^{\mathcal Z^\ps}(\nu)} \int_{V_-(K_\nu)^2}^{V_+(K_\nu)^2} dV \, \frac{\partial_V \hat Y(\nu,U_\nu,V)}{\hat Y(\nu,U_\nu,V)} 
{\beth^{Q^+} \left(\nu ,\hat Y(\nu,U_\nu,V) \right)} \\
& \qquad \times 
\left(
\frac{\hat Y(\nu,U_\nu,V) -1}{\hat Y(\nu,U_\nu,V)}
+ \frac{
\frac{1}{y_+^\nu} + \frac{1}{y_-^\nu}
}{2}
\right)\\
& \qquad \times 
\sqrt{ \left( \frac{1}{y_+^\nu} - \frac{\hat Y(\nu,U_\nu,V)-1}{\hat Y(\nu,U_\nu,V)} \right) \left( \frac{\hat Y(\nu,U_\nu,V)-1}{\hat Y(\nu,U_\nu,V)} - \frac{1}{y_-^\nu}\right)}.
\end{align*}
Since we identify all the roots of the factors of the square root, it is fortunately already factorized, and only the constant coefficient needs to be identified. We can hence write:
\begin{align*}
&\sqrt{ \left( \frac{1}{y_+^\nu} - \frac{\hat Y(\nu,U_\nu,V)-1}{\hat Y(\nu,U_\nu,V)} \right) \left( \frac{\hat Y(\nu,U_\nu,V)-1}{\hat Y(\nu,U_\nu,V)} - \frac{1}{y_-^\nu}\right)}
\\
&=\sqrt{ \left( \frac{1}{y_+^\nu} - \frac{1}{\hat Y(\nu,U_\nu,1/V)} \right) \left( \frac{1}{\hat Y(\nu,U_\nu,1/V)} - \frac{1}{y_-^\nu}\right)}
\\
&=
f_1(K_\nu)
\frac{ \sqrt{\left({V_+(K_\nu)^2} -V \right) \left(V - V_-(K_\nu)^2\right)}}{ V (V+1)} 
\left(\frac{1}{V_+(K_\nu)} -V \right) \left(V - \frac{1}{V_-(K_\nu)}\right),
\end{align*}
for any $V \in [V_-(K_\nu)^2,V_+(K_\nu)^2]$, where $f_1(K_\nu)$ is an explicit rational function of $K_\nu$. 
\bigskip

We turn our attention to the other factors in the integral. A second factor is also easily factorized. Indeed we have:
\begin{align*}
\frac{1}{\nu t_\nu^3 \, \beth^{\mathcal Z^\ps}(\nu)}
& \frac{\partial_V \hat Y(\nu,U_\nu,V)}{\hat Y(\nu,U_\nu,V)} 
{\beth^{Q^+}\!\! \left(\nu ,\hat Y(\nu,U_\nu,V) \right)} =
f_2(K_\nu) \,
\frac{\left(V^{2} - V \frac{3-K_\nu^2}{(K_\nu+1)^2} +1 \right)}{\left( V - V_+(K_\nu) \right)^2 \left( V - \frac{1}{V_+(K_\nu)} \right)^2},
\end{align*}
where $f_2(K_\nu)$ is another explicit rational function of $K_\nu$.
Note that the two complex roots roots of the numerator in the last display as well as the four roots $1/V_\pm(K_\nu)$ and $V_\pm^i (K_\nu)$ all converge to $1$ as $K_\nu \to K_{\nu_c}^+$. 

The last factor in the integral is of the form:
\begin{align*}
\frac{\hat Y(\nu,U_\nu,V) -1}{\hat Y(\nu,U_\nu,V)}
+ \frac{
\frac{1}{y_+^\nu} + \frac{1}{y_-^\nu}
}{2} =\frac{\mathrm{Pol}_3(K_\nu,V)}{V(V+1)},
\end{align*}
where $\mathrm{Pol}_3(K_\nu,V)$ is a polynomial of degree $3$ in $V$ with explicit coefficients in $K_\nu$ whose roots will not matter for our analysis. 

We now have
\begin{align} \label{eq:probaLT}
\mathbb P_\infty^{\nu} (|\mathfrak C| < \infty)
& =
\frac{f(K_\nu)}{2 \pi} \int_{V_-(K_\nu)^2}^{V_+(K_\nu)^2} dV \,
\frac{\mathrm{Pol}_3(K_\nu,V) \cdot \left( V - \frac{1}{V_-(K_\nu)}\right)}{V^2 (V+1)^2} \notag \\
& \qquad \times \frac{V^2 - \frac{3-K_\nu^2}{(K_\nu+1)^2} V + 1}{\left( V - V_+(K_\nu) \right)^2 \left(\frac{1}{V_+(K_\nu)} -V \right)}
\sqrt{\left({V_+(K_\nu)^2} -V \right) \left(V - V_-(K_\nu)^2\right)},
\end{align}
with $f(K_\nu)$ an explicit rational function of $K_\nu$.

This last integral can be computed explicitly but it requires a bit of work. Indeed Maple does not seem able to compute it directly and needs our help. First, after doing a partial fraction decomposition, the integral we want to compute can be decomposed into a sum of integrals of the following form:
\[
\frac{1}{\pi}\int_{V_-(K_\nu)^2}^{V_+(K_\nu)^2} dV \,
 \sqrt{\left({V_+(K_\nu)^2} -V \right) \left(V - V_-(K_\nu)^2\right)} \,
 \frac{\mathrm{Pol}(V)}{(V - \mathrm{pole})^j},
\]
where $\mathrm{Pol}$ is a polynomial of degree $6$ in $V$, $\mathrm{pole} \in \{ 0 , -1, V_+(K_\nu) , 1/V_+(K_\nu) \}$, and $j \in \{1,2\}$. Then, after doing the change of variable $V = V_-(K_\nu)^2 + \left( V_+(K_\nu)^2 - V_-(K_\nu)^2\right) \xi$, we are left with integrals of the form
\[
\frac{1}{\pi}\int_{0}^{1} d\xi \,
 \sqrt{\xi (1- \xi) } \,
 \frac{\xi^i }{(1 - z(\mathrm{pole}) \, \xi)^j},
\]
where $\mathrm{pole} \in \{ 0 , -1, V_+(K_\nu) , 1/V_+(K_\nu) \}$, $j \in \{1,2\}$, $i \in \{ 0 , \ldots ,6 \}$, and where
\[
z(\mathrm{pole}) =\frac{ V_+(K_\nu)^2 -  V_-(K_\nu)^2}{\mathrm{pole} -  V_-(K_\nu)^2} < 1.
\]
Each of these integrals has an explicit rational expression in terms of $z(\mathrm{pole})$ and $\sqrt{1-z(\mathrm{pole})}$. Summing every contribution gives an explicit formula in terms of $K_\nu$ for the probability $\mathbb P_\infty^{\nu} (|\mathfrak C| < \infty)$. See the Maple file \cite{Maple} for the detailed computations.

The first expression we obtain with Maple is terrible. However, with some work, we were able to simplify it significantly into the following expression valid for $\nu \geq \nu_c$:
\begin{equation} \label{eq:ProbaPerco}
{\tiny
\begin{aligned}
&\mathbb P_\infty^{\nu} (|\mathfrak C| < \infty)
 ={}
\frac{1}
{2 \left(K_\nu^{2}+4 K_\nu+1\right) \left(7 K_\nu^{2}+20 K_\nu+15 \right)}
\,
\left({\frac{3 K_\nu^{2}+4 K_\nu-1}{K_\nu^{2}+4 K_\nu+5}}\right)^{1/4}
\left(
\frac{\left(3K_\nu^{2}+ 8 K_\nu+ 7\right) \left(K_\nu+1\right)}{\left(3-K_\nu^{2}\right)^{3}}
\right)^{1/2}\\
&\Bigg[
(K_\nu+1) \bigg(
\sqrt{2}\left(2+K_\nu\right)^{\frac{3}{2}} \left(145 K_\nu^{8}+1160 K_\nu^{7}+4612 K_\nu^{6}+11832 K_\nu^{5}+23430 K_\nu^{4}+39000 K_\nu^{3}+46916 K_\nu^{2}+31912 K_\nu+8753
\right)
\\
& \qquad +162 \left(K_\nu^{8}+8 K_\nu^{7}+\frac{76}{3} K_\nu^{6}+\frac{2872}{81} K_\nu^{5}+\frac{518}{81} K_\nu^{4}-\frac{3752}{81} K_\nu^{3}-\frac{5308}{81} K_\nu^{2}-\frac{3416}{81} K_\nu-\frac{1039}{81}\right)\,
\left(K_\nu^{3}+3 K_\nu^{2}+9 K_\nu+11\right)
\bigg)\\
&-
\left(3 K_\nu^{2}+4 K_\nu-1 \right)^{3/2} \, \sqrt{K_\nu^{2}+4 K_\nu+5} \,
\bigg(
32 \sqrt{2}\, \left(2+K_\nu\right)^{\frac{3}{2}} \left(K_\nu^{3}+3 K_\nu^{2}+9 K_\nu+11\right) \left(K_\nu+1\right)^{2}\\
& \qquad \qquad +
\frac{307 K_\nu^{8}}{8}+{307 K_\nu^{7}}+\frac{2179 K_\nu^{6}}{2}+{2133 K_\nu^{5}}+\frac{10185 K_\nu^{4}}{4}+{2337 K_\nu^{3}}+\frac{4211 K_\nu^{2}}{2}+{1343 K_\nu}+\frac{2579}{8}
\bigg)
\Bigg]^{1/2}
\end{aligned}
}
\end{equation}
We tried to simplify this expression further but were not able to. If the reader has some ideas on how to proceed, we would welcome any suggestions. In particular, the last big term should have a factor $(3-K^2)^3$, which would cancel out the same factor in the third term.

The asymptotic behavior of the probability when $\nu \to \nu_c^+$ is driven by the factor with exponent $1/4$. Indeed, all the other factors do not cancel at $K_{\nu_c}$, while $K_{\nu_c}$ is a root of the polynomial $3 K_\nu^{2}+4 K_\nu-1$. Doing an asymptotic expansion as $K_\nu \to K_{\nu_c}^+$ of our expression and plugging the expansion \eqref{eq:Knucdev} of $K$ in terms of $\nu$ finally gives the expansion of Theorem~\ref{th:expo}.

\subsection{Perimeter and volume critical exponents (proof of Theorem~\ref{th:mainIIPT})}

\subsubsection{Perimeter distribution}

We start from \eqref{eq:boltn} that we reproduce here:
\begin{equation*}
\mathbb P_n^\nu \left( \mathfrak C(\mathbf T_n^\nu) = \mathfrak m \right) = \frac{[t^{3n}] \prod_{f \in F(\mathfrak m)} q_{\mathrm{deg}(f)} (\nu ,t)}{[t^{3n}] \mathcal Z ^\ps ( \nu ,t)}.
\end{equation*}
Summing over maps in $\mathcal M_k$ gives
\[
\mathbb P_n^\nu \left( |\partial \mathfrak C (\mathbf T_n^\nu)| = k \right)  = \frac{[t^{3n}] \left(q_k(\nu, t) \cdot W_{ \mathbf q (\nu, t)}^{(k)} \right)}{[t^{3n}] \mathcal Z^\ps(\nu,t)}.
\]
Proposition \ref{prop:Wnut} then gives
\begin{equation}\label{eq:exprPeri}
\mathbb P_n^\nu \left( |\partial \mathfrak C (\mathbf T_n^\nu)| = k \right)  = \frac{[t^{3n}] \left( \sqrt{\nu t^3}^{-k} \, q_k(\nu, t) \cdot t^k Q^+_k (\nu, t) \right)}{[t^{3n}] \mathcal Z^\ps(\nu,t)}.
\end{equation}
From this expression, Proposition \ref{prop:asymptQk} and Proposition \ref{prop:serQty} give
\begin{equation} \label{eq:perimlaw iipt}
\lim_{n \to \infty} \mathbb P_n^\nu \left( |\partial \mathfrak C(\mathbf T_n^\nu)| = k \right) =
\delta_k(\nu) \, t_\nu^k Q_k^+(\nu, t_\nu)  + q_k(\nu,t_\nu) \, \sqrt{\nu t_\nu^3}^{-k} \, \frac{[y^k] \beth^{Q^+}(\nu,y) }{\beth^{\mathcal Z^\ps}(\nu)}.
\end{equation}

\paragraph{Case $\nu \leq \nu_c$.}
In this range of values for $\nu$, we know from Theorem~\ref{th:expo} that the root spin cluster is almost surely finite. Therefore, in this regime, the perimeter of the root spin cluster of the IIPT is a finite random variable whose law is given by the above limits.

Proposition~\ref{prop:asymptalephQ+} and Proposition~\ref{prop:asymptoq} ensure that
\begin{align*}
[y^k] \beth^{Q^+}(\nu,y) &\underset{k \to \infty}{\sim} \frac{\aleph_{4/3}^{\beth^{Q^+}}(\nu)}{\Gamma(4/3)} 2^{-k} \, k^{1/3},\\
q_k (\nu,t_\nu) \, \sqrt{\nu t_\nu^3}^{-k} &\underset{k\to \infty}{\sim} \frac{ \aleph^{Q^+} (\nu)}{\Gamma (1-\alphaa(\nu))} \, 2^k \, k^{-\alphaa(\nu)}.
\end{align*}
Corollary~\ref{cor:Q} and Corollary~\ref{coro:asymptodeltak} give similarly
\begin{align*}
t_\nu^k \, Q^+_k(\nu,t_\nu) &\underset{k \to \infty}{\sim} \frac{\aleph^{Q^+}(\nu)}{\Gamma\Big(1-\alphaa(\nu)\Big)} \, 2^{-k} \, k^{-\alphaa(\nu)},\\
\delta_k(\nu) &\underset{k \to \infty}{\sim}
\frac{1}{\beth^{\mathcal Z^\ps}(\nu)}
\frac{\aleph_{4/3}^{\beth^{Q^+}}(\nu)}{2^{4/3} \, \Gamma(4/3)} \,
2^k \, k^{1/3}.
\end{align*}
From there, it is easy to see that we have:
\begin{equation} \label{eq:asymperimiiptHT}
\mathbb P_\infty \left( |\partial \mathfrak C| = k \right) \underset{k \to \infty}{\sim} 
\frac{\aleph^{Q^+}(\nu) \, \aleph_{4/3}^{\beth^{Q^+}}(\nu)}{\beth^{\mathcal Z^\ps}(\nu) \, \Gamma(4/3) \, \Gamma\Big(1-\alphaa(\nu)\Big)}
 \,
\left(1 + 2^{-4/3} \right) \,
k^{-(\alphaa(\nu) -1/3)}
\end{equation}
where $\alphaa(\nu) = 5/3$ if $\nu < \nu_c$ and $\alphaa(\nu_c) = 7/3$. The statements on the tail distribution of the perimeter of the root spin cluster for $\nu \leq \nu_c$ follow.

\paragraph{Case $\nu > \nu_c$.}
When $\nu > \nu_c$, the root spin cluster is infinite with positive probability and we have to prove first that its perimeter is finite almost surely. One possibility would be to prove that the probabilities \eqref{eq:perimlaw iipt} sum to 1, and then establish their asymptotic behavior, which will be of order $(y_\nu - 1)^{-k}$ with polynomial corrections in $k$. A less tedious solution that suffices for our purpose is to establish exponential bounds for the probabilities $\mathbb P_n^\nu \left( |\partial \mathfrak C (\mathbf T_n^\nu)| = k \right)$ which hold uniformly in $n$.

We start again from~\eqref{eq:exprPeri}. Fix $\nu > \nu_c$ and $y \in (2,y_+^\nu)$. For every $n \geq 0$, we have
\[
\sum_{k \geq 1} [t^{3n}] \left( t^k \, Q^+_k (\nu, t) \right) y^k = [t^{3n}] \left( \sum_{k \geq 1} t^k  Q^+_k (\nu, t)  y^k \right) = [t^{3n}] Q^+ \left(\nu, t, t y \right),
\]
and therefore, for every $n \geq 1$ and $k \geq 1$, we have
\[
[t^{3n}]  t^k \, Q^+_k (\nu, t) \leq y^{-k} \, [t^{3n}] Q^+ \left(\nu, t, t y \right).
\]
To get a similar bound for the second factor of the numerator of~\eqref{eq:exprPeri}, we first observe that $q_k (\nu, t) \sqrt{\nu t^3}^{-k}$ is the formal power series in $t^3$ for $k\neq 3$ and is a formal Laurent series, whose expansion starts with $t^{-3}$ for $k=3$. Hence, for the same range of $y$ as above and $n \geq -1$:
\begin{align*}
\sum_{k \geq 1} [t^{3n}] & \left( q_k (\nu, t) \sqrt{\nu t^3}^{-k} \right)   \cdot (1-1/y)^k \\
&= [t^{3n}] \left( \sum_{k \geq 1} q_k (\nu, t) \sqrt{\nu t^3}^{-k}  \cdot  (1-1/y)^k \right),\\
&= [t^{3n}] \left( (\nu t)^{3/2} \left(\frac{1}{\sqrt{\nu t^3}} (1-1/y) \right)^3
+ \frac{(1-1/y)}{1-(1-1/y)} Q^+ \left(\nu,t, \frac{t}{1-(1-1/y)} \right)  \right),\\
& =(1-1/y)^3 \mathbf{1}_{\{n = -1\}} + [t^{3n}] \left( \frac{y-1}{y}Q^+ \left(\nu, t, ty \right) \right).
\end{align*}
Therefore, for every $n \geq 0$ and $k \geq 1$, we also have
\[
[t^{3n}]  \left( q_k (\nu, t) \sqrt{\nu t^3}^{-k} \right) \leq
(1-1/y)^{-k}
[t^{3n}] \left( \frac{y-1}{y}Q^+ \left(\nu, t, ty \right) \right).
\]
Combining the bounds we obtained, we get that for every $n \geq 1$ and every $k \geq 1$, if $y \in (2,y_+^\nu)$, then
\[
\mathbb P_n^\nu \left( |\partial C(\mathbf T_n^\nu)| = k \right) \leq \frac{
(y-1)^{-k}
[t^{3n}] \left( \frac{y-1}{y} \left( Q^+ \left( \nu,t, ty \right) \right)^2 \right)}{[t^{3n}] \mathcal Z^\ps(\nu,t)}
\]
Since the coefficients in $t$ of $Q^+ \left( \nu,t, ty \right)$ are increasing in $y$, the asymptotics of Lemma \ref{prop:serQty} show that for every $\nu > \nu_c$ and $y \in (2,y_+(\nu,t_\nu))$, there exists a constant $C(\nu)$ such that, for every $k \geq 1$ and every $n \geq 1$, one has
\[
\mathbb P_n^\nu \left( |\partial C (\mathbf T_n^\nu)| = k \right)
\leq C(\nu) \, (y-1)^{-k}.
\]
Since $y >2$, this proves that the perimeters of the root spin cluster under $\mathbb P_n^\nu$ for $n \geq 1$ form a tight sequence of random variables. Since we know that each probability $\mathbb P_n^\nu \left( |\partial \mathfrak C (\mathbf T_n^\nu)| = k \right)$ converges, the random variables $|\partial \mathfrak C (\mathbf T_n^\nu)|$ converge in law towards a limiting random variable with exponential tail. This finishes the proof of the statement of Theorem~\ref{th:mainIIPT} on the distribution of the perimeter of the root spin cluster in the low temperature regime.

\subsubsection{Volume distribution}

In this whole section, we only consider $\nu \leq \nu_c$. We start with the second formula of Proposition~\ref{prop:formulaperco}: for $\nu >0$ and $g \in (0,1)$ we have
\begin{align*}
\mathbb E_\infty^\nu & \left[ g^{|V(\mathfrak C (\mathbf T_\infty^\nu))|}\right]\\
& =
-\frac{g}{2 \pi} \int_{c_-^\nu(g)}^{c_+^\nu(g)} \frac{dz}{z} \Delta \left(\nu,\sqrt{g \, \nu t_\nu ^3} \, z \right) \cdot  \left(z + \frac{c_+^\nu(g) + c_-^\nu(g)}{2} \right) \sqrt{(c_+^\nu(g)-z)(z-c_-^\nu(g))},\\
& =
-\frac{g}{2 \pi \, \beth^{\mathcal Z^\ps}(\nu)} \int_{c_-^\nu(g)}^{c_+^\nu(g)} dz \, \frac{\sqrt{g \, \nu t_\nu ^3}}{1-\sqrt{g \, \nu t_\nu ^3} \, z} \beth^{Q^+} \left(\nu, \frac{1}{1-\sqrt{g \, \nu t_\nu ^3} \, z} \right) \cdot  \left(z + \frac{c_+^\nu(g) + c_-^\nu(g)}{2} \right)\\
& \qquad \qquad \qquad \qquad \qquad  \times \sqrt{(c_+^\nu(g)-z)(z-c_-^\nu(g))}.
\end{align*}

We want an asymptotic expansion as $g\to 1^-$ of the previous quantity. It relies on two things: first we  established the asymptotic expansion of $z^+(g)$ and $z^\diamond (g)$ at $g = 1$ in Lemma~\ref{lem:asymzgtcrit}; second we established the analytic properties and singular behavior at $y=2$ of the function $y \mapsto \beth^{Q^+}(\nu,y)$ in Proposition~\ref{prop:asymptalephQ+}. Combining the two will lead to the desired expansion.

\bigskip

The change of variable $z(\nu,g,\xi)=c_-^\nu(g) + \left(c_+^\nu(g) - c_-^\nu(g) \right) \xi$ in the last display gives the following identity:
\[
\mathbb E_\infty^\nu  \left[ g^{|V(\mathfrak C (\mathbf T_\infty^\nu))|}\right]
=
\frac{1}{\pi} \int_0^1 d\xi \, \sqrt{\xi (1-\xi)} \, \phi_1(\nu,g,\xi) \, \phi_2(\nu,g,\xi),
\]
with
\begin{align*}
\phi_1(\nu,g,\xi) & = -\frac{g}{2} \, \frac{\sqrt{g \, \nu t_\nu ^3}}{1-\sqrt{g \, \nu t_\nu ^3} \, z(\nu,g,\xi)} \,  \left(z(\nu,g,\xi) + \frac{c_+^\nu(g) + c_-^\nu(g)}{2} \right) \,  \left(c_+^\nu(g) - c_-^\nu(g) \right)^2,\\
\phi_2(\nu,g,\xi) & = \frac{1}{\beth^{\mathcal Z^\ps}(\nu)} \, \beth^{Q^+} \left(\nu, \frac{1}{1-\sqrt{g \, \nu t_\nu ^3} \, z(\nu,g,\xi)} \right).
\end{align*}
The asymptotic expansions for $z^+_\nu(g)$ and $z^\diamond _\nu(g)$ obtained in~\eqref{eq:zdgexp} and~\eqref{eq:zpgexp} can be translated into the following expansions for $c_+^\nu(g)$ and $c_-^\nu(g)$ via~\eqref{eq:cpmqg}:
\begin{align*}
c_+^\nu(g) &= \frac{1}{2 \sqrt{\nu t^3_\nu}} - \left( {\sqrt{z^+_\nu}} \aleph^\diamond (\nu) + z^+_\nu \aleph^\bullet(\nu) \right)^{-\frac{1}{\alphaa(\nu)-1/2}} \, \left( 1-g \right)^{\frac{1}{\alphaa(\nu)-1/2}} + o \left(\left( 1-g \right)^{\frac{1}{\alphaa(\nu)-1/2}} \right), \\
c_-^\nu(g) &= \frac{1}{ \sqrt{\nu t^3_\nu} \, y_-(\nu,t_\nu)} + o \left(\left( 1-g \right)^{\frac{1}{\alphaa(\nu)-1/2}} \right),
\end{align*}
where every constant is explicit. This in turn gives the following expansion in $g$ for $\phi_1$
\[
\phi_1(\nu,g,\xi) = \phi_1(\nu,1,\xi) + \mathcal O  \left(\left( 1-g \right)^{\frac{1}{\alphaa(\nu)-1/2}} \right),
\]
where the error term is uniform in $\xi \in [0,1]$. Therefore we have
\begin{equation} \label{eq:volcluster1}
\mathbb E_\infty^\nu  \left[ g^{|V(\mathfrak C (\mathbf T_\infty^\nu))|}\right]
=
\frac{1}{\pi} \int_0^1 d\xi \, \sqrt{\xi (1-\xi)} \, \phi_1(\nu,1,\xi) \, \phi_2(\nu,g,\xi) + \mathcal O  \left(\left( 1-g \right)^{\frac{1}{\alphaa(\nu)-1/2}} \right).
\end{equation}
Indeed, we have
\begin{align*}
&\left|
\mathbb E_\infty^\nu  \left[ g^{|V(\mathfrak C (\mathbf T_\infty^\nu))|}\right]
- \frac{1}{\pi} \int_0^1 d\xi \, \sqrt{\xi (1-\xi)} \, \phi_1(\nu,1,\xi) \, \phi_2(\nu,g,\xi)
\right|\\
& \qquad \qquad\qquad \leq
\frac{1}{\pi} \int_0^1 d\xi \, \sqrt{\xi (1-\xi)} \, |\phi_1(\nu,1,\xi)-
\phi_1(\nu,g,\xi)| \, \phi_2(\nu,g,\xi),\\
& \qquad \qquad\qquad \leq \mathrm{C}(\nu) \,
\left( 1-g \right)^{\frac{1}{\alphaa(\nu)-1/2}}
\frac{1}{\pi} \int_0^1 d\xi \, \sqrt{\xi (1-\xi)} \,  \, \phi_2(\nu,g,\xi),
\end{align*}
with $\mathrm C(\nu)$ a positive constant.  We will see that the integral in the last display is bounded with $g$, which will establish~\eqref{eq:volcluster1}.

\bigskip

To study the dependency in $g$ of the integral in~\eqref{eq:volcluster1}, we will proceed in a similar fashion than in the proof of Lemma~\ref{lem:asymptfbfg} by subtracting for $\beth^{Q^+}$ its singular expansion obtained in Proposition~\ref{prop:asymptalephQ+}. Recall from this proposition, that $\beth^{Q^+}(\nu,y)$ is analytic on $\{y\,\Big|\,|y|\leq 2 \text{ and }y\neq 2\}$, and has a unique dominant singularity at $y_\nu=2$. Hence, $\phi_2(\nu,g,\xi)$ is analytic for $(g,\xi)\in (0,1]\times[0,1]\backslash\{(1,1)\}$.

Hence, if we set
\[
\varphi(\nu,g,\xi) = \phi_1(\nu,1,\xi) \phi_2(\nu,g,\xi) - 
\frac{\phi_1(\nu,1,1)}{\beth^{\mathcal Z^\ps}(\nu)} \, 
\sum_{i \in \left\{\frac{4}{3},\frac{2}{3},\frac{1}{3},0,-\frac{1}{3},-\frac{2}{3}\right\}}\frac{\aleph_i^{\beth^{Q^+}} (\nu)}{\left( 2 - \frac{1}{1-\sqrt{g \, \nu t_\nu ^3} \, z(\nu,g,\xi)} \right)^{i}},
\]
then the function 
\[
g \mapsto \frac{1}{\pi} \int_0^1 d\xi \, \sqrt{\xi (1-\xi)} \, \varphi(\nu,g,\xi)
\]
is differentiable in $g$ and has a non singular development up to an error term $o(1-g)$.

The first singular terms in the development of $\mathbb E_\infty^\nu  \left[ g^{|V(\mathfrak C (\mathbf T_\infty^\nu))|}\right]$ then come from the integrals
\begin{equation*}
I_i(\nu,g)=\frac{\phi_1(\nu,1,1) \, \aleph_i^{\beth^{Q^+}} (\nu)}{\beth^{\mathcal Z^\ps}(\nu)} \, \frac{1}{\pi} \int_0^1 {d\xi}  \,
\sqrt{\xi(1-\xi)} \,
\left( 2 - \frac{1}{1-\sqrt{g \, \nu t_\nu ^3} \, z(\nu,g,\xi)} \right)^{-i},
\end{equation*}
for $i \in \left\{ \frac{4}{3}, \frac{2}{3}, \frac{1}{3},0,-\frac{1}{3},-\frac{2}{3}\right\}$. We will only detail the case when $i = 4/3$ as the other cases are treated similarly and have singular terms of smaller order.

We can simplify $I_{4/3}$ into
\begin{multline*}
I_{4/3}(\nu,g) = \frac{\phi_1(\nu,1,1) \, \aleph_{4/3}^{\beth^{Q^+}} (\nu)}{\beth^{\mathcal Z^\ps}(\nu)} \,
\left( \frac{1 - 2 \sqrt{g \, \nu t_\nu ^3} \, c_-^\nu(g)}{1 - \sqrt{g \, \nu t_\nu ^3} \, c_-^\nu(g)} \right)^{-4/3} \\
\times
\frac{1}{\pi} \int_0^1 {d\xi}  \,
\sqrt{\xi(1-\xi)} \,
\left(
\frac{
1- 
\dfrac{2 \sqrt{g \, \nu t_\nu ^3} \, (c_+^\nu(g) - c_-^\nu(g))}{1-2 \sqrt{g \, \nu t_\nu ^3} \, c_-^\nu(g)} \xi
}
{
1- 
\dfrac{ \sqrt{g \, \nu t_\nu ^3} \, (c_+^\nu(g)-c_-^\nu(g))}{1- \sqrt{g \, \nu t_\nu ^3} \, c_-^\nu(g)} \xi
}
\right)^{-4/3}.
\end{multline*}
In this last quantity, we have as $g \to 1^-$:
\[
\frac{ \sqrt{g \, \nu t_\nu ^3} \, (c_+^\nu(g)-c_-^\nu(g))}{1- \sqrt{g \, \nu t_\nu ^3} \, c_-^\nu(g)} \to \frac{\frac{1}{2} - \frac{1}{y_-^\nu}}{1 - \frac{1}{y_-^\nu}} < 1 
\quad \text{and} \quad 
\frac{ 2 \sqrt{g \, \nu t_\nu ^3} \, (c_+^\nu(g)-c_-^\nu(g))}{1- 2 \sqrt{g \, \nu t_\nu ^3} \, c_-^\nu(g)} \to 1^-.
\]
The same line of reasoning as before shows that the leading singular term will come from the numerator and will be given by the leading singular term of
\begin{align*}
I(\nu,g) & =\frac{\phi_1(\nu,1,1) \, \aleph_{4/3}^{\beth^{Q^+}} (\nu)}{\beth^{\mathcal Z^\ps}(\nu)} \,
\left( \frac{1 - 2 \sqrt{g \, \nu t_\nu ^3} \, c_-^\nu}{1 - \sqrt{g \, \nu t_\nu ^3} \, c_-^\nu} \right)^{-4/3} \, \left( 1- 
\frac{ \sqrt{\nu t_\nu ^3} \, (c_+^\nu-c_-^\nu)}{1- \sqrt{\nu t_\nu ^3} \, c_-^\nu} \right)^{4/3} \\
& \qquad \times \frac{1}{\pi} \int_0^1 {d\xi}  \,
\xi^{1/2} (1-\xi)^{1/2} \,
\left(
1- 
\frac{2 \sqrt{g \, \nu t_\nu ^3} \, (c_+^\nu(g) - c_-^\nu(g))}{1-2 \sqrt{g \, \nu t_\nu ^3} \, c_-^\nu(g)} \xi
\right)^{-4/3},\\
& = \frac{\phi_1(\nu,1,1) \, \aleph_{4/3}^{\beth^{Q^+}} (\nu)}{\beth^{\mathcal Z^\ps}(\nu)} \, \left(2 -\frac{4}{y_-^\nu} \right)^{-4/3} \, 
_2F_1  \left(\frac{4}{3},\frac{3}{2}; 3; \frac{2 \sqrt{g \, \nu t_\nu ^3} \, (c_+^\nu(g) - c_-^\nu(g))}{1-2 \sqrt{g \, \nu t_\nu ^3} \, c_-^\nu(g)} \right).
\end{align*}
We also have
\begin{multline*}
\frac{2 \sqrt{g \, \nu t_\nu ^3} \, (c_+^\nu(g) - c_-^\nu(g))}{1-2 \sqrt{g \, \nu t_\nu ^3} \, c_-^\nu(g)} \\
=
1- \frac{2 \sqrt{ \nu t_\nu ^3} \,  \left( {\sqrt{z^+}} \aleph^\diamond (\nu) + z^+ \aleph^\bullet(\nu) \right)^{-\frac{1}{\alphaa(\nu)-1/2}} }{1-2/y_-^\nu}
\, \left( 1-g \right)^{\frac{1}{\alphaa(\nu)-1/2}} + o \left(\left( 1-g \right)^{\frac{1}{\alphaa(\nu)-1/2}} \right).
\end{multline*}
So that, by standard singular expansion of hypergeometric functions at $1$, the leading singular term of $I(\nu,g)$ is
\[
- \mathrm{C}(\nu) \,  \left( 1-g \right)^{\frac{1}{6(\alphaa(\nu)-1/2)}}
\]
with
\begin{multline*}
\mathrm{C}(\nu) =  \frac{\phi_1(\nu,1,1) \, \aleph_{4/3}^{\beth^{Q^+}} (\nu)}{\beth^{\mathcal Z^\ps}(\nu)} \, \left(2 -\frac{4}{y_-^\nu} \right)^{-4/3} \,  \frac{54 \cdot 2^{\frac{1}{3}} \cdot\Gamma(\frac{2}{3})^{3}}{\pi^{2}} \\
\cdot \left( \frac{2 \sqrt{ \nu t_\nu ^3} \,  \left( {\sqrt{z^+}} \aleph^\diamond (\nu) + z^+ \aleph^\bullet(\nu) \right)^{-\frac{1}{\alphaa(\nu)-1/2}} }{1-2/y_-^\nu} \right)^{1/6}.
\end{multline*}
This term is also the leading singular term of $\mathbb E_\infty^\nu \left[ g^{|V(\mathfrak C (\mathbf T_\infty^\nu))|}\right]$ and we finally get:
\begin{align*}
\mathbb E_\infty^\nu \left[ g^{|V(\mathfrak C (\mathbf T_\infty^\nu))|}\right]
& = 1 -  \mathrm{C}(\nu) \,  \left( 1-g \right)^{\frac{1}{6(\alphaa(\nu)-1/2)}}+
o \left( \left( 1-g \right)^{\frac{1}{6(\alphaa(\nu)-1/2)}} \right).
\end{align*}

\bigskip

As in the proof of Theorem~\ref{th:mainBoltzmann}, this expansion is not quite enough to get the asymptotic behavior of $\mathbb P^\nu_\infty \left(|V(\mathfrak C (\mathbf T_\infty^\nu))| =n \right)$ because of possible complex singularities of modulus $1$ (which we do not believe to exist). It is however enough to get the tail probabilities:
\[
\frac{1-\mathbb E_\infty^\nu \left[ g^{|V(\mathfrak C (\mathbf T_\infty^\nu))|}\right]}{1-g}
= \sum_{n \geq 0} \mathbb P^\nu_\infty \left(|V(\mathfrak C (\mathbf T_\infty^\nu))| > n \right) \, g^n
\underset{g \to 1^-}{\sim}
\mathrm{C}(\nu) \,  \left( 1-g \right)^{\frac{1}{6(\alphaa(\nu)-1/2)}-1},
\]
and the classical Tauberian theorem (see e.g. Theorem VI.13 of \cite{FS} and the following discussion) gives
\[
\mathbb P^\nu_\infty \left(|V(\mathfrak C (\mathbf T_\infty^\nu))| > n \right)
\underset{n \to \infty}{\sim}
\frac{\mathrm{C}(\nu)}{\Gamma \left(1 -  \frac{1}{6(\alphaa(\nu)-1/2)}\right)} \,  n^{-\frac{1}{6(\alphaa(\nu)-1/2)}},
\]
finishing the proof of Theorem~\ref{th:mainIIPT}.

\section{Cluster interface and Looptrees}\label{sec:looptrees}

In this section, we study the geometry of the root spin cluster boundary seen as a discrete looptree. Following the work of Curien and Kortchemski in the setting of site percolation on the UIPT in \cite{CuKo}, we encode the boundary of the root spin cluster by a two type Galton--Watson tree. The asymptotic properties of this tree can be fully characterized thanks to the enumerative results on triangulations with simple monochromatic boundary obtained in section~\ref{sec:monochromatic}. The arguments developed in \cite{CuKo} can then be carried out to our setting without major difficulties to prove Theorem~\ref{th:looptrees}.

\subsection{Encoding by a two type Galton--Watson tree}

Recall the definition of discrete looptrees given in the introduction and the construction that associates to a tree $\tau$, the looptree $\mathrm{Loop}(\tau)$, as illustrated in Figure~\ref{fig:deflooptree}. 

We can define a different mapping that associates a tree $\mathrm{Tree}_2(L)$ with alternating white and black vertices and rooted at a white vertex to a looptree $L$  as follows: we put a black vertex inside each cycle of $L$ (except the exterior cycle), and connect it to each white vertex of $L$ that belongs to this cycle. Notice that the perimeter of $L$ is equal to the total number of vertices of $\mathrm{Tree}_2(L)$. This mapping is bijective, we denote by $\mathrm{Loop}_2(\tau)$ the looptree associated to an alternating white and black tree $\tau$ rooted at a white vertex. The index $2$ in these two mappings is to stress out the fact that the corresponding trees have two alternating colors as we have the similar construction $\mathrm{Loop}(\tau)$ for uncolored trees. See Figure \ref{fig:looptree} for an illustration.
\begin{figure}[!ht]
\centering
\includegraphics[width=0.9\linewidth]{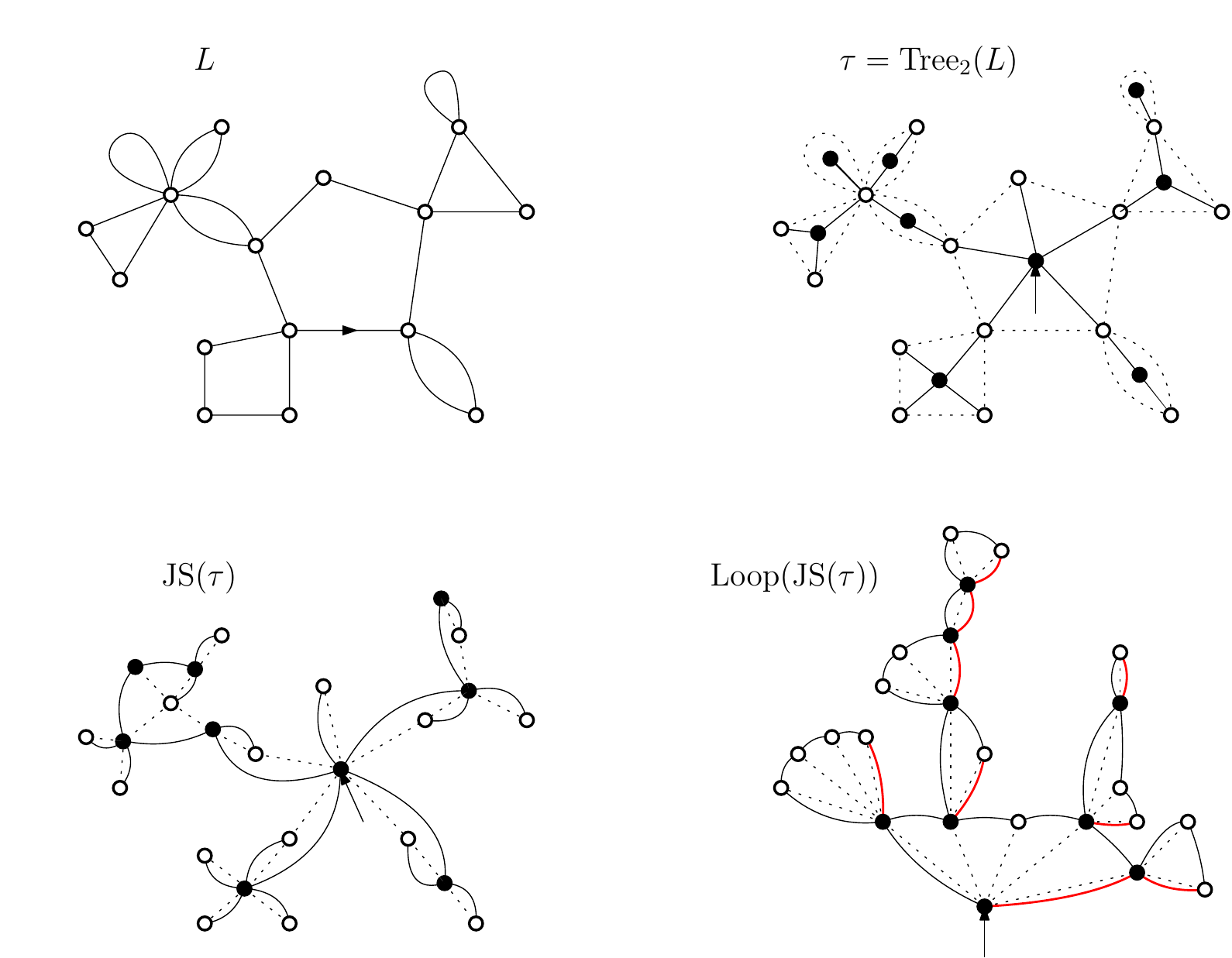}
\caption{\label{fig:looptree} A looptree $L$, its associated alternating two type tree $\tau = \mathrm{Tree}_2(L)$, the tree associated by Janson-Stefánsson bijection $\mathrm{JS}(\tau)$, and the looptree $\mathrm{Loop}(\mathrm{JS}(\tau))$. Contracting the red edges of $\mathrm{Loop}(\mathrm{JS}(\tau))$ gives back $L$.}
\end{figure}

\bigskip

For $\nu >0$, define the two probability measures $\mu^\circ_\nu$ and $\mu^\bullet_\nu$ on $\mathbb Z^+$ by setting, for $k \geq 0$:
\begin{align*}
\mu^\circ_\nu (k) &= \frac{1}{1+Z^+(\nu,t_\nu, t_\nu x_\nu)} \, \left(\frac{Z^+(\nu,t_\nu, t_\nu x_\nu)}{1+Z^+(\nu,t_\nu, t_\nu x_\nu)} \right)^k,\\
\mu^\bullet_\nu (k) &=  \frac{(t_\nu x_\nu)^{k+1} \, Z_{k+1}^+ (\nu,t_\nu)}{Z^+(\nu,t_\nu, t_\nu x_\nu)}.
\end{align*}
Denote by $\mathrm{GW}_{\mu^\circ_\nu,\mu^\bullet_\nu}$ the law of an alternating two type Galton-Watson tree with respective offspring distributions  $\mu^\circ_\nu$ for white vertices (having black offspring) and $\mu^\bullet_\nu$ for black vertices (having white offspring).

\begin{prop} \label{prop:2typeGW}
For every $\nu >0$ and every alternating two type tree $\tau$ rooted at a white vertex, we have
\begin{align*}
\mathbb P_\infty^\nu &(\mathrm{Tree}_2(\partial\mathfrak C (\mathbf T^\nu_\infty)) = \tau | |\mathfrak H(T^\nu_\infty)| < \infty) \\
& \qquad = \frac{\delta_{|V(\tau)| -1}(\nu) \, x_\nu^{1 - |V(\tau)|} \, \left(1 + Z^+(\nu,t_\nu, t_\nu x_\nu)\right)^{|V(\tau)|+1}}{Z^+(\nu,t_\nu, t_\nu x_\nu)}
\mathrm{GW}_{\mu^\circ_\nu,\mu^\bullet_\nu}(\tau).
\end{align*}
In particular, the law of $\mathrm{Tree}_2(\partial \mathfrak C(\mathbf T^\nu_\infty))$ conditioned on $|\partial \mathfrak C(\mathbf T^\nu_\infty)| = n$ and $|\mathfrak H(\mathbf T^\nu_\infty)| < \infty$ is $\mathrm{GW}_{\mu^\circ_\nu,\mu^\bullet_\nu}$ conditioned to have $n$ vertices.
\end{prop}

\begin{proof}
Fix $\mathfrak h$ a finite triangulation with spins and a boundary (not necessarily simple) with positive spins. Using Proposition \ref{prop:asymptQk} we have:
\begin{align*}
\mathbb P_n^\nu \left( \mathfrak H (\mathfrak C) = \mathfrak h \right)
&= \frac{[t^{3n}] \left( \nu^{m(\mathfrak h)} t^{|\mathfrak h|} \cdot (\nu t)^{-|\partial \mathfrak h|/2}  q_{|\partial \mathfrak h|}(\nu,t) \right)}{[t^{3n}] \mathcal Z^\ps (\nu,t)},\\
& \underset{n \to \infty}{\rightarrow} \nu^{m(\mathfrak h)} t_\nu^{|\mathfrak h|} \cdot t_\nu^{|\partial \mathfrak h|} \delta_{|\partial \mathfrak h|}(\nu).
\end{align*}
As a consequence, by continuity with respect to the local topology, we have:
\begin{equation} \label{eq:hullproba}
\mathbb P_\infty^\nu (\mathfrak H (\mathfrak C) = \mathfrak h) =  \nu^{m(\mathfrak h)} t_\nu^{|\mathfrak h|} \cdot t_\nu^{|\partial \mathfrak h|} \delta_{|\partial \mathfrak h|}(\nu).
\end{equation}

Fix $\tau$ an alternating white and black tree rooted at a white vertex and denote by $|\partial \tau|$ the perimeter of $\mathrm{Loop}_2(\tau)$. Summing \eqref{eq:hullproba} over every hull having the same looptree as boundary we get:
\begin{align*}
\mathbb P_\infty^\nu (\mathrm{Tree}(\partial\mathfrak C) = \tau | |\mathfrak H(\mathfrak C)| < \infty)
& = \delta_{|\partial \tau|}(\nu) \, t_\nu^{|\partial \tau|} \, \sum_{\mathfrak h \, : \, \partial \mathfrak h = \mathrm{Loop}(\tau)}
\nu^{m(\mathfrak h)} t_\nu^{|\mathfrak h|},\\
& = \delta_{|\partial \tau|}(\nu) \, t_\nu^{|\partial \tau|} \,
\prod_{v \in V_\bullet(\tau)} Z_{\mathrm{deg}(v)}^+ (\nu,t_\nu),\\
& = \delta_{|\partial \tau|}(\nu) \, 
\prod_{v \in V_\bullet(\tau)} t_\nu^{\mathrm{deg}(v)} \, Z_{\mathrm{deg}(v)}^+ (\nu,t_\nu),\\
&= \delta_{|\partial \tau|}(\nu) \, x_\nu^{- | \partial \tau|} \, \prod_{v \in V_\bullet(\tau)} (t_\nu x_\nu)^{\mathrm{deg}(v)} \, Z_{\mathrm{deg}(v)}^+ (\nu,t_\nu).
\end{align*}
By construction of $\mathrm{Loop}_2(\tau)$, we have $\sum_{u \in V_\bullet(\tau)} \mathrm{deg} (u)  = |\partial \tau|$. Moreover, since $\mathrm{Loop}_2(\tau)$ is an alternating two-type tree, we clearly have:
\[\sum_{u \in V_\circ(\tau)} \mathrm{deg} (u)  = |E(\tau)| = |V(\tau)|  +1= |V_\circ(\tau)| + |V_\bullet(\tau)| +1.
\]
Hence, we get for any $ a > 0$
\begin{align*}
\frac{1}{a} & = \left(\prod_{u \in V_\circ(\tau)} \frac{1}{a^{\mathrm{deg}(u) - 1}} \right) \, \left(\prod_{v \in V_\bullet(\tau)} a \right),\\
& = \left(1 - \frac{1}{a} \right)^{-|V(\tau)|} \, \left(\prod_{u \in V_\circ(\tau)} \left(1 - \frac{1}{a} \right) \frac{1}{a^{\mathrm{deg}(u) - 1}} \right) \, \left(\prod_{v \in V_\bullet(\tau)} (a-1) \right).
\end{align*}
Setting 
\[
a = 1 + \frac{1}{Z^+(\nu,t_\nu, t_\nu x_\nu)},
\]
we get
\begin{align*}
\mathbb P_\infty^\nu &(\mathrm{Tree}_2(\partial\mathfrak C) = \tau | |\mathfrak H(\mathfrak C)| < \infty) \\
& \qquad = \delta_{|V(\tau)| -1}(\nu) \, x_\nu^{1 - |V(\tau)|} \, \left(1 + Z^+(\nu,t_\nu, t_\nu x_\nu)\right)^{|V(\tau)|} \, \frac{1+Z^+(\nu,t_\nu, t_\nu x_\nu)}{Z^+(\nu,t_\nu, t_\nu x_\nu)} \\
& \qquad \qquad \times
\left( \prod_{u \in V_\circ(\tau)}
\frac{1}{1+Z^+(\nu,t_\nu, t_\nu x_\nu)}
\left(\frac{Z^+(\nu,t_\nu, t_\nu x_\nu)}{1+Z^+(\nu,t_\nu, t_\nu x_\nu)} \right)^{\mathrm{deg}(u)-1} 
\right) \\
& \qquad \qquad \times
\left( \prod_{v \in V_\bullet(\tau)} \frac{(t_\nu x_\nu)^{\mathrm{deg}(v)} \, Z_{\mathrm{deg}(v)}^+ (\nu,t_\nu)}{Z^+(\nu,t_\nu, t_\nu x_\nu)} \right),
\end{align*}
which finishes the proof, since the tree is rooted at a white vertex.
\end{proof}

The case of finite random triangulations is very similar:

\begin{prop} \label{prop:2typeGWbolt}
For every $\nu >0$ and every alternating two type tree $\tau$ rooted at a white vertex, we have
\begin{align*}
\mathbb P^\nu &(\mathrm{Tree}_2(\partial\mathfrak C (\mathbf T^\nu)) = \tau ) \\
& \qquad =
\frac{q_{|V(\tau)| -1}(\nu,t_\nu ) \, (\sqrt{\nu t_\nu^3} \,x_\nu)^{1 - |V(\tau)|} \, \left(1 + Z^+(\nu,t_\nu, t_\nu x_\nu)\right)^{|V(\tau)|+1}}{\mathcal Z^\ps(\nu,t_\nu) \, Z^+(\nu,t_\nu, t_\nu x_\nu)}
\mathrm{GW}_{\mu^\circ_\nu,\mu^\bullet_\nu}(\tau).
\end{align*}
In particular, the law of $\mathrm{Tree}_2(\partial \mathfrak C(\mathbf T^\nu))$ conditioned on $|\partial \mathfrak C(\mathbf T^\nu)| = n$ is $\mathrm{GW}_{\mu^\circ_\nu,\mu^\bullet_\nu}$ conditioned to have $n$ vertices.
\end{prop}
\begin{proof}
By definition we have
\[
\mathbb P^\nu (\mathfrak H (\mathfrak C) = \mathfrak h) = \frac{1}{\mathcal Z^\ps (\nu,t_\nu)} \nu^{m(\mathfrak h)} t_\nu^{|\mathfrak h|} \cdot (\nu \, t_\nu)^{-|\partial \mathfrak h|/2} q_{|\partial \mathfrak h|}(\nu).
\]
From there, the same calculation as in the proof of Proposition~\ref{prop:2typeGW} gives
\[
\mathbb P^\nu (\mathrm{Tree}(\partial\mathfrak C) = \tau )
= \frac{q_{|\partial \tau|}(\nu) \, (\nu t_\nu^3)^{-|\partial \tau|/2} \, x_\nu^{- | \partial \tau|}}
{\mathcal Z^\ps (\nu,t_\nu)} \, 
 \, \prod_{v \in V_\bullet(\tau)} (t_\nu x_\nu)^{\mathrm{deg}(v)} \, Z_{\mathrm{deg}(v)}^+ (\nu,t_\nu),
\]
and the result follows as in the proof of Proposition~\ref{prop:2typeGW}.
\end{proof}

\subsection{Proof of Theorem~\ref{th:looptrees}}

We are now ready to prove Theorem \ref{th:looptrees}. The same proof can originally be found in \cite[Proof of Theorem 1.2]{CuKo} in the setting of site percolation on the UIPT. We adapt it here in our setting for the sake of completeness.

\begin{proof}[Proof of Theorem \ref{th:looptrees}]
The proofs for $\partial \mathfrak C_\infty^\nu(n)$ and for $\partial \mathfrak C^\nu(n)$ are identical since both of them rely on the law of the associated two type Galton--Watson trees, which is the same by Propositions \ref{prop:2typeGW} and \ref{prop:2typeGWbolt}. We only write the details in the infinite volume case.

Recall from the introduction that $\partial \mathfrak C_\infty^\nu(n)$ denotes $\partial \mathfrak C(\mathbf T_\infty^\nu)$ conditioned on the event that $\mathfrak H (\mathbf T_\infty^\nu)$ is finite and that the perimeter of $\mathfrak C(\mathbf T_\infty^\nu)$ is $n$. Proposition \ref{prop:2typeGW} ensures that $\mathrm{Tree}_2(\partial \mathfrak C_\infty^\nu(n))$ has law $\mathrm{GW}_{\mu^\circ_\nu,\mu^\bullet_\nu}^{(n)}$, which is the law $\mathrm{GW}_{\mu^\circ_\nu,\mu^\bullet_\nu}$ conditioned on the event that the tree has $n$ vertices.

We can use the Janson--Stef\'ansson bijection \cite{JStree} to map alternating two type trees rooted at a white vertex and uncolored rooted trees. For an alternating two type tree $\tau$, we denote  by $\mathrm{JS}(\tau)$ its image by the Janson--Stef\'ansson bijection. It is not needed to know details about this bijection as all the arguments used to prove Theorem \ref{th:looptrees} are established in \cite{CuKo}. The illustration given in Figure \ref{fig:looptree} should suffice for our purpose.

It is proved in \cite[Appendix A]{JStree} and in \cite[Proposition 3.6]{CuKo}) that the Janson--Stef\'ansson bijection 
maps a two-type Galton--Watson tree with law $\mathrm{GW}_{\mu^\circ_\nu,\mu^\bullet_\nu}$ conditioned to have $n$ vertices to a one type Galton--Watson tree conditioned to have $n$ vertices and with reproduction law $\mu_\nu$ defined by
\[
\mu_\nu (k) =
\begin{cases}
\frac{1}{1+Z^+(\nu,t_\nu, t_\nu x_\nu)} & \text{for $k=0$},\\
\frac{Z^+(\nu,t_\nu, t_\nu x_\nu)}{1+Z^+(\nu,t_\nu, t_\nu x_\nu)} \, \mu^\bullet_\nu(k-1) & \text{for $k \geq 1$}.
\end{cases}
\]
Therefore, $\mathrm{JS} ( \mathrm{Tree}_2(\partial \mathfrak C_\infty^\nu(n)) )$ is a Galton--Watson tree with reproduction law $\mu_\nu$ conditioned to have $n$ vertices.

The generating function of $\mu_\nu$ is given by
\[
G_{\mu_\nu} ( z) = \sum_{k \geq 0} \mu_\nu(k) \, z ^k = \frac{1+Z^+(\nu,t_\nu, t_\nu x_\nu \, z)}{1+Z^+(\nu,t_\nu, t_\nu x_\nu)}.
\]
According to Lemma \ref{lem:expZplusxc}, this function is differentiable at $z=1$ and we have
\begin{equation} \label{eq:looptreeder}
G_{\mu_\nu} ^\prime (1) = \frac{-\aleph_1^{Z^+}(\nu)}{1+ Z^+(\nu,t_\nu ,t_\nu x_\nu)},
\end{equation}
Hence, the distribution is critical when $\nu < \nu_c$ and subcritical when $\nu \geq \nu_c$.
We can separate into two cases depending on whether the reproduction law $\mu_\nu$ is critical or subcritical.
\paragraph{Case $\nu \geq \nu_c$.} In this case, the reproduction law $\mu_\nu$ is subcritical. It is proved in \cite[Theorem 5.5]{JScondensation} that, with probability tending to $1$ as $n\to \infty$, the random tree $\mathrm{JS} ( \mathrm{Tree}_2(\partial \mathfrak C_\infty^\nu(n)) )$ has a unique vertex $v_n^\star$ of maximal degree and that, in probability,
\[
\frac{\mathrm{deg}(v_n^\star)}{n} \underset{n \to \infty}{\rightarrow} 1 - G_{\mu_\nu} ^\prime (1).
\]
By \cite[Corollary 2]{KorGW}, the maximal size of the connected components of $\mathrm{JS} ( \mathrm{Tree}_2(\partial \mathfrak C_\infty^\nu(n)) ) \setminus \{ v_n^\star\}$ divided by $n$ vanishes in probability as $n \to \infty$. Since $\mu_\nu^\circ$ has exponential tails, the vertex $v_n^\star$ corresponds necessarily to a black vertex of $\mathrm{Tree}_2(\partial \mathfrak C_\infty^\nu(n))$, and therefore to a macroscopic loop of $\partial \mathfrak C_\infty^\nu(n)$. In addition, the other loops of $\partial \mathfrak C_\infty^\nu(n)$ are microscopic and their size divided by $n$ vanish in probability as $n\to \infty$. This implies that we have the following convergence in distribution for the Gromov--Hausdorff topology:
\[
\frac{1}{n} \partial \mathfrak C_\infty^\nu(n) \underset{n\to \infty}{\rightarrow} \left(1 + \frac{\aleph_1^{Z^+}(\nu)}{1+ Z^+(\nu,t_\nu ,t_\nu x_\nu)} \right)  \mathscr{C}_1.
\]
The constant in the right hand side of the previous equation is explicit and calculated in the Maple companion file \cite{Maple}.

\paragraph{Case $\nu < \nu_c$.} In this case, the reproduction law $\mu_\nu$ is critical and we have from \eqref{eq:looptreeder} and Lemma~\ref{lem:expZplusxc}:
\[
\mu_\nu(k) \underset{k\to \infty}{\sim} \frac{3}{4\sqrt{\pi}}
\frac{-\aleph^{Z^+}(\nu)}{1+ Z^+(\nu,t_\nu ,t_\nu x_\nu)}
\, k^{-5/2}.
\]
From there classical results (see for instance \cite{DuqLtree} and \cite{KorLtree}) imply that $n^{-1/3} \mathrm{JS} ( \mathrm{Tree}_2(\partial \mathfrak C_\infty^\nu(n)) )$ converges in distribution toward the $3/2$-stable Lévy tree for the Gromov-Hausdorff topology. The canonical looptree $\mathrm{Loop}(\mathrm{JS} ( \mathrm{Tree}_2(\partial \mathfrak C_\infty^\nu(n)) )$ associated to $\mathrm{JS} ( \mathrm{Tree}_2(\partial \mathfrak C_\infty^\nu(n)) )$ as defined in the introduction and in \cite{CKlooptrees}, illustrated in Figure~\ref{fig:looptree} in the present setting, has a Gromov--Hausdorff distance to $\partial \mathfrak C_\infty^\nu(n))$ bounded by twice its height. The results of \cite[Section 4.2]{CKlooptrees} show that we have the following convergence in distribution for the Gromov--Hausdorff topology:
\[
\frac{1}{n^{2/3}} \partial \mathfrak C_\infty^\nu(n)) 
 \underset{n \to \infty}{\rightarrow}
\left(\frac{-\aleph^{Z^+}(\nu)}{1+ Z^+(\nu,t_\nu ,t_\nu x_\nu)} \right)^{-2/3} \, \mathscr{L}_{3/2},
\]
which finishes the proof.
\end{proof}

\newpage

\begin{appendices}
\section{Generating series of \texorpdfstring{$\mathbf{q}$}{q}-Boltzmann maps of the cylinder}\label{appendix:cylinder}

\subsection{Parametrization in terms of Zhukovsky's variables}

In the book by Eynard~\cite{EynardBook}, the generating series $W_{\mathbf q , 2}$ is denoted $W_2^{(0)}$ where the ``2'' stands for the number of boundaries and the ``(0)'' for the genus of the underlying surface (which is equal to 0 here, since we only consider maps on the sphere). In Theorem~3.2.1, Eynard obtains the following rational parametrization for $W_{\mathbf q , 2}$ in terms of the so-called Zhukovsky's variables:
\begin{theo}[Theorem~3.2.1 of~\cite{EynardBook}]\label{theo:Eynard}
\begin{equation}\label{eq:formuleEynard}
W_{\mathbf q , 2}\Big(z(x_1),z(x_2)\Big)z'(x_1)z'(x_2)=\frac{-1}{(x_1x_2-1)^2}, 
\end{equation}
where $z(x)=z^\diamond(\mathbf q)+\sqrt{z^+(\mathbf q)}\Big(x+\dfrac{1}{x}\Big)$. 
\end{theo}
In~\cite{EynardBook}, the rational parametrization is in fact only given with the additional assumption that the faces have bounded degree. We refer to~\cite[Section~6]{BBGb} for a combinatorial derivation of the rational parametrization, which does not rely on this assumption. 
\bigskip

We can then check from Theorem~\ref{theo:Eynard} that Proposition~\ref{prop:gencylinders} is correct. It suffices to replace $W_{\mathbf{q},\bullet}(z)$ in~\eqref{eq:cylindre} by its following expression in terms of the rational parametrization, see~\cite[Section~3.1.3]{EynardBook}: 
\[
	W_{\mathbf{q},\bullet}\Big(z(x)\Big)=\sqrt{z^+(\mathbf q)}\left(x-\frac{1}{x}\right),
\]
and to check that one gets the same expression as in~\eqref{eq:formuleEynard}. We perform these computations in the Maple companion file \cite{Maple}.

\subsection{Derivation of Proposition~\ref{prop:gencylinders} via slice decomposition}\label{sub:slices}
The slice decomposition, which has been introduced by Bouttier and Guitter in~\cite{BouttierGuitterIrreducible}, consists in cutting a planar map along some well-chosen geodesics. The submaps obtained after the cutting, which have geodesic boundaries by construction, are called \emph{slices}. 

Originally designed to deal with irreducible maps, the slice decomposition can in fact be applied to general maps, see~\cite[Chapter~2]{BouttierHDR}. In this presentation, Bouttier focuses on bipartite maps, but explains in Remark~2.1 how to extend his results to non-bipartite maps. All the bijective constructions remain valid in the non-bipartite case, the only modifications and complications appear in the enumerative part. 
\bigskip

To recall the main bijective results obtained via the slice decompositions, we start with a few (classical) definitions. A \emph{Motzkin path} is a lattice path made of steps $(0,-1)$, $(0,0)$ and $(0,1)$, which starts at $(0,0)$. We consider in particular Motzkin bridges, which are paths ending on the $x$-axis and Motzkin excursions, which are paths ending at height $-1$ and which stay non-negative before their very last step.

For the purpose of map enumeration, we consider the enumeration of \emph{weighted Motzkin paths}, where each step carries the following weight:
\begin{itemize}
	\item each \emph{down-step} $(-1,0)$ carries a weight $1$, 
	\item each \emph{level-step} $(0,0)$ carries a weight $z^\diamond(\mathbf q)$, 
	\item each \emph{up-step} $(1,0)$ carries a weight $z^+(\mathbf q)$,
\end{itemize}
and the weight of a path is the product of the weights of its steps.  For $h\in \mathbb{Z}$, we let $M_h(z)$ to be the generating series of weighted Motzkin walks that end at height $h$, and similarly $E(z)$ denotes the generating series of weighted Motzkin excursions, that is: 
\[
M_h(z)=\sum_{\substack{w\in\{ \text{Motzkin walks}\\ \text{ending at height }h\}}}\frac{\text{weight}(w)}{z^{\text{length of }w}} \qquad \text{ and }\qquad E(z)=\sum_{\substack{w\in\{ \text{Motzkin}\\
\text{excursions}\}}}\frac{\text{weight}(w)}{z^{\text{length of }w}}.
\]
Then, it follows directly from the discussion following Theorem~2.1 in~\cite{BouttierHDR}, that:

\begin{theo}
We have the following equality of weighted generating series: 
\begin{equation}\label{eq:slicesAnnular}
W_{\mathbf{q},2}(z_1,z_2)=\sum_{h\geq 1}h M_h(z_1)M_{-h}(z_2).
\end{equation}
\end{theo}

The expression for $W_{\mathbf{q},2}(z_1,z_2)$ given in Proposition \ref{prop:gencylinders} then follows easily from this result and classical enumeration of Motzkin paths.

\section{Analytic combinatorics toolbox}

We gather in this appendix two results from the book of Flajolet and Sedgewick \cite{FS} that are used several times in the present work.

\subsection{Transfer theorem for algebraic series} \label{sec:algebraic}

In this work, all our generating series are algebraic and fall into the framework of transfer theorems developed by Flajolet and Sedgewick.

\begin{theo}[Corollary VI.1 of \cite{FS}] \label{th:transfer}
Suppose that $f(z)$ is algebraic and has a unique dominant singularity at $\rho>0$ where it satisfies
\[
f(z) \underset{z \to \rho}{\sim} c \, (1-\rho)^{-a}
\] 
for some $c \in \mathbb R_+$ and $a \notin \{0,-1,-2. \ldots\}$. Then the coefficients of $f$ satisfy
\[
[z^n] f (z) \underset{n \to \infty}{\sim} c \, \frac{n^{a-1}}{\Gamma(a)}, 
\]
where $\Gamma$ denotes Euler's Gamma function.
\begin{proof}
The only thing that one has to verify in order to apply Corollary VI.1 of \cite{FS} is the $\Delta$-analyticity of $f$. In our case, $f$ is algebraic, so it has finitely many singularities. {\color{red}Since $\rho$ is the unique dominant singularity, this implies that} $f$ is analytic in a slit open disk $D(0,\rho+\varepsilon) \setminus [\rho, \rho +\varepsilon)$ for some $\varepsilon >0$ and the transfer theorem follows.
\end{proof}
\end{theo}

\subsection{Hadamard product of series} \label{sec:Hadamard}

Let $f(z) = \sum_{n \geq 0} a_n z^n$ and $g(z) = \sum_{n \geq 0} b_n z^n$ be two formal power series. The Hadamard product of $f$ and $g$ is the formal power series defined by
\[
f \odot g (z) = \sum_{n \geq 0} a_n b_n z^n.
\]
When the functions $f$ and $g$ are analytic in a large enough domain, the Hadamard product has a representation as a contour integral:

\begin{theo}[Equation (58) in \cite{FS}]
Let $f$ and $g$ be two functions which are analytic in the same domain $\Delta \subset \mathbb C$. Then
\[
f \odot g (z)  = \oint_\gamma f(w) g\left(\frac{w}{z} \right) \frac{dw}{w},
\]
where the contour $\gamma$ in the $w$-plane is such that $f(w)$ and $g(z/w)$ are analytic. In other words, $\gamma \subset \Delta \cap (z \Delta^{-1})$.
\end{theo}

\end{appendices}

\newpage

\addcontentsline{toc}{section}{Table of notations}
\section*{Table of notations}\label{sec:notations}

\subsection*{Planar maps}

\begin{tabular}{p{3cm}p{12cm}}
$(\mathfrak t,\sigma)$ & a finite rooted triangulation $\mathfrak t$ with spins $\sigma : V(\mathfrak t) \mapsto \{\ps, \ns \}$, $\sigma$ is sometimes omitted to shorten notation\\
$|\mathfrak t|$ & number of edges of the triangulation $\mathfrak t$\\
$m(\mathfrak t ,\sigma)$ & number of monochromatic edges of the triangulation with spins $(\mathfrak t,\sigma)$\\
$\mathcal T$ & the set of all rooted planar triangulations with spins\\
$\mathcal T^\ps$ & the set of all rooted planar triangulations with spins such that the root vertex carries a spin $\ps$ and with a monochromatic root edge\\
$\mathfrak m$ & a finite rooted non atomic planar map\\
$\mathcal M$ & the set of all finite rooted non atomic planar maps\\
$\mathcal M^\bullet$ & the set of all finite rooted and pointed non atomic planar maps\\
\end{tabular}

\subsection*{Probability distributions}

\begin{tabular}{p{3cm}p{12cm}}
$\mathbb P^\nu$ & Probability distribution on finite triangulations with spins, defined in~\eqref{eq:Pnudef}. The probability of $(\mathfrak t,\sigma)$ according to $\mathbb P^\nu$ is proportional to $\nu^{m(\mathfrak t,\sigma)} \, t_\nu^{|\mathfrak t|}$\\
$\mathbb P^\nu_n$ & Probability distribution $\mathbb P^\nu$ conditioned on triangulations with $3n$ edges \\
$\mathbb P^\nu_\infty$ & Law of the Infinite Ising Planar Triangulation with parameter $\nu$, it is the weak limit of $\mathbb P^\nu_n$ as $n \to \infty$ for the local topology
\end{tabular}

\subsection*{Generating series}

\begin{tabular}{p{3cm}p{12cm}}
$(\nu,t,y)$ & counting variables for triangulations with spins and with a boundary\\
& $\nu$: monochromatic edges\\
& $t$: edges\\
& $y$: perimeter\\
$Z_p^+(\nu,t)$ & generating series of triangulations with a simple boundary of perimeter $p$ and $\ps$ spins on the boundary\\
$Z^+(\nu,t,y)$ & generating series of triangulations with simple boundary and $\ps$ spins on the boundary\\
$Q_p^+(\nu,t)$ & generating series of triangulations with a (not necessarily simple) boundary of perimeter $p$ and $\ps$ spins on the boundary\\
$Q^+(\nu,t,y)$ & generating series of triangulations with (not necessarily simple) boundary and $\ps$ spins on the boundary\\
$\mathcal Z^\ps(\nu,t)$ & generating series of triangulations in $\mathcal T^\ps$\\
$U (\nu , t^3)$ & rational parametrization for $t^3$ for fixed $\nu$, given in~\eqref{eq:wU}\\
$V (\nu, U(\nu,t^3) , y)$ &  rational parametrization for $y$ for fixed $(\nu,t)$, given in~\eqref{eq:defyV}\\
$t_\nu$ & unique dominant singularity in $t$ of $Z_p^+(\nu,t)$, $Q_p^+(\nu,t)$, $U (\nu , t^3)$ for fixed $\nu$ \\
$U_\nu := U(\nu,t_\nu)$ &\\
$y_+(\nu,t) >0$ & unique dominant singularity in $y$ of $Z^+(\nu,t,y)$, $Q^+(\nu,t,y)$, $V (\nu , U(\nu,t^3) , y)$ for fixed $(\nu,t)$ \\
$y_\nu := y_+(\nu,t_\nu)$\\
$y_-(\nu,t) < 0$ & unique non-dominant singularity in $y$ of $Z^+(\nu,t,y)$, $Q^+(\nu,t,y)$, $V (\nu , U(\nu,t^3) , y)$ for fixed $(\nu,t)$
\end{tabular}

\subsection*{Critical exponents and coefficients}
\begin{tabular}{p{5cm}p{10cm}}
$\displaystyle
\alphaa(\nu) =
\begin{cases}
5/3 & \text{for $\nu < \nu_c$,}\\
7/3 & \text{for $\nu = \nu_c$,}\\
5/2 & \text{for $\nu > \nu_c$.}
\end{cases}$ & exponent of the dominant singularity in $y$ of series related to $Q^+(\nu,t_\nu,t_\nu y)$\\
$\beth_i^S(\nu)$ & coefficient of the term $(1-y/y_\nu)^{i}$ for a series $S(\nu,y)$\\
$\beth^S(\nu)$ & coefficient of the dominant singular term for a series $S(\nu,y)$\\
$\displaystyle \gamma(\nu) = 
\begin{cases}
5/2 & \text{for $\nu \neq \nu_c$,}\\
7/3 & \text{for $\nu = \nu_c$,}
\end{cases}$ & exponent of the dominant singularity in $t^3$ of series related to $\mathcal Z^\ps(\nu,t)$\\
$\aleph_i^S(\nu)$ & coefficient of the term $(1-t^3/t_\nu^3)^{i}$ for a series $S(\nu,t^3)$\\
$\aleph^S(\nu)$ & coefficient of the dominant singular term for a series $S(\nu,t^3)$
\end{tabular}

\subsection*{Boltzmann planar maps}
\begin{tabular}{p{3cm}p{12cm}}
$\mathbf q (\nu,t)$ & weight sequence $\left(q_k(\nu,t)\right)_{k \geq 0}$, defined in~\eqref{eq:qk}\\
$w_\mathbf q (\cdot )$ & weight of a map or a set of maps given by the sequence $\mathbf q$, defined in~\eqref{eq:defwq}\\
$W_\mathbf q ^{(l)}$ & weight of disks of perimeter $l$, defined in~\eqref{eq:defDiskL}\\
$W_\mathbf q (z)$ & disk generating function, defined in~\eqref{eq:defDiskPartition}\\
$W_{\mathbf q , \bullet}(z)$ & pointed disk generating function, defined in~\eqref{eq:defPointedDiskPartition}\\
$c_+(\mathbf q) , c_-(\mathbf q)$ & parameters of the universal expression for $W_{\mathbf q , \bullet}(z)$\\
& $c_+^\nu : = c_+(\mathbf q(\nu,t_\nu)) = \frac{1}{\sqrt{\nu t_\nu^3} y_\nu}$ and $c_-^\nu : = c_-(\mathbf q(\nu,t_\nu)) = \frac{1}{\sqrt{\nu t_\nu^3} y_-(\nu,t_\nu)}$\\
$f_\mathbf q^\bullet , f_\mathbf q ^\diamond$ & Bouttier -- Di Francesco -- Guitter functions associated to $\mathbf q$, defined in~\eqref{eq:deffbullet} and~\eqref{eq:deffdiamond}\\
$z^+(\mathbf q), z^\diamond(\mathbf q)$ & solutions of the system of equations involving the BDG functions\\
& $c_+(\mathbf q) = z^\diamond (\mathbf q)+ 2 \sqrt{z^+(\mathbf q)}$ and $c_-(\mathbf q) = z^\diamond (\mathbf q)- 2 \sqrt{z^+(\mathbf q)}$\\
& $z^+_\nu : = z^+(\mathbf q(\nu,t_\nu))$ and $z^\diamond_\nu : = z^\diamond(\mathbf q(\nu,t_\nu))$\\
$\mathbf q_g (\nu,t)$ & weight sequence $\left(g^{(k-2)/2}q_k(\nu,t)\right)_{k \geq 0}$\\
$c_+(\mathbf q,g) , c_-(\mathbf q,g)$ & parameters of the universal expression for $W_{\mathbf q_g , \bullet}(z)$\\
& $c_+^\nu(g) : = c_+(\mathbf q_g(\nu,t_\nu))$ and $c_-^\nu(g) : = c_-(\mathbf q_g(\nu,t_\nu))$\\
$z^+(\mathbf q,g), z^\diamond(\mathbf q,g)$ &  solutions of the perturbed system of equations involving the BDG functions\\
& $z^+_\nu(g) : = z^+(\mathbf q(\nu,t_\nu),g)$ and $z^\diamond_\nu (g) : = z^\diamond(\mathbf q(\nu,t_\nu),g)$
\end{tabular}

\newpage

\addcontentsline{toc}{section}{References}

\bibliographystyle{plain}
\bibliography{IsingClusters}

\begin{thebibliography}{10}

\bibitem{Maple}
Maple companion file.
\newblock Available on the authors' webpage:
  \url{http://www.lix.polytechnique.fr/~albenque/} and
  \url{http://www.normalesup.org/~menard/}.

\bibitem{IsingAMS}
Marie Albenque, Laurent M\'{e}nard, and Gilles Schaeffer.
\newblock Local convergence of large random triangulations coupled with an
  {I}sing model.
\newblock {\em Trans. Amer. Math. Soc.}, 374(1):175--217, 2021.

\bibitem{AJM90}
Jan Ambj{\o}rn, J~Jurkiewicz, and Yu~M Makeenko.
\newblock Multiloop correlators for two-dimensional quantum gravity.
\newblock {\em Physics Letters B}, 251(4):517--524, 1990.

\bibitem{AngelPerco}
O.~Angel.
\newblock Growth and percolation on the uniform infinite planar triangulation.
\newblock {\em Geom. Funct. Anal.}, 13(5):935--974, 2003.

\bibitem{AngelSchramm}
O.~Angel and O.~Schramm.
\newblock Uniform infinite planar triangulations.
\newblock {\em Comm. Math. Phys.}, 241(2-3):191--213, 2003.

\bibitem{IsingBernoulli}
Andr\'{a}s B\'{a}lint, Federico Camia, and Ronald Meester.
\newblock The high temperature {I}sing model on the triangular lattice is a
  critical {B}ernoulli percolation model.
\newblock {\em J. Stat. Phys.}, 139(1):122--138, 2010.

\bibitem{SLEdim}
Vincent Beffara.
\newblock The dimension of the {SLE} curves.
\newblock {\em Ann. Probab.}, 36(4):1421--1452, 2008.

\bibitem{BS}
I.~Benjamini and O.~Schramm.
\newblock Recurrence of distributional limits of finite planar graphs.
\newblock {\em Electron. J. Probab.}, 6(23), 2001.

\bibitem{BernardiBousquet}
O.~Bernardi and M.~Bousquet-M{\'e}lou.
\newblock Counting colored planar maps: algebraicity results.
\newblock {\em J.~Combin.~Theory~Ser.~B}, 101(5):315--377, 2011.

\bibitem{BeCuMie}
O.~Bernardi, N.~Curien, and G.~Miermont.
\newblock A {B}oltzmann approach to percolation on random triangulations.
\newblock {\em Canad. J. Math.}, 71(1):1--43, 2019.

\bibitem{BeHoSun}
O.~Bernardi, N.~Holden, and X.~Sun.
\newblock Percolation on triangulations: a bijective path to {L}iouville
  quantum gravity.
\newblock {\em arXiv preprint arXiv:1807.01684, to appear in Memoirs of the
  AMS}, 2018.

\bibitem{BBD}
G.~Borot, J.~Bouttier, and B.~Duplantier.
\newblock Nesting statistics in the {O}(n) loop model on random planar maps.
\newblock {\em arXiv preprint arXiv:1605.02239}, 2016.

\bibitem{BBGa}
G.~Borot, J.~Bouttier, and E.~Guitter.
\newblock A recursive approach to the {O}(n) model on random maps via nested
  loops.
\newblock {\em J.~Phys.~A}, 45(4):045002, 2011.

\bibitem{BBGc}
G.~Borot, J.~Bouttier, and E.~Guitter.
\newblock Loop models on random maps via nested loops: the case of domain
  symmetry breaking and application to the {P}otts model.
\newblock {\em J.~Phys.~A}, 45(49):494017, 2012.

\bibitem{BBGb}
G.~Borot, J.~Bouttier, and E.~Guitter.
\newblock More on the {O}(n) model on random maps via nested loops: loops with
  bending energy.
\newblock {\em J.~Phys.~A}, 45(27):275206, 2012.

\bibitem{borot2018nesting}
Gaëtan Borot and Elba Garcia-Failde.
\newblock Nesting statistics in the {O}(n) loop model on random maps of
  arbitrary topologies.
\newblock {\em arXiv preprint arXiv:1609.02074}, 2018.

\bibitem{BoulatovKazakov}
D.~Boulatov and V.~Kazakov.
\newblock The {I}sing model on a random planar lattice: the structure of the
  phase transition and the exact critical exponents.
\newblock {\em Phys. Lett. B}, 186(3-4):379--384, 1987.

\bibitem{BMJ}
M.~Bousquet-M{\'e}lou and A.~Jehanne.
\newblock Polynomial equations with one catalytic variable, algebraic series
  and map enumeration.
\newblock {\em Journal of Combinatorial Theory, Series B}, 96(5):623--672,
  2006.

\bibitem{BMS}
M.~Bousquet-M{\'e}lou and G.~Schaeffer.
\newblock The degree distribution in bipartite planar maps: applications to the
  {I}sing model.
\newblock {\em arXiv preprint math/0211070}, 2002.

\bibitem{BDFG}
J.~Bouttier, P.~Di~Francesco, and E.~Guitter.
\newblock Planar maps as labeled mobiles.
\newblock {\em Electron. J. Combin.}, 11(1):Research Paper 69, 27, 2004.

\bibitem{BouttierHDR}
J{\'e}r{\'e}mie Bouttier.
\newblock Planar maps and random partitions.
\newblock {\em arXiv preprint arXiv:1912.06855}, 2019.

\bibitem{BouttierGuitterIrreducible}
J{\'e}r{\'e}mie Bouttier and Emmanuel Guitter.
\newblock On irreducible maps and slices.
\newblock {\em Combinatorics, Probability and Computing}, 23(6):914--972, 2014.

\bibitem{BuddPeel}
Timothy Budd.
\newblock The peeling process of infinite {B}oltzmann planar maps.
\newblock {\em Electron. J. Combin.}, 23(1):Paper 1.28, 37, 2016.

\bibitem{ChassaingDurhuus}
{Ph.} Chassaing and B.~Durhuus.
\newblock Local limit of labeled trees and expected volume growth in a random
  quadrangulation.
\newblock {\em The Annals of Probability}, pages 879--917, 2006.

\bibitem{C}
L.~Chen.
\newblock Basic properties of the infinite critical-{FK} random map.
\newblock {\em Ann. Inst. Henri Poincar\'{e} D}, 4(3):245--271, 2017.

\bibitem{ChenTurunen}
Linxiao Chen and Joonas Turunen.
\newblock Critical {I}sing model on random triangulations of the disk:
  enumeration and local limits.
\newblock {\em Comm. Math. Phys.}, 374(3):1577--1643, 2020.

\bibitem{ChenTurunen2}
Linxiao Chen and Joonas Turunen.
\newblock Ising model on random triangulations of the disk: phase transition.
\newblock {\em arXiv preprint arXiv:2003.09343}, 2020.

\bibitem{CuKo}
N.~Curien and I.~Kortchemski.
\newblock Percolation on random triangulations and stable looptrees.
\newblock {\em Probab. Theory Related Fields}, 163(1-2):303--337, 2015.

\bibitem{CurienSF}
Nicolas Curien.
\newblock Peeling random planar maps.
\newblock {\em Saint-Flour lecture notes}, 2019.

\bibitem{CKlooptrees}
Nicolas Curien and Igor Kortchemski.
\newblock Random stable looptrees.
\newblock {\em Electron. J. Probab.}, 19:no. 108, 35, 2014.

\bibitem{DCS}
Hugo Duminil-Copin and Stanislav Smirnov.
\newblock Conformal invariance of lattice models.
\newblock In {\em Probability and statistical physics in two and more
  dimensions}, volume~15 of {\em Clay Math. Proc.}, pages 213--276. Amer. Math.
  Soc., Providence, RI, 2012.

\bibitem{DS}
B.~Duplantier and S.~Sheffield.
\newblock {L}iouville quantum gravity and {KPZ}.
\newblock {\em Invent. Math.}, 185(2):333--393, 2011.

\bibitem{DuqLtree}
Thomas Duquesne.
\newblock A limit theorem for the contour process of conditioned
  {G}alton-{W}atson trees.
\newblock {\em Ann. Probab.}, 31(2):996--1027, 2003.

\bibitem{EynardBook}
Bertrand Eynard.
\newblock Counting surfaces.
\newblock {\em Progress in Mathematical Physics}, 70, 2016.

\bibitem{RationalAlgebraicCurves}
Rafael~Sendra Ferrer, Sonia~P{\'e}rez D{\'\i}az, and F~Winkler.
\newblock {\em Rational algebraic curves: a computer algebra approach},
  volume~22 of {\em Algorithms and computation in Mathematics}.
\newblock Springer, 2008.

\bibitem{FS}
{Ph.} Flajolet and R.~Sedgewick.
\newblock {\em Analytic combinatorics}.
\newblock Cambridge University Press, Cambridge, 2009.

\bibitem{GarbanKPZ}
C.~Garban.
\newblock Quantum gravity and the {KPZ} formula.
\newblock In {\em S\'eminaire {B}ourbaki}, volume~64, 2011-2012.

\bibitem{GMSS}
M.~Gorny, E.~Maurel-Segala, and A.~Singh.
\newblock The geometry of a critical percolation cluster on the {UIPT}.
\newblock {\em Ann.~Inst.Henri~Poincar{\'e}}, 54(4):2203--2238, 2018.

\bibitem{GHSSurvey}
E.~Gwynne, N.~Holden, and X.~Sun.
\newblock Mating of trees for random planar maps and {L}iouville quantum
  gravity: a survey.
\newblock {\em arXiv preprint arXiv:1910.04713}, 2019.

\bibitem{HoldenSun}
N.~Holden and X.~Sun.
\newblock Convergence of uniform triangulations under the cardy embedding.
\newblock {\em arXiv preprint arxiv1905.13207, to appear in Acta Mathematica},
  2019.

\bibitem{Ising}
Ernst Ising.
\newblock Beitrag zur {T}heorie des {F}erromagnetismus.
\newblock {\em Zeitschrift für {P}hysik}, 31:253--258, 1925.

\bibitem{JStree}
Svante Janson and Sigurdur~\"{O}rn Stef\'{a}nsson.
\newblock Scaling limits of random planar maps with a unique large face.
\newblock {\em Ann. Probab.}, 43(3):1045--1081, 2015.

\bibitem{JScondensation}
Thordur Jonsson and Sigurdur~\"{O}rn Stef\'{a}nsson.
\newblock Condensation in nongeneric trees.
\newblock {\em J. Stat. Phys.}, 142(2):277--313, 2011.

\bibitem{Bipolar}
Richard Kenyon, Jason Miller, Scott Sheffield, and David~B Wilson.
\newblock Bipolar orientations on planar maps and {SLE}$_{12}$.
\newblock {\em The Annals of Probability}, 47(3):1240--1269, 2019.

\bibitem{Kesten}
Harry Kesten.
\newblock Scaling relations for {$2$}{D}-percolation.
\newblock {\em Comm. Math. Phys.}, 109(1):109--156, 1987.

\bibitem{KPZ}
V.~G. Knizhnik, A.~M. Polyakov, and A.~B. Zamolodchikov.
\newblock Fractal structure of {$2$}{D}-quantum gravity.
\newblock {\em Modern Phys. Lett. A}, 3(8):819--826, 1988.

\bibitem{KorLtree}
Igor Kortchemski.
\newblock A simple proof of {D}uquesne's theorem on contour processes of
  conditioned {G}alton-{W}atson trees.
\newblock In {\em S\'{e}minaire de {P}robabilit\'{e}s {XLV}}, volume 2078 of
  {\em Lecture Notes in Math.}, pages 537--558. Springer, Cham, 2013.

\bibitem{KorGW}
Igor Kortchemski.
\newblock Limit theorems for conditioned non-generic {G}alton-{W}atson trees.
\newblock {\em Ann. Inst. Henri Poincar\'{e} Probab. Stat.}, 51(2):489--511,
  2015.

\bibitem{CRlooptrees}
Igor Kortchemski and Lo\"{\i}c Richier.
\newblock The boundary of random planar maps via looptrees.
\newblock {\em Ann. Fac. Sci. Toulouse Math. (6)}, 29(2):391--430, 2020.

\bibitem{KrikunQuad}
M.~Krikun.
\newblock Local structure of random quadrangulations.
\newblock {\em arXiv preprint math/0512304}, 2005.

\bibitem{Kr}
M.~Krikun.
\newblock {Uniform infinite planar triangulation and related time-reversed
  critical branching process}.
\newblock {\em J. Math. Sci. (N.Y.)}, 131(2):5520--5537, 2005.

\bibitem{LGbm}
J.-F. Le~Gall.
\newblock Uniqueness and universality of the {B}rownian map.
\newblock {\em Ann.~Probab.}, 41(4):2880--2960, 2013.

\bibitem{LGsurvey}
J.-F. Le~Gall.
\newblock Random geometry on the sphere.
\newblock In {\em Proceedings of the {I}nternational {C}ongress of
  {M}athematicians---{S}eoul 2014. {V}ol. 1}, pages 421--442. Kyung Moon Sa,
  Seoul, 2014.

\bibitem{LGM}
J.-F. Le~Gall and G.~Miermont.
\newblock Scaling limits of random planar maps with large faces.
\newblock {\em Ann. Probab.}, 39(1):1--69, 2011.

\bibitem{LeGallMiermont}
Jean-Fran{\c{c}}ois Le~Gall and Gr{\'e}gory Miermont.
\newblock Scaling limits of random planar maps with large faces.
\newblock {\em Annals of Probability}, 39(1):1--69, 2011.

\bibitem{Schnyder}
Yiting Li, Xin Sun, and Samuel~S Watson.
\newblock Schnyder woods, sle (16), and liouville quantum gravity.
\newblock {\em arXiv preprint arXiv:1705.03573}, 2017.

\bibitem{MarckertMiermontInvariance}
J.-F. Marckert and G.~Miermont.
\newblock Invariance principles for random bipartite planar maps.
\newblock {\em Ann. Probab}, 35(5):1642--1705, 2007.

\bibitem{Mar18b}
Cyril Marzouk.
\newblock {On scaling limits of planar maps with stable face-degrees}.
\newblock {\em ALEA Lat. Am. J. Probab. Math. Stat.}, 15:1089--1122, 2018.

\bibitem{Mperco}
L.~M\'enard.
\newblock Percolation probability and critical exponents for site percolation
  on the {UIPT}.
\newblock {\em Preprint}, 2022.

\bibitem{MiermontInvariance}
G.~Miermont.
\newblock An invariance principle for random planar maps.
\newblock {\em Discrete~Math.~Theor.~Comput.~Sci}, 2006.

\bibitem{Miebm}
G.~Miermont.
\newblock The {B}rownian map is the scaling limit of uniform random plane
  quadrangulations.
\newblock {\em Acta Math.}, 210(2):319--401, 2013.

\bibitem{MieSF}
G.~Miermont.
\newblock Aspects of random maps.
\newblock {\em Saint-{F}lour lecture notes}, 2014.

\bibitem{MillerSurvey}
Jason Miller.
\newblock Liouville quantum gravity as a metric space and a scaling limit.
\newblock In {\em Proceedings of the {I}nternational {C}ongress of
  {M}athematicians---{R}io de {J}aneiro 2018. {V}ol. {IV}. {I}nvited lectures},
  pages 2945--2971. World Sci. Publ., Hackensack, NJ, 2018.

\bibitem{MSa}
Jason Miller and Scott Sheffield.
\newblock Liouville quantum gravity and the {B}rownian map {I}: the {${\rm
  QLE}(8/3,0)$} metric.
\newblock {\em Invent. Math.}, 219(1):75--152, 2020.

\bibitem{MSb}
Jason Miller and Scott Sheffield.
\newblock Liouville quantum gravity and the {B}rownian map {II}: {G}eodesics
  and continuity of the embedding.
\newblock {\em Ann. Probab.}, 49(6):2732--2829, 2021.

\bibitem{MSc}
Jason Miller and Scott Sheffield.
\newblock Liouville quantum gravity and the {B}rownian map {III}: the conformal
  structure is determined.
\newblock {\em Probab. Theory Related Fields}, 179(3-4):1183--1211, 2021.

\bibitem{CLEdim}
Jason Miller, Nike Sun, and David~B. Wilson.
\newblock The {H}ausdorff dimension of the {CLE} gasket.
\newblock {\em Ann. Probab.}, 42(4):1644--1665, 2014.

\bibitem{Mullin}
R.~C. Mullin.
\newblock On the enumeration of tree-rooted maps.
\newblock {\em Canadian J. Math.}, 19:174--183, 1967.

\bibitem{Onsager}
Lars Onsager.
\newblock Crystal statistics. {I}. {A} two-dimensional model with an
  order-disorder transition.
\newblock {\em Phys. Rev. (2)}, 65:117--149, 1944.

\bibitem{Sch98}
G.~Schaeffer.
\newblock {\em Conjugaison d'arbres et cartes combinatoires al{\'e}atoires}.
\newblock PhD thesis, Universit{\'e} Bordeaux~I, 1998.

\bibitem{autopromo}
G.~Schaeffer.
\newblock {Planar Maps}.
\newblock In {\em Handbook of Enumerative Combinatorics}, volume~87, chapter~5.
  CRC Press, 2015.

\bibitem{She}
S.~Sheffield.
\newblock Quantum gravity and inventory accumulation.
\newblock {\em Ann. Probab.}, 44(6):3804--3848, 2016.

\bibitem{Turunen}
Joonas Turunen.
\newblock Interfaces in the vertex-decorated {I}sing model on random
  triangulations of the disk.
\newblock {\em arXiv preprint arXiv:2003.11012}, 2020.

\bibitem{Tuttetri}
William~T. Tutte.
\newblock A census of planar triangulations.
\newblock {\em Canadian J. Math.}, 14:21--38, 1962.

\bibitem{Tuttecensus}
William~T. Tutte.
\newblock A census of planar maps.
\newblock {\em Canadian J. Math.}, 15:249--271, 1963.

\bibitem{Yang}
C.~N. Yang.
\newblock The spontaneous magnetization of a two-dimensional {I}sing model.
\newblock {\em Phys. Rev. (2)}, 85:808--816, 1952.

\end{thebibliography}

\end{document}